\documentclass[12pt]{amsart}
\usepackage{a4wide}
\usepackage{psfrag}
\usepackage{graphicx}
\usepackage{verbatim}
\usepackage{amssymb}

\numberwithin{equation}{section}
\newtheorem{sublem}{Sublemma}
\newtheorem{lem}{Lemma}[section]
\newtheorem{prop}[lem]{Proposition}
\newtheorem{thm}{Theorem}

\theoremstyle{definition}

\theoremstyle{remark}
\newtheorem{rem}{Remark}

\newcommand{\T}{\mathbb{T}}
\newcommand{\la}{\lambda}

\newcommand{\vf}{\varphi}
\newcommand{\co}{\text{const.}}
\newcommand{\tht}{\theta}
\newcommand{\om}{\omega}
\newcommand{\vp}{\varphi}
\newcommand{\ve}{\varepsilon}
\newcommand{\N}{\mathcal{N}}
\newcommand{\Z}{\mathbb{Z}}
\newcommand{\PR}{\widehat{\mathbb{R}}}
\newcommand{\R}{\mathbb{R}}

\title[The dynamics of a class of quasi-periodic Schr\"odinger cocycles]
{The dynamics of a class of quasi-periodic Schr\"odinger cocycles}
\author{Kristian Bjerkl\"ov}
\email{bjerklov@kth.se}
\address{Department of Mathematics, KTH Royal Institute of Technology, 100 44 Stockholm, Sweden}
\thanks{Research supported by the Swedish Research Council (VR)}

\begin{document}
\begin{abstract} Let $f:\mathbb{T}\to\mathbb{R}$ be a Morse function of class $C^2$ with 
exactly two critical points, let $\omega\in\T$ be Diophantine, and let $\lambda>0$ be sufficiently large (depending on $f$ and $\omega$). 
For any value of the parameter $E\in \mathbb{R}$ we make a careful analysis of the dynamics of the 
skew-product map 
$$
\Phi_E(\theta,r)=\left(\theta+\omega,\lambda f(\theta)-E-1/r\right),
$$
acting on the "torus" $\mathbb{T}\times\widehat{\mathbb{R}}$. Here $\PR$ denotes the projective space $\mathbb{R}\cup\{\infty\}$. 

The map $\Phi_E$ is intimately related to the quasi-periodic
Schr\"odinger cocycle $(\omega,A_E): \T\times \mathbb{R}^2 \to \T\times \mathbb{R}^2$,  
$(\theta,x)\mapsto (\theta+\omega, A_E(\theta)\cdot x)$,
where $A_E:\T\to \text{SL}(2,\mathbb{R})$ is given by
$$
A_{E}(\theta)=\left(\begin{matrix}0 & 1 \\ -1 & \lambda f(\theta)-E \end{matrix}
\right), \quad E\in \mathbb{R}.
$$
More precisely, $(\omega,A_E)$ naturally acts on the space $\T\times \PR$, and  $\Phi_E$ is the map thus obtained.   

The cocycle $(\omega,A_E)$ arises when investigating the eigenvalue equation $H_\theta u=Eu$, where $H_\theta$ is the quasi-periodic Schr\"odinger operator
$$
(H_\theta u)_n=-(u_{n+1}+u_{n-1}) + \la f(\theta+(n-1)\omega)u_n,
$$
acting on the space $l^2(\Z)$. It is well-known that the spectrum of $H_\theta$, $\sigma(H)$, is independent of the phase $\theta\in\T$. 
Under our assumptions on $f, \omega$ and $\lambda$, Sinai \cite{Sin} has shown that $\sigma(H)$ is a Cantor set, and the operator $H_\theta$ has
a pure-point spectrum, with exponentially decaying eigenfunctions, for a.e. $\theta\in\T$ 

The analysis of $\Phi_E$ allows us to derive three main results: 

(1) The (maximal) Lyapunov exponent of the Schr\"odinger cocycle $(\omega,A_E)$ 
is $\gtrsim \log \lambda$, uniformly in $E\in \mathbb{R}$. This implies that the map $\Phi_E$ has exactly two ergodic probability measures for all $E\in \mathbb{R}$; 

(2) If $E$ is on the edge of an open gap in the spectrum $\sigma(H)$, then there exist a phase $\theta\in\T$ and a vector $u\in l^2(\Z)$, exponentially decaying at $\pm\infty$,
such that $H_\theta u=Eu$; 

(3) The map $\Phi_E$ is minimal iff $E\in \sigma(H)\setminus\{\text{edges of open gaps}\}$. In particular, $\Phi_E$ is minimal
for all $E\in \R$ for which the fibered rotation number $\alpha(E)$
associated to $(\omega,A_E)$ is irrational with respect to $\omega$.

The assumptions on $f$ are fulfilled for $f(\theta)=\cos 2\pi \theta$; thus the results apply to the so-called almost Mathieu case.
\end{abstract}
\maketitle

\begin{section}{Introduction}
In this paper we shall consider the one-parameter family of quasi-periodic Schr\"odinger cocycle maps, parameterized by the
real number $E$, and given by 
\begin{equation}\label{CC}
\begin{aligned}
(\omega,A_E)&: \mathbb{T}\times \mathbb{R}^2 \to \mathbb{T}\times \mathbb{R}^2\\
&(\theta,x)\mapsto (\theta+\omega,A_E(\theta)x)\quad (\T=\mathbb{R}/\mathbb{Z})
\end{aligned}
\end{equation}
where $\omega\in \T$ is irrational, and $A_E(\theta)$ is defined as
$$
A_{E}(\theta)=\left(\begin{matrix}0 & 1 \\ -1 & \lambda f(\theta)-E \end{matrix}
\right)\in SL(2,\mathbb{R})\quad (E \in\mathbb{R}).
$$
We shall investigate the dynamics of this family of cocycles $(\omega,A_E)$ in the case where 
the \emph{coupling constant} $\lambda\gg 1$ and   
$f:\mathbb{T}\to\mathbb{R}$ belongs to a class of smooth functions. 

We introduce the notation  
$$
A_{E}^n(\theta)=
\begin{cases}
A_{E}(\theta+(n-1)\omega)\cdots A_{E}(\theta+\omega)A_{E}(\theta), & n>0 \\
I, & n=0 \\
A_{E}(\theta+n\omega)^{-1}\cdots A_{E}(\theta-\omega)^{-1}, &n<0
\end{cases}
$$
Then the $n$:th iterate of the cocycle $(\omega,A_E)$ is given by $(\theta+n\omega,A_E^n(\theta)x)$.

We say that the cocycle $(\omega,A_E)$ is \emph{uniformly hyperbolic} if there exists a continuous splitting 
$W^u(\theta)\oplus W^s(\theta)=\mathbb{R}^2$ and constants $C>0$, $\delta>0$ such that for all $\theta\in \T$ and $n\geq1$ we have
$$
\begin{aligned}
|A^n_E(\theta)x|\leq Ce^{-\delta n}|x|, \quad x\in W^s(\theta) \\
|A^{-n}_E(\theta)x|\leq Ce^{-\delta n}|x|, \quad x\in W^u(\theta).
\end{aligned}
$$

For $\theta\in \T$ we let $H_\theta$ be the discrete Schr\"odinger operator, acting on $l^2(\mathbb{Z})$, and defined by
\begin{equation}\label{operator}
(H_{\theta}u)_n =-(u_{n+1}+u_{n-1})+\la f(\theta+(n-1)\omega) u_n.
\end{equation}
This operator has been intensively studied in the literature (see, e.g., \cite{BG, E, FSW, GS2, GS3, Kl, Sin} and references therein).
The link between the operator $H_\theta$ and the cocycle (\ref{CC}) is that $A_E^n(\theta)$ is the fundamental solution of the eigenvalue equation
$$
-(u_{n+1}+u_{n-1})+\la f(\theta+(n-1)\omega) u_n =Eu_n,
$$
i.e., given $x=(u_0,u_1)^T\in \mathbb{R}^2$, we have $(u_n,u_{n+1})^T=A_E^n(\theta)x$ for all $n\in \mathbb{Z}$.

It is well-known that the spectrum of the operator $H_\theta$, which we shall denote by $\sigma(H_\theta)$, does not depend on the phase $\theta$
(under the condition that $f$ is continuous). We shall therefore often write just $\sigma(H)$. It is also a general fact that 
the cocycle $(\ref{CC})$ is uniformly hyperbolic iff $E\notin \sigma(H)$ (see \cite{Jo}). Thus, the spectrum $\sigma(H)$ coincides with the 
dynamical spectrum, or the Sacker-Sell spectrum, that is, the set of $E$ for which the cocycle is not uniformly hyperbolic.

A central quantity in the study of (\ref{CC}) is the (maximal) \emph{Lyapunov exponent} $\gamma(E)$, which is defined by 
$$
\gamma(E)=\lim_{n\to\infty}\frac{1}{n}\int_{\mathbb{T}}\log \|A_{E}^n(\theta)\|d\theta\geq 0
$$
(the limit exists by subadditivity, and is $\geq 0$ since the matrices have determinant $1$). Here
$\|\cdot\|$ denotes the operator norm.
By Kingman's subadditive ergodic theorem, using the fact that $A_{E}\in \text{SL}(2,\mathbb{R})$, we know that 
\begin{equation}\label{kingman}
\lim_{n\to\infty}\frac{1}{n}\log |A_{E}^n(\theta)x|=\pm \gamma(E)\quad \text{for a.e. }\theta\in\T \text{ and all } 
0\neq x\in \mathbb{R}^2.
\end{equation}
The Lyapunov exponent is intimately connected to spectral properties of the (family of) operators $H_\theta$. For example, if $\gamma(E)>0$ for all
$E\in\R$ then, by results of Pastur and Ishii, it follows that $H_\theta$ has no ac-spectrum for a.e., and hence all \cite{LS}, $\theta\in \T$.

It is in  general a very hard task to obtain (non-trivial) lower bounds on $\gamma(E)$. However, this is a central issue.
Not only because of the connections to the Schr\"odinger operator, but since it seems to contain fundamental problems in the general theory of smooth dynamical systems.
Of course the quasi-periodic Schr\"odinger equation is a very special system, but techniques developed for its investigation would likely 
be applicable elsewhere. In his seminal paper \cite{Fu} Furstenberg studies the growth of products of random matrices. This work has motivated
much further research. One of the main purposes of the present paper is to develop techniques suitable, in particular, for the investigation of the cocycle $(\omega,A_E)$.
That is, for fixed $\omega$,$f$ and $\lambda$, we investigate the one-parameter family $(\omega,A_E)$ in the space of (smooth) $SL(2,\mathbb{R})$-cocycles, parameterized by $E\in\R$.  
Avila \cite{A} has recently shown very general results on density (and prevalens) of positive Lyapunov exponents for $SL(2,\R)$-cocycles, including
the Schr\"odinger case. 
Recall also that the measure of the set of 
$E$ for which $\gamma(E)=0$ can be at most $4$ \cite{DS}

\bigskip
Before stating our main results, and giving a brief overview of related works of this well-studied object, 
we shall shortly discuss the setting for our approach. This approach is in the spirit of the theory of non-linear dynamical 
systems and uses philosophy from the study of one-dimensional dynamics \cite{BC, CE, Ja}. Similar techniques has
been used in \cite{Bj1, Bj2, Y}. However, in the latter papers one has to use parameter exclusion, either in $E$ or $\omega$, to avoid having certain resonances.  
A main difference in the present paper is that we will be able to deal with these resonances, which naturally appear, and thus avoid using such parameter exclusion. 
To achieve this we will have to take a different route than the ones used in \cite{Bj1, Bj2}.    

The way we are going to control the dynamics of the cocycle $(\omega,A_E)$ is to 
make a detailed analysis of the induced projective action of the fiber maps $A_E(\theta)$. That is, we study how $(\omega,A_E)$ acts on the  
space $\T\times \widehat{\mathbb{R}}$, where $\widehat{\mathbb{R}}=\mathbb{R}\cup\{\infty\}$ is the projective line. 
Since 
$$
A_{E}(\theta)\binom{1}{r}=\left(\begin{matrix}0 & 1 \\ -1 & \lambda f(\theta)-E \end{matrix}
\right)\binom{1}{r}=r\binom{1}{\lambda f(\theta)-E-1/r},
$$
we thus obtain the family of maps $\Phi_E:\T\times \widehat{\mathbb{R}}\to \T\times \widehat{\mathbb{R}}$ given by
\begin{equation}\label{mapPhi}
\Phi_E(\theta,r)=\left(\theta+\omega,\lambda f(\theta)-E-\frac{1}{r}\right)
\end{equation}
and parameterized by $E\in \R$. This shall be our main object of study. We introduce the following natural notation. Given (initial conditions) $\theta_0, r_0$, we let
\begin{equation}\label{notaIter}
(\theta_n,r_n)=\Phi_E^n(\theta_0,r_0), \quad n\in \mathbb{Z}.
\end{equation}
With this notation we see that we have 
\begin{equation}\label{matrix_r}
A^n_{E}(\theta_0)\binom{1}{r_0}=r_0\cdots r_{n-1}\binom{1}{r_n}=\binom{r_0\cdots r_{n-1}}{r_0\cdots r_{n}}.
\end{equation}
Therefore, recalling (\ref{kingman}), in order to control exponential growth to the matrix products, 
we have to control the iterates of the map (\ref{mapPhi}), and keep track of products $r_0\cdots r_n$, at least
for a set of $\theta_0\in \T$ of positive measure.

Since $\PR \cong S^1$ the maps $\Phi_E$ are in fact a class of quasi-periodically forced circle maps.  
The techniques we use are actually non-linear, in the sense that we do not
really use the fact that the map $\Phi$ comes from a linear system. Therefore one can extend the analysis we develop in this 
paper to other families of quasi-periodically perturbed circle maps. 
Such a route has been taken, for example, in \cite{J1,J2}, where the techniques in \cite{Bj1} and \cite{Y} are adapted to a wider class of circle maps. 
It should also be possible to analyse wider classes of quasi-periodically forced systems.
However, since we here mostly are interested in the phenomenological problems, we shall focus only on the particular maps $\Phi_E$. 

\bigskip
We are now ready to state the results about the cocycle (\ref{CC}) that our analysis will yield.

\begin{thm}\label{Maintheorem}
Assume that the function $f:\mathbb{T}\to\mathbb{R}$ is of class $C^2$, and that it has exactly two critical points, 
both non-degenerate. 
Assume further that the base frequency $\omega\in \T$ satisfies the Diophantine condition 
\begin{equation}\label{DC}
\inf_{p\in\mathbb{Z}}|q\omega-p|>\frac{\kappa}{|q|^\tau}\quad\text{for all } q\in\mathbb{Z}\setminus \{0\}
\end{equation}
for some constants $\kappa>0$ and $\tau \geq 1$. Then there exists a $\lambda_0=\lambda_0(f,\kappa,\tau)>0$ such 
that for any $\lambda>\lambda_0$ we have
$$
\gamma(E)\geq \frac{2\log\lambda}{3} \quad\text{for all } E\in \mathbb{R}.
$$
Moreover, if $E$ is on the edge of an open gap in the spectrum $\sigma(H)$, then there exists a phase $\theta\in \T$
and $u\in l^2(\mathbb{Z})$, exponentially decaying at $\pm\infty$, such that $H_\theta u=Eu$.

\end{thm}

\begin{rem}
a) Recall that the Diophantine condition (\ref{DC}) is satisfied for (Lebesgue) a.e. $\omega\in \T$.

b) Under the same assumptions on $f$ and $\omega$ as above, Sinai \cite{Sin} has shown that for all sufficiently large $\lambda$,
the discrete Schr\"odinger operator (\ref{operator}) has a pure-point spectrum with exponentially decaying eigenfunctions for a.e. $\theta\in \T$ 
(almost identical results have also been obtained in \cite{FSW}). 
From a theorem by Kotani \cite{K} (see \cite{Sim} for exactly our setting) this implies
that $\gamma(E)>0$ for a.e. $E\in \mathbb{R}$. Moreover, Sinai also shows that the spectrum $\sigma(H)$ is a Cantor set. 

c) The existence of (exponentially decaying) eigenfunctions for energies $E$ at all the gap edges extends results in \cite{So} (see also \cite{Bj2}).  

d) As we mentioned above, the approach we will use to prove Theorem \ref{Maintheorem} is to make a careful 
analysis of the map $\Phi_E$.  
We have used the same approach in \cite{Bj1,Bj2}. However, in the present paper we will have to
handle the problems with resonances that appear, since we want to be able to consider all value of $E$. This is in fact a 
major problem, as we shall see. 

e) By using the same techniques as the ones we develop here it is 
possible to prove the following: if $f:\T\to\R$ is of class $C^2$ and has a unique, non-degenerate, minimum (or maximum), if
$\omega$ is Diophantine, and if $\lambda$ is large, then $\gamma(E)\geq (2/3)\log\la$ for all $E$ in an interval containing the bottom (or the top) of the spectrum $\sigma(H)$.
The techniques should also apply to the investigation of the quasi-periodic continuous Schr\"odinger operator for all energies in an interval containing the bottom of 
the spectrum, thus extending the results in \cite{Bj4}.

\end{rem}

The analysis of the map $\Phi_E$ also gives us some results about its own ergodic properties. Before stating them 
we first recall the well-known fact that there is a fibered rotation number, $\alpha(E)$, associated to the cocycle $(\omega,A_E)$ \cite{DSo, H, JM}. The function 
$E\mapsto \alpha(E)$ is continuous and monotone, maps $\R$ onto $[0,1/2]$, and is constant on open gaps in the resolvent set of $H$  
(in fact, $2\alpha(E)=k(E)$, where $k(E)$ is the integrated density of states of $H$; for finer regularity results, see \cite{GS1,GS2}). 
Moreover, on each such open gap there is an integer $k$ such that  $2\alpha(E)=\{k\omega\}$ (so-called gap labeling; see \cite{JM}). 
We let $$\mathcal{M}=\{\{k\omega\}: k\in \mathbb{Z}\}$$ denote the frequency module.

\begin{thm}\label{thm2}
Let $f,\omega$ and $\lambda$ be as in Theorem \ref{Maintheorem}. For all values of the parameter $E\in\R$ the map
$\Phi_E$ has exactly two ergodic probability measures. Furthermore, 
$$
\Phi_E \text{ is minimal }  \Longleftrightarrow E\in \sigma(H)\setminus\{\text{edges of open gaps}\}. 
$$
In particular, if $2\alpha(E)\notin \mathcal{M}$, then $\Phi_E$ is minimal.
\end{thm}
\begin{rem}\label{rem_proj}
(a) That $\Phi_E$ has exactly two ergodic probability measures follows from Oseledets' theorem, using that $\gamma(E)>0$ (see, for example, \cite[Section 4.17]{H}),
together with the fact that the transformation $\theta\mapsto \theta+\omega$ on $\T$ is uniquely ergodic.

(b) If the cocycle $(\omega,A_E)$ is not uniformly hyperbolic, and $\gamma(E)>0$, then the map $\Phi_E$ has a unique minimal set $M\subset \T\times \PR$ \cite{H, Jo3}. 
We show that if $E\in \sigma(H)$ (so $(\omega,A_E)$ is not 
uniformly hyperbolic; if it is uniformly hyperbolic it can clearly not be minimal) and $M\neq\T\times \PR$, then $E$ must be on the edge of an open gap. 
Moreover, if $E$ is on the edge of an open gap, then it is known that $M$ cannot be the whole space $\T\times \PR$ (\cite{Jo4, JNOT}).
We could describe finer properties of the set $M$ also for these values of $E$, in the same way as we did in \cite{Bj2}. Moreover, in each open gap we could analyze, as in \cite{BS}, 
the asymptotic of the minimal angle between the stable and unstable bundles $W_E^{s,u}(\theta)$ as the cocycle $(\omega,A_E)$ looses hyperbolicity by letting $E$ approach a gap edge. 
However, to keep the present paper of a reasonable size, we will
postpone such extensions to future publications.

(c) The above results implies the following: if the minimal set $M$ is almost-automorphic (in particular, it is not the whole space), then $2\alpha(E)\in \mathcal{M}$. 
This question was raised by Johnson \cite{Jo2}. In that paper he gives an example of an an almost-periodic (but not quasi-periodic) differential equation for which the projective flow 
has an almost-automorphic covering, and for which the rotation number (and any integer multiple of it) is not in the frequency module. We have thus seen that this cannot happen in our case.
However, it seems to be an open question whether it is possible to construct a quasi-periodic example like the one by Johnson (he asks this question in \cite{Jo2}).

(d) We note that Theorem \ref{thm2} applies to the almost Mathieu case, i.e., the case where $f(\theta)=\cos(2\pi\theta)$. This is by far the most studied Schr\"odinger cocycle 
(see, for example, \cite{Ji} and references therein). In this case Puig \cite{Pu1} has shown, under the conditions that $\omega$ is Diophantine and $\lambda$ is sufficiently large, that
all the gaps are open, meaning that to each frequency in $\mathcal{M}$ there corresponds an open gap. In this case Theorem \ref{thm2} thus implies that $\Phi_E$ is minimal
iff $2\alpha(E)\notin \mathcal{M}$. An interesting question is whether one can have examples (of $f$) where $\Phi_E$ is minimal and $2\alpha(E)\in \mathcal{M}$.

(e) In \cite{Bj1} we showed that $\Phi_E$ "often" is minimal (here we could not consider all parameter values $E$; instead we could handle a more general class of $f$). 
Once we have gained control on the dynamics of $\Phi_E$ (which is the main part of this paper; see Section \ref{the_induction}), 
the way to prove minimality follows the same strategy as in \cite{Bj1}.  

\end{rem}

\bigskip
We shall now briefly recall some related results from the literature. Since we in this
paper only shall treat the situation when the coupling constant $\lambda$ is large, we focus on results in this direction.
There are, of course, many striking results in the case where $\lambda$ is small (for example \cite{E1,Pu2}).

During the last decade there has been a remarkable progress in the understanding of the spectral properties of the quasi-periodic Schr\"odinger operator
(\ref{operator}). In particular in the case where $f$ is very smooth (real-analytic or Gevrey) \cite{BG, E, GS1,GS2,GS3, Ji, Kl}. For the rest of this
summary we focus on the more robust quantity $\gamma(E)$.

It seems a very natural question to ask under which conditions on the function $f:\T\to \mathbb{R}$ (and $\omega$)
one has that the Lyapunov exponent $\gamma(E)\gtrsim \log\la$ for all $E$, provided that $\lambda$ is large. 
This problem has been studied by many authors. Herman \cite{H}, 
using his ``subharmonic trick'', showed that this holds for all (non-constant) trigonometric polynomials. Later
Sorets-Spencer \cite{SS} extended this result to also include all non-constant real-analytic $f$. In both these results,
the frequency $\omega$ plays no role, so they work for any irrational $\omega$. 
After these two fundamental results, further development started to become considerable more technical. 
By extending techniques from \cite{BG,GS1}, Klein \cite{K} showed that $\gamma(E,\lambda)\gtrsim \log\la$ for all 
$E\in\mathbb{R}$ for
``non-flat'' functions $f$ belonging to a Gevrey class, and $\omega$ satisfying a (strong) Diophantine condition.
Earlier Eliasson \cite{E} had shown that for ``non-flat'' $f$ from a Gevrey class, and $\omega$ Diophantine, the Schr\"odinger
operator $H_\theta$ has pure point spectrum for a.e. $\theta\in\mathbb{T}$ for all large $\lambda$. By the Kotani theory \cite{K}
it therefore follows that $\gamma(E)>0$ for a.e. $E$. In \cite{Bj1} we showed that if $f$ is $C^1$ (and non-constant), 
then there is a large set of $\omega$ such that $\gamma(E)\gtrsim \log\la$ for a large set of $E$, under the
condition that $\lambda$ is large. Here we had to exclude certain values of $E$ due to resonances.
By using a method of variation to avoid resonances, Chan \cite{C} has shown that 
$\gamma(E)\gtrsim \log\la$ for all $E$ for a very large class of $C^3$-functions $f$, and a large set of $\omega$. 
The techniques used are in the spirit of those developed in \cite{GS1,GS2,GS3}.  

There are also ``negative'' results. In \cite{BDJ} we show that for any given sequence $0<\lambda_1<\lambda_2<\ldots$ of positive 
real numbers, there is a generic set of $f\in C(\T)$ such that $\inf_{E\in\mathbb{R}}\gamma(E)=0$ for every $\lambda=\lambda_m$, $m\geq 1$. Since $\gamma(E)$ need not be continuous (see, for example,
\cite{WY}), this
does not necessarily mean that $\gamma(E)=0$ for some $E$.
We also remark that one can find arbitrarily large (i.e., $\max f - \min f$ being large) real-analytic $f$ such that, letting $\lambda=1$, one has $\gamma(E)=0$ for certain $E\in \R$ \cite{Bj3}.

We have so far only focused on base dynamics given by irrational translation on $\T$. Of course it is of great importance to investigate Schr\"odinger cocycles 
over translation $T:\theta\mapsto \theta+\omega$ on $\T^d$, $d\geq 2$. However, this is considerable more difficult (especially in the regime of large coupling).
We mention results by Bourgain \cite{B}.

\bigskip
The rest of this paper is organized as follows. In Section \ref{basic_estimates} we introduce some notations and derive
some elementary estimates which are used along the paper. In particular we handle the (trivial) case when the parameter $E$ is large. After this we
begin the harder analysis. This analysis is based on an inductive construction, which is the content of Section \ref{the_induction}. This long section is divided into three subsections:
(\ref{a_few_comments}) comments on the approach, (\ref{base_case}) the base case, and (\ref{induction}) the inductive step. In the first two of these subsections we have tried to 
make the philosophy of our approach as accessible as possible. In the third one we formulate and prove the inductive step.
With the control on the iterates obtained in Section \ref{the_induction}, we can rather easily finish the proof of Theorems \ref{Maintheorem} and \ref{thm2}.     
The details is the content of Section \ref{the_proof}.

The paper also contains two appendices, Appendix A and B. They contain results of a more computational nature. 
The point of collecting such results at the end of the paper
is to make the (already quite involved) steps in the inductive construction more transparent.   
In the first of these appendices we derive derivative estimates which are frequently used in 
Section \ref{the_induction}. It also contains results on how to obtain information on iterates, knowing something about another orbit  
(so-called "shadowing"; see subsection \ref{a_few_comments}). Appendix B is devoted to some geometric results about graphs $r=\psi(\theta)$ 
(for a few classes of functions) on sets of $\theta$ where $\psi(\theta)$ is close to $0$. These results are then used in the geometric parts 
of the inductive construction. The formulae derived and used in the paper are very explicit and are tailor fit for the study of the map $\Phi_E$.
However, the general properties and mechanisms should be the same for other classes of maps.  
\end{section}

\begin{section}{Notations and basic estimates}\label{basic_estimates}

The function $f:\T\to\mathbb{R}$ and the base frequency $\omega$ are from now on assumed to be fixed, being as in the statement of Theorem 
\ref{Maintheorem}. By $\Phi$ we shall always denote the map defined in (\ref{mapPhi}).

\medskip
A convention which we will use in the paper is to denote positive numerical constants by ``$\co$''. These
constants are only allowed to depend on $f$, and $\kappa,\tau$ 
(from the Diophantine condition on $\omega$). They do not depend on $\la$ and $E$. Moreover, if they appear in an inductive argument, they do 
not depend on the step of the induction.

\begin{subsection}{Some notations}

Given as subset $I\subset \T$, we shall denote the supremum norm by $\|\cdot\|_{C^0(I)}$, i.e., 
$$
\|g\|_{C^0(I)}=\sup_{\theta\in I}|g(\theta)|.
$$
Moreover, if $g:\T\to \mathbb{R}$ is $C^2$  we use the notation
$$
\|g\|_{C^2(I)}=\sup_{\theta\in I}\max\{|g(\theta)|,|g'(\theta)|,|g''(\theta)|\}.
$$

If $A$ is a matrix, then $\|A\|$ denotes the operator norm of $A$; 
if $I\subset \T$ is a Lebesgue measurable set, then we denote by $|I|$ the Lebesgue measure of $I$.

By $[\cdot]$ we denote the integer part of a real number.
We also introduce the projections onto the first and second coordinate, respectively, by
$$
\pi_1(\theta,r)=\theta \quad\text{and} \quad \pi_2(\theta,r)=r.
$$

By $\partial I$ we denote the boundary of a set $I\subset \T$.
 
If $I\subset \mathbb{T}$ (any set), and $\theta_0\in \T$, then we let 
\begin{equation}\label{function_N}
\mathcal{N}(\theta_0;I)=\begin{cases}\min\{k\geq 0: \theta_k\in I\}, &\text{if the such a $k$ exists}\\ 
\infty &\text{otherwise.}
\end{cases}
\end{equation}
Thus, $\N(\theta_0;I)$ is the ``first entry time'' of the point $\theta_0$ to the set $I$ under iteration of the map 
$\theta \mapsto \theta+\omega$. Note that $\N(\theta_0;I)=0$ if $\theta_0\in I$.  

\end{subsection}
\begin{subsection}{Basic lemmas and estimates}
We now state a results which relates uniform hyperbolicity of the cocycle $(\omega,A_E)$ with information on products $r_0r_1\cdots r_k$.
If for some parameter value $E$ we can obtain a "uniform" growth on the products for all $\theta$ in a small interval, then the cocycle $(\omega,A_E)$
is in fact uniformly hyperbolic. More precisely:
\begin{lem}\label{BE_UH}
Assume that for some value of $E$ there is an (non-degenerate) interval $I\subset \T$ and constants $c>0, K>1$ such that
for all $\theta_0\in I$ there is an $r_0\neq \infty$ such that
$$
\max\{|r_0\cdots r_{k-1}|, |r_0\cdots r_{k}|\}\geq cK^k \quad\text{ for all } k\geq 0. 
$$ 
Then the cocycle $(\omega,A_E)$ is uniformly hyperbolic.
\end{lem}
\begin{rem}Such a result was also used and proved in \cite{BS}. 
\end{rem}
\begin{proof} It is a general fact that the cocycle $(\omega,A_E)$ is uniformly hyperbolic if (and only if) there exists constants $c'>0, K>1$ such 
that
\begin{equation}\label{BE_eqUH}
\|A^n_E(\theta)\|\geq c'K^n \quad\text{for all } \theta\in \T \text{ and } n\geq 0.
\end{equation}
For a proof see \cite[Proposition 2]{Yoc}. 

We note that if there is an interval $I'\subset \T$ and a constant $c''>0$ such that
\begin{equation}\label{BE_eqUH2}
\|A^n_E(\theta)\|\geq c''K^n \quad\text{for all } \theta\in I' \text{ and } n\geq 0,
\end{equation}
then there is a constant $c'>0$ such that (\ref{BE_eqUH}) holds. 
Indeed, since $\omega$ is irrational there exists an $N>0$ such that $\bigcup_{k=0}^N(I'+k\omega)=\T$. Thus, for any
$\theta\in \T$ we can find $k\in [0,N]$ such that $\theta_k\in I'$. This gives
$$
\|A^n_E(\theta)\|=\|A^{n-k}(\theta_k)A^k(\theta)\|\geq \frac{\|A^{n-k}(\theta_k)\|}{\|A^k(\theta)\|}>c'K^n,
$$  
some $c'>0$.
Here we have used the inequality $\|AB\|\geq \|A\|/\|B\|$ for $SL(2,\mathbb{R})$ matrices (which holds since $\|B^{-1}\|=\|B\|$), and
the fact that $k$ and the $\|A_E(\cdot)\|$ are bounded by constants which are independent of $n$. 

Finally, using the assumptions on the growth of the products together with the relation (\ref{matrix_r}), i.e., 
$$
A^n_E(\theta)\binom{1}{r_0}=\binom{r_0\cdots r_{n-1}}{r_0\cdots r_n},
$$
we see that (\ref{BE_eqUH2}) holds.
\end{proof}

In this section the coupling constant $\lambda$ is assumed to be sufficiently large, 
depending only on the function $f$. Thus, in the lemmas in this section we presuppose that $\lambda$ is a large number.

\bigskip
We shall now define an important set, denoted $J_0$. For each $E\in \mathbb{R}$, let the "critical" set $J_0=J_0(E)$ be defined by
\begin{equation}\label{setI}
J_0=\{\theta\in\T:|\lambda f(\theta)-E|\leq 2\lambda^{3/4}\};
\end{equation}
see Fig. \ref{fig1}.
By the assumptions on $f$ it immediately follows that $|J_0|\to 0$ as $\lambda\to\infty$. Thus, $J_0$ 
is a small set. We shall derive more information about $J_0$ in Lemma \ref{BE3} below, but we first state a trivial, but fundamental, fact
about iterations of the map $\Phi$ for $\theta$ outside the set $J_0$. 
\begin{figure}
\psfrag{E}{$E$}
\psfrag{V}{$\la f(\theta)$}
\psfrag{J}{$J_0$}
\psfrag{J2}{$J_0$}
\psfrag{t}{$\theta$}
\includegraphics[width=8cm]{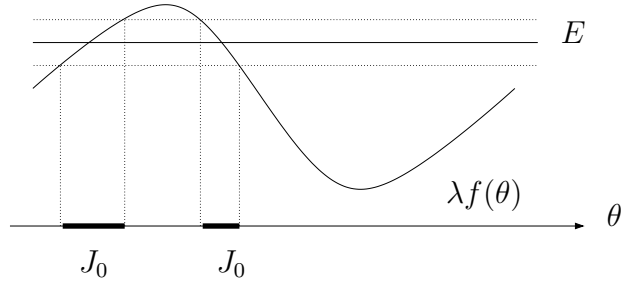}
\caption{The set $J_0$.}\label{fig1}
\end{figure}

\begin{lem}\label{BE1}
For any fixed $E\in \mathbb{R}$, the following holds for iterations of $\Phi$:
\begin{itemize}
\item If $|r_0|> 1/\la^{3/4}$ and $\theta_0\in \overline{\T\setminus J_0}$, then $|r_1|> \la^{3/4}$. 
\item If $|r_0|\geq \la^{3/4}$ and $\theta_0\in \overline{\T\setminus J_0}$, then $|r_1|>(3/2)\la^{3/4}$.
\end{itemize}
\end{lem}
\begin{proof}By definition we have $|\la f(\theta)-E|\geq 2\la^{3/4}$ on $\overline{\T\setminus J_0}$. Using the assumption on $\theta_0, r_0$ we therefore get 
$$
|r_1|=|\lambda f(\theta_0)-E-1/r_0|> 2\la^{3/4}-\la^{3/4}=\la^{3/4}.
$$
The second statement is proved in the same way.
\end{proof}

Next we introduce the interval $\mathcal{E}\subset \mathbb{R}$, which is defined by 
\begin{equation}\label{setE}
\mathcal{E}=(\lambda f_{\text{min}}-2\la^{3/4},\lambda f_{\text{max}}+2\la^{3/4}).
\end{equation}
In fact, for all $E\notin \mathcal{E}$ we can easily estimate the Lyapunov exponent. Moreover, the cocycle $(\omega,A_E)$ is uniformly hyperbolic outside $\mathcal{E}$.
\begin{lem}\label{BE2}
If $E\notin \mathcal{E}$, then the cocycle $(\omega,A_E)$ is uniformly hyperbolic and the Lyapunov exponent $\gamma(E)\geq 3(\log\lambda)/4$.
\end{lem}
\begin{proof}Note that the assumption on $E$ implies that we have $|\la f(\theta)-E|\geq 2\la^{3/4}$ for all $\theta\in\T$. Thus $\overline{\T\setminus J_0}=\T$.
Take $r_0$ such that $|r_0|\geq \la^{3/4}$, and take any $\theta_0\in\T$. 
Applying Lemma \ref{BE1} repeatedly shows  
that $|r_k|\geq \la^{3/4}$ for all $k\geq 0$. Hence we have
$$
|r_0\cdots r_{k-1}|\geq \la^{(3/4)k}\quad\text{ for all } k\geq 1,
$$
so
$$
\liminf_{k\to\infty}\frac{1}{k}\log|r_0\cdots r_{k-1}|\geq \frac{3}{4}\log\la. 
$$
By (\ref{matrix_r}) and (\ref{kingman}) this shows that $\gamma(E)\geq (3/4)\log\la$. That the cocycle is uniformly hyperbolic follows from Lemma \ref{BE_UH}.
\end{proof}
Thus we know that Theorem \ref{Maintheorem} holds for all $E\notin\mathcal{E}$. This is of course the trivial case. 
In the sequel we shall therefore focus only on $E\in\mathcal{E}$. 

For $E\in \mathcal{E}$ we have the following obvious result, 
which we often shall use without any reference.  
\begin{lem}\label{BE99}
For any $\theta_0\in\T$, if $E\in \mathcal{E}$ and $|r_0|<\la^{-2}$, then $|r_1|>\la^2/2$.
\end{lem}
\begin{proof}
We have $|r_1|=|\la f(\theta_0)-E-1/r_0|>\la^2-\co\la>\la^2/2$.
\end{proof}

\bigskip

The next lemma contains information about the set $J_0$ defined in (\ref{setI}); see Fig. \ref{fig1}. 
\begin{lem}\label{BE3}
For every $E\in \mathcal{E}$ the set $J_0$ consist of one or two intervals, and 
$$
\begin{aligned}
\min_{\theta\in J_0}&\max\{|\la f'(\theta)|,\la f''(\theta)\}\geq \co \la; \quad\text{ or } 
\\ \min_{\theta\in J_0}&\max\{|\la f'(\theta)|,-\la f''(\theta)\}\geq \co \la.
\end{aligned}
$$

Moreover: 
\begin{enumerate}
\item if $J_0$ consists of two intervals, $J_0^1,J_0^2$, then 
$$
\la f(J^i)-E=[-2\la^{3/4},2\la^{3/4}],
$$
and $f'(\theta)>0$ on $J_0^1$ and $f'(\theta)<0$ on $J_0^2$;

\smallskip
\item if $J_0=[a,b]$, then $\la f(J_0)-E\subset [-2\la^{3/4},2\la^{3/4}]$, and $\la f(a)-E=\la f(b)-E=\pm 2\la^{3/4}$. 
\end{enumerate}
\end{lem}
\begin{proof} Since $f$ has exactly two (non-degenerate) critical points there is a constant $c>0$ such that
$\min_{\theta\in\T}\max\{|f'(\theta)|,|f''(\theta)|\}\geq c$, and there are two closed intervals $A,B$, $A\cup B=\T$, such that 
$f$ is strictly increasing on $A$ and $f$ is strictly decreasing on $B$. From this it follows that  
a horizontal line $y=a$ can intersect the curve
$y=f(\theta)$ in $0$ (if $a\notin [f_{\min},f_{\max}]$), $1$ (if $a=f_{\min},f_{\max}$) or $2$ points (if $a\in (f_{\min},f_{\max})$) in $\T$. 
Moreover, since 
$$
J_0=\{\theta\in\T:|f(\theta)-E/\la|\leq 2\lambda^{-1/4}\},
$$
and $\lambda$ is very large, at most one of $f$:s critical points can be in $J_0$.
And since $E\in \mathcal{E}$, it follows that $J_0\neq \emptyset$. The above observations therefore give the result.
\end{proof}

Later in the paper we shall often have lower bounds on either the first or the second derivative of some function (as in the previous lemma, for example), 
together with some information about the
image of the function. In the
following lemma we shall derive some easy estimates which follow from such information.

\begin{lem}\label{BE9}
Assume that $I$ is an interval and that $g$ is a $C^2$-function on $I$ satisfying
$$
\min_{\theta\in I}\max\{|g'(\theta)|,g''(\theta)\}\geq D>0.
$$
Assume further that $\delta$ is a number such that $0<\delta\leq D/4$.
\begin{enumerate}
\item If $g(I)=[-\delta,\delta]$, and $g'(\theta)>0$ on $I$, then 
$$
|I|\leq 2\sqrt{\frac{\delta}{D}}.
$$
Moreover, if $I'=\{\theta\in I: |g(\theta)|\leq \frac{3}{4}\delta\}$, then $I'$ is an interval, and
$$
g'(\theta)\geq \sqrt{\delta D/2} \text{ on } I'.
$$

\item If $g(I)\subset [-\delta,\delta]$, then $|I|\leq 4\sqrt{\frac{\delta}{D}}$.
\end{enumerate}
\end{lem}
\begin{rem}In all our applications of this lemma, we shall have $D\gg 1$ and $\delta\ll 1$.
\end{rem}

\begin{proof}[Proof of Lemma \ref{BE9}]
We shall write $I=[a,b]$. 

(1) First we note that if $g'(t)\geq D$ for some $t$, then $g'(\theta)\geq D$ for all $\theta\in I$ such that $\theta\geq t$. 
Thus, if $g'(a)\geq D$, then $g'(\theta)\geq D$ for all $\theta\in I$, and we get $|I|\leq 2\delta/D<2\sqrt{\delta/D}$.

We now assume that $g'(a)<D$, and let $J=\{\theta\in I: g'(\theta)\leq D\}$. Then $J$ is an interval, which the above argument shows, 
and we have $g''(\theta)\geq D$ on $J$.
We write $J=[a,c]$. Since $g(a)=-\delta$ and $g'(a)>0$, it follows that $|J|\leq \sqrt{2(g(c)+\delta)/D}$ (since $g'(\theta)\geq D(\theta-a)$ on $J$). Moreover, since 
$g'(\theta)\geq D$ outside $J$ we have
we have $|I\setminus J|\leq (\delta-g(c))/D$. From the assumption $0<\delta\leq D/4$ it follows that $\sqrt{2(x+\delta)/D}+(\delta-x)/D$ has its largest value on 
$[0,\delta]$ when $x=\delta$. This gives the required upper bound on $|I|$. 

We turn to the estimate on $g'(\theta)$ on $I'$. Let $s\in I$ be any point for which $g(s)\geq -3\delta/4$. If $s\in I\setminus J$ then $g'(s)\geq D>\sqrt{\delta D/2}$. 
If not, then $[a,s]\subset J$. In this case we have $g'(s)-g'(\theta)\geq D(s-\theta)$ for $\theta\leq s$, i.e., $g'(\theta)\leq g'(s)-D(s-\theta)$.
Thus we have $\delta/4\leq g(s)-g(a)\leq (s-a)g'(s)-D(s-a)^2/2$, so
$$
g'(s)\geq \frac{\delta/4}{s-a}+\frac{D(s-a)}{2}.
$$
It is easy to check that the RHS is as small as possible when $s-a=\sqrt{\delta/(2D)}$. Hence it follows that $g'(s)\geq \sqrt{\delta D/2}$.

(2) The extreme case is when $g(a)=g(b)=\delta$, and that there is a $c\in I$ such that $g(c)=-\delta$ (and hence $g'(c)=0$). Similar estimates to the ones above show that
$c-a,b-c\leq 2\sqrt{\delta/D}$. 
\end{proof}

Combining Lemmas \ref{BE3} and \ref{BE9} gives us
\begin{lem}\label{BE10}
For all $E\in \mathcal{E}$ we have that the set $J_0$ consists of one or two intervals, and 
$$
|J_0|\leq \co \la^{-1/8}.
$$
\end{lem}
\begin{proof}Use $\delta=2\la^{3/4}$ and $D=\co \la$ in Lemma \ref{BE9}.
\end{proof}

\bigskip

As we have already seen, one of the main quantities that  we have to keep track of are products of the form $$r_0r_1\cdots r_k$$ (recall, for example,
formula (\ref{matrix_r}), or the proof of Lemma \ref{BE2}). 
A little technicality which arises
is that we sometimes ``multiply with zero'', that is, one $r_j$ could be zero. But if this happens, we have the following ``cancellation'':
$$r_jr_{j+1}=r_j(\la f(\theta_j)-E-1/r_j)=r_j(\la f(\theta_j)-E) -1 \approx -1 \text{ if }  r_j \approx 0.$$ In a product $r_0\cdots r_k$, for which
$r_k$ is not too small (if $r_k=0$, the product is zero) we shall therefore pair two factors $r_j$ and $r_{j+1}$
if $r_j$ is very small, and use the ``cancellation''. This is summarized in the next lemma.

\begin{lem}\label{BE4}
Assume that $E\in \mathcal{E}$. For any $(\theta_0,r_0)\in\T\times \widehat{\mathbb{R}}$ we have the following:  if $|r_k|\geq \la^{-2}$, some $k\geq 0$, then
we can write the product $r_0r_1\cdots r_k$ as
$$r_0r_1\cdots r_k=\rho_0\cdots \rho_l ~(l\leq k),$$ where each $\rho_i$ satisfies
$$
|\rho_i|\geq \la^{-2};\quad \text {and either}
$$
$$
\rho_i=r_j \text{ and } |r_j|\geq \la^{-2}, \quad\text{or} \quad \rho_i=r_jr_{j+1} \text{ and } |r_j|< \la^{-2}.
$$
Moreover, we have the upper bounds
$$
|\rho_i|<\la^2, i\in [1,l].
$$

\end{lem}
\begin{proof}
Let $j_0\in [0,k-1]$ be the smallest index such that $|r_{j_0}|<\lambda^{-2}$, if there is one (by assumption we have $|r_k|\geq \la^{-2}$). 
We let $\rho_i=r_i$ for all $i\in [0,j_0-1]$, and let $\rho_{j_0}=r_{j_0}r_{j_0+1}$.
Then
$|\rho_{j_0}|=|r_{j_0}r_{j_0+1}|=|r_{j_0}||\la f -E -1/r_{j_0}|>1/2\gg\la^{-2}$. 
Let $j_1\in [j_0+2,k-1]$ be the smallest integer such that $|r_{j_1}|<\lambda^{-2}$, and proceed as before.
By induction the first statemet of the lemma follows. The second statement is shown similarly.
\end{proof}

This pairing introduced in the previous lemma  will be important when we later shall differentiate expressions like $r_0(\theta)\cdots r_k(\theta)$ 
with respect to $\theta$ (since some of the $r_j$ may be extremely large). 
We then rewrite them as $\rho_0\cdots \rho_l$ as above, and obtain uniform bounds on the $\rho_i$. 
This "problem" is of course an artifact coming from our choice of working whith $\PR$ as fiber space;
had we viewed $\Phi$ as a map on $\T\times S^1$ instead, and used derivatives to estimate the Lyapunov exponent (the contraction in the fiber direction), we would
not have this trouble. However, other things seem to be more convenient working in our coordinates. 

The next lemma gives information about the contraction in the fiber direction under an assumption on the growth of the product of the iterates.
It is formulated in exactly the setting that we will use it in the comming sections.

\begin{lem}\label{BE20}
Assume that $E\in\mathcal{E}$ and that for some $\theta_0\in\T$ the following holds: for all $|r_0|\geq \la^{3/4}$ and all $k\geq 1$ such that $|r_k|\geq \la^{-2}$ we have
$|r_1\cdots r_k|\geq \la^{(2/3)k}$. If $|t_0|\geq \la^{3/4}$ and $|s_0|\geq \la^{3/4}$, and if $|s_{k+1}|<\la^{3/2}$ for some $k\geq 1$, then
$$
|t_{k+1}-s_{k+1}|\leq 2\la^{-((4/3)k+3/4)}.
$$
\end{lem}
\begin{proof}Since $|s_{k+1}|<\la^{3/2}$ we must have $|s_k|>\la^{-2}$ (recall Lemma \ref{BE99}). Therefore it follows from the assumptions that
$|s_1\cdots s_k|\geq \la^{(2/3)k}$. Furthermore,
$
|t_1-s_1|=|1/s_0-1/t_1|\leq 2\la^{-3/4}.
$
We now show that we must have $|t_{k+1}|<2\la^{3/2}$. To do this, we assume that there exists a $p_0$, $|p_0|\geq \la^{3/4}$, such that $|p_{k+1}|=2\la^{3/2}$.
Then we would have $|p_k|>\la^{-2}$, and using the assumption we would also have $|p_1\cdots p_k|\geq \la^{(2/3)k}$. Using formula (\ref{contrF}) we get
$$
\la^{3/2}<|p_{k+1}-s_{k+1}|=\frac{|p_1-s_1|}{|p_1\cdots p_k||s_1\cdots s_k|}\leq 2\la^{-3/4}\cdot \la^{-(4/3)k}
$$
which is a contradiction. We conclude that $|r_{k+1}|<2\la^{3/4}$. From the computations above (replacing $p$ by $r$) we get the desired estimate on
$|r_{k+1}-s_{k+1}|$. 
 
\end{proof}

\bigskip
We end this section by stating two elementary facts which follow from the Diophantine condition (\ref{DC}) on $\omega$, that is, the condition
$$
\inf_{p\in\mathbb{Z}}|q\omega-p|>\frac{\kappa}{|q|^\tau}\quad\text{for all } q\in\mathbb{Z}\setminus \{0\}. 
$$
\begin{lem}\label{BE7}
Assume that $J\subset \mathbb{T}$ is an interval of length $\leq\varepsilon$. Then
$$
J\cap\bigcup_{0<|m|\leq N}(J+m\omega)=\emptyset
$$
with $N=[(\kappa/\varepsilon)^{1/\tau}]$.
\end{lem}
\begin{proof}The proof is standard. If
$(J+m\omega)\cap J\neq \emptyset$, we must have $|m\omega+p|\leq \varepsilon$ for some integer $p$.
Thus we get from (\ref{DC}) 
$$
\frac{\kappa}{|m|^\tau}<|m\omega+p|\leq \varepsilon,
$$ 
which implies that $|m|>(\kappa/\varepsilon)^{1/\tau}$.
\end{proof}
By using the above lemma we easily get the following result, which will be used in the so-called resonant case in the next section (see, for example. Fig \ref{fig9}).

\begin{lem}\label{BE8}Assume that $J^{i}\subset \T$ ($i=1,2$) are two disjoint intervals, each of length $\leq \ve$.
Assume further that there is an integer $\nu\in [1,N]$ such that $J^1\cap(J^2-\nu\omega)\neq \emptyset$, where
$N=[(\kappa/2\varepsilon)^{1/\tau}]$. If we let $I=J^1\cup(J^2-\nu\omega)$, then
$$
(J^1\cup J^2)\cap\bigcup_{m\in [-N+\nu,N]\setminus\{0,\nu\}}(I+m\omega)=\emptyset; \quad \text{and}
$$
$$
I\cap J^2=\emptyset  \quad \text{and} \quad (I+\nu\omega)\cap J^1=\emptyset.
$$
\end{lem}
\begin{proof}From the assumptions it follow that the interval $I$ is of length $\leq 2\ve$. Using Lemma \ref{BE7}, recalling the definition of $N$, we get
\begin{equation}\label{BE8_eq1}
I\cap\bigcup_{m\in[-N,N]\setminus\{0\}}(I+m\omega)=\emptyset.
\end{equation}
In particular we have $I\cap (I+\nu\omega)=\emptyset$, since $\nu\in [1,N]$. Therefore,  
since $J^1\subset I$ and $J^2\subset I+\nu\omega$, we see that the two last statements of the lemma hold.
Moreover, (\ref{BE8_eq1}) implies
$$
(I+\nu\omega)\cap\bigcup_{m\in[-N+\nu,N+\nu]\setminus\{\nu\}}(I+m\omega)=\emptyset.
$$
Combining this with (\ref{BE8_eq1}) gives the first statement.

\end{proof}

\end{subsection}
\end{section}

\begin{section}{The induction}\label{the_induction}

In this section we shall make an inductive construction which will give us detailed control 
of the iterates of the map $\Phi$, as defined in (\ref{mapPhi}), for a set of positive measure of initial points $(\theta,r)$. 
This will be the content of Proposition \ref{inductive_lemma} (the inductive step), combined with Lemmas \ref{BC3} and \ref{BC4} (the base case). 
With this information at hand, the statements in Theorems \ref{Maintheorem} and \ref{thm2} will follow easily, as we will see in Section \ref{the_proof}.

\bigskip
We shall from now on always assume that $\lambda$ is sufficiently
large, the size only depending on $f,\kappa$ and $\tau$ (recall that we have fixed the function $f$ and the Diophantine number $\omega$ ). 
In particular, the size of $\lambda$ is not
allowed to depend on on which step of the induction we are, nor on the parameter $E$. 
We stress that there are only finitely many conditions on the size of $\lambda$ (all being of finite size).

For any fixed value of the energy $E\in \mathcal{E}$, $\mathcal{E}$ being the set defined in (\ref{setE}), we shall investigate the dynamics of $\Phi=\Phi_E$. 
We assume from now on that we have have chosen an $E\in \mathcal{E}$.
Thus, all the parameters appearing in the map $\Phi$ are fixed.
 
The global goal is now to construct, by means of an inductive procedure, a set 
$\Theta_\infty=\Theta_\infty(f,\omega,E,\la)\subset \T$ of positive Lebesgue measure such that 
for all $\theta_0\in \Theta_\infty$ and $|r_0|\geq \la^{3/4}$ we have a good control on the iterates $r_k$ for all $k\geq 0$. This will be achieved 
by constructing sets $\Theta_0\supset \Theta_1 \supset \ldots \supset \Theta_\infty$ on which we have control on longer and longer iterates. 
With this information available we will be able to derive the results in Theorems \ref{Maintheorem} and \ref{thm2}. 

\begin{subsection}{A few comments on the approach}\label{a_few_comments}
Before going into the construction, we first comment one some of the main ingredients in the approach that we shall use. In this discussion we focus on 
the base case of the induction, which is the content of the next section, Section \ref{base_case}. However, the overall philosophy is the same also in the inductive step (Section \ref{induction}).

As we already have mentioned, the key quantity to control is products like $r_0\cdots r_k$. We will see that such a control also will give us information about the geometry of
certain objects which will play a crucial role in our analysis. 
First, from Lemma \ref{BE1} it follows that if $\theta_j\notin J_0$ for
$j\in [0,k-1]$ and $|r_0|\geq \la^{3/4}$, then we have $|r_j|\geq\la^{3/4}$ for all $j\in [0,k]$, and $|r_0\cdots r_k|\geq \la^{(3/4)(k+1)}$.
Since $\omega$ is irrational, when iterating any point $(\theta_0,r_0)$ we will always, sooner or later, enter over the set $J_0$, i.e., there is a smallest $k\geq 0$ such that
$\theta_k\in J_0$. After this, a priori,  we loose control on further iterates,
since we have no control on $r_{k+1}$ (if $r_k$ is large, we can have that $r_{k+1}=\la f(\theta_k)-E-1/r_k$ is small). 
To understand what happens further (to most points), after entering over $J_0$, we shall first use what we call probing.

\emph{Probing}: In this discussion we shall, for simplicity, assume that $J_0$ consists of a single (tiny, since $\lambda\gg 1$; see Lemma \ref{BE10}) interval.
(This means that $E$ is close to $\la \max f$ or $\la \min f$; see Fig. \ref{fig1}.) 
From the Diophantine condition on $\omega$ it follows that
if $\theta\in J_0$, then there is an $N\gg 1$ such that $\theta+j\omega\notin J_0$ for all $j$ such that $1\leq |j|\leq N$ (see Lemma \ref{BE7}). The "probe" is defined as follows. 
Let $M$ and $K$ be integers such that $1\ll K\ll M \ll N$, and let 
\begin{equation}\label{FC_eq1}
\vf(\theta)=\pi_2(\Phi^{M+K}(\theta-M\omega,\infty)).
\end{equation}
Letting $\Gamma$ be the segment
$$
\Gamma=\{(\theta,\infty): \theta\in I_0-M\omega\},
$$
we then have 
$$
\Phi^{M+K}(\Gamma)=\{(\theta+K\omega,\vf(\theta)): \theta\in I_0\},
$$
see Fig. \ref{fig2}.
What we are interested in, and what we will need to control, is $\left.\vf\right|_{J_0}$. More precisely, we shall let $J_1\subset J_0$ be the set
$$
J_1=\{\theta\in J_0: |\vf(\theta)|\leq 2\la^{-3/4}\}
$$
and we need estimates on this set. To describe the set $J_1$, and to make use of the probe to obtain control on other iterates, we shall use shadowing. 
\begin{figure}
\psfrag{o}{$\infty$}
\psfrag{f}{$\Phi^{M+K}(\Gamma)$}
\psfrag{i}{$J_0-M\omega$}
\psfrag{i2}{$J_0+K\omega$}
\psfrag{i3}{$J_0$}
\psfrag{t}{$\theta$}
\psfrag{g}{$\Gamma$}
\includegraphics[width=8cm]{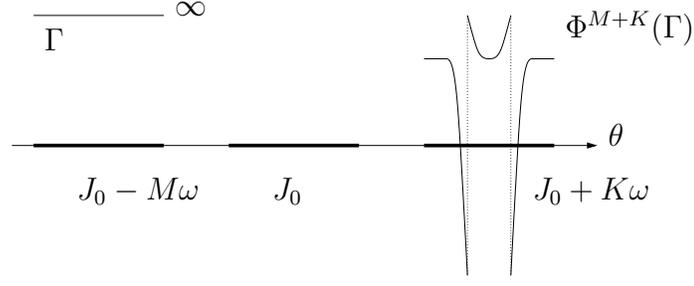}
\caption{The "probe".}\label{fig2}
\end{figure}

\emph{Shadowing}: By shadowing we mean the following: Assume that $\theta_0\in \T$ and that for some $r_0\in \PR$ (we shall often use $r_0=\infty$; cf. the definition of
$\vf$ above) 
we have control on the orbit $(\theta_j,r_j)$ up to some time $j=T$. Given an $s_0\in \PR$ we want to compare the orbit of $(\theta_0,s_0)$ with that of $(\theta_0,r_0)$.  
In fact, the difference $r_0-s_0$ can be computed explicitly (see formula (\ref{A_form1}) in Appendix A), and if we have a good control on the products $r_1\cdots r_j$,
we will have that $r_j-s_j$ are very close for most $s_0$ (see Proposition \ref{A5}). 

By applying shadowing we can use information about the probe $\vf$ to gain control on longer orbits (not only until they hit $J_0$) 
for many points $(\theta_0,r_0)$. We shall "follow the probe through $J_0$".  
This means the following. If $\theta_0\in J_0-M\omega$ and $|r_0|\geq \la^{3/4}$, then $|r_{M+K}|$ will be very close to $\vf(\theta_0)$; see Fig. \ref{fig2b}. 
This holds since $M\gg K$ and since
$\theta_j\notin J_0$ for $j\in [0,M-1]$. Therefore $r_{M+1}$ will be very very close to $\widetilde{\vf}(\theta_0)$, where $\widetilde{\vf}$ 
is defined by 
\begin{equation}\label{FC_eq2}
\widetilde{\vf}(\theta)=\pi_2(\Phi^{M+1}(\theta-M\omega,\infty)).
\end{equation}
With this notation we can write
$$
\Phi^{M+1}(\Gamma)=\{(\theta+\omega,\widetilde{\vf}(\theta)):\theta\in I_0\};
$$
see Fig. \ref{fig2c}. 
Since $K$ is much smaller than
$M$, the two points $(\theta_{M+1},r_{M+1})$ and $(\theta_{M+1},\widetilde{\vf}(\theta_0))$ cannot separate much under $K-1$ further iterations. Thus, 
if $\theta_{M+K}\notin J_1+K\omega$, 
then we know, from the definition of the set $J_1$, that $|r_{M+K}|>\la^{-3/4}$. And since $J_0+K\omega\subset \T\setminus J_0$ we can use Lemma \ref{BE1} to conclude that $|r_{M+K+1}|>\la^{3/4}$.

Let us see how we now can use this information to study longer orbits. The prize we have to pay is that $\theta_0$ has to start outside a certain small set. More precisely,
assume that $\theta_0\in \T\setminus \bigcup_{m=-{M+1}}^0(J_0+m\omega)$ and $|r_0|\geq \la^{3/4}$. Then there is a $k_0\geq 0$ such that $\theta_{k_0}\in J_0-M\omega$ and 
$\theta_j\notin J_0$ for $j\in [0,k_0-1]$. Since the latter holds, we know that $|r_{k_0}|\geq \la^{3/4}$. If $\theta_{k_0}\notin J_1-M\omega$ ($J_1$ will in fact be a extremely small
set, so this is satisfied for most $\theta_{0}$), then  we can use the above argument to the point $(\theta_{k_0},r_{k_0})$, that is, 
follow the probe, and conclude that $|r_{k_0+M+K+1}|\geq \la^{3/4}$. Now, since $0\ll K \ll M \ll N$, we have that $J_0+(K+1)\omega \subset \T\setminus \bigcup_{m=-{M+1}}^0(J_0+m\omega)$.
Thus we can repeat this argument to the point $(\theta_{k_0+K+M+1}, r_{k_0+M+K+1})$, and hence gain control on the iterates of $(\theta_0,r_0)$, until the first time that 
an iterate $\theta_j$ hit the set $J_1$.   
\begin{figure}
\psfrag{o}{$\infty$}
\psfrag{i}{$J_0-M\omega$}
\psfrag{i2}{$J_0+K\omega$}
\psfrag{i3}{$J_0$}
\psfrag{t}{$\theta$}
\psfrag{p1}{$(\theta_0,r_0)$}
\psfrag{p2}{$(\theta_{M+K},r_{M+K})$}
\includegraphics[width=9cm]{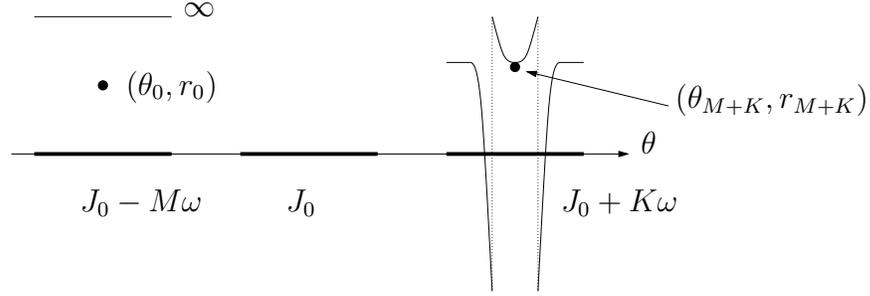}
\caption{"Shadowing the probe"}\label{fig2b}
\end{figure}

We will also need to use shadowing to control the geometry of the probe $\vf$ itself (which in turn is used to control the set $J_1$; without good control on $J_1$ we cannot
go to the next scale of the induction). Let $\vf$ and
$\widetilde{\vf}$ be defined as in (\ref{FC_eq1}) and (\ref{FC_eq2}). Since $J_0\cap (J_0-j\omega)=\emptyset$ for $j\in [1,M]$ it will in fact follow that
$\widetilde{\vf}(\theta)\approx \la f(\theta)-E$ for all $\theta\in J_0$; see Proposition \ref{A1}.  
Thus, from the assumptions on $f$ we can tell something about $\widetilde{\vf}$ over $J_0$.
To carry this information to statements about $\vf$ we let 
$$
\psi(\theta)=\pi_2(\Phi^{K-1}(\theta+\omega,\infty)).
$$
If we let $\Gamma'$ be the segment
$
\Gamma'=\{(\theta,\infty): \theta\in I_0+\omega\},
$
we have
$$
\Phi^{K-1}(\Gamma')=\{(\theta+K\omega,\psi(\theta)) : \theta\in I_0\}.
$$
By using the fact that $J_0\cap (J_0+j\omega)=\emptyset$ for $j\in [1,K]$, we will see, in the next section, that  
$\psi(\theta)\approx\la f(\theta+(K-1)\omega)-E$ for $\theta\in J_0$. Now we can
use shadowing to compare $\psi$ and $\vf$. We will get $$\vf(\theta)=\psi(\theta)-h(\theta)/(\widetilde{\vf}(\theta)-w(\theta)),$$ where $h(\theta)>0$ and $w(\theta)$ are
very small on $J_0$ (see Proposition \ref{A5}). With this information we can derive results about the set $J_1$. Estimates for this is the content of Appendix B.   

A crucial fact is that either $J_1$ consists of two intervals, and $|\vf'|$ is large on both of them, or $J_1$ consists of a single interval, and $\vf$ looks like a "parabola".
The same picture will appear on each scale of the induction, and this is why it is sufficient to assume that the function $f$ is $C^2$.

\begin{figure}
\psfrag{inf1}{$\infty$}
\psfrag{inf2}{$\infty$}
\psfrag{f}{$\Phi^{M+1}(\Gamma)$}
\psfrag{p}{$\Phi^{K-1}(\Gamma')$}
\psfrag{I1}{$J_0-M\omega$}
\psfrag{I2}{$J_0+\omega$}
\psfrag{I3}{$J_0+K\omega$}
\psfrag{t}{$\theta$}
\psfrag{g1}{$\Gamma$}
\psfrag{g2}{$\Gamma'$}
\includegraphics[width=9cm]{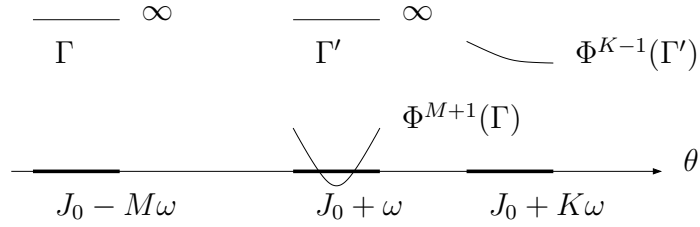}
\caption{Shadowing.}\label{fig2c}
\end{figure}

\emph{Resonance}: Since we want to analyze the map $\Phi$ for any value of $E$ we cannot avoid so-called resonances. On the first scale this means that
the set $J_0$ consists of two intervals, $J_0^1,J_0^2$, and $(J_0^1+\nu\omega)\cap J_0^2\neq \emptyset$ for some small integer $\nu$. If this happens
we shall let $I_0$ be the single interval $J_0^1\cup (J_0^2-\nu\omega)$ and send probes through it (here will have to use shadowing twice).
If there is no resonance (on the first scale), we can treat each $J_0^1$ and $J_0^2$ independently, in the same way as in the argument above. This
is the situation we always had in \cite{Bj1}, on each scale. The price we had to pay to have this was that we could not analyse the dynamics for all values
of the parameter $E$. That we in this paper can analyze the dynamics of $\Phi_E$ for any value of $E$ is the main achievement.

\end{subsection}

\begin{subsection}{The base case}\label{base_case}
We shall fist define a set $I_0\subset \mathbb{T}$ and integers $M_0\gg K_0 \gg 1$. Their definitions will depend on the set
$J_0$, which was defined in the previous section, and on the frequency $\omega$. There will be two cases:
(1) $J_0$ consists of one single interval or two so-called
``non-resonant'' intervals, and (2) $J_0$ consists of two "resonant" intervals.

From Lemma \ref{BE10} we know that the set $J_0$ consists of one or two intervals, each of length $<\co \lambda^{-1/8}$. 
Lemma \ref{BE7} tells us that for any interval $J\subset \T$ of length $<\co\la^{-1/8}$, we always have that
\begin{equation}\label{bc_nonres}
J\cap(J+k\omega)=\emptyset \text{ for all } 0<|k|\leq \co\la^{1/(8\tau)}. 
\end{equation}
However, in the case when $J_0$ consists of two intervals it is certainly possible to have $J_0\cap (J_0+k\omega)\neq \emptyset$ 
for small $k$, since we do not know the location of the two intervals in $J_0$. This "resonance phenomena", which can happen on any step of the induction, 
is a major problem for the analysis of 
the map $\Phi$, as we shall see, and must be handled with care. 
 
We now check for resonance: 

(1) (non-resonant case) If 
\begin{equation}\label{bc_check_res}
J_0\cap\bigcup_{0<|m|\leq [\lambda^{1/(25\tau)}]}(J_0+m\omega)=\emptyset
\end{equation}
we let
$$
I_0=J_0 
$$
and 
$$
K_0=[\lambda^{1/(60\tau)}],~ M_0=K_0^2
$$
(note that, in particular, this condition is satisfied when $J_0$ consists of one single interval, as we see from (\ref{bc_nonres})).
We also let 
$$\nu_0=0,$$
which will indicate that there is no resonance.

(2) (resonant case) If (\ref{bc_check_res}) does not hold, $J_0$ must consist of two intervals, and there must be an integer $$\nu_0\in [1, [\lambda^{1/(25\tau)}]]$$ 
such that $(J_0^1+\nu_0\omega)\cap J_0^2\neq \emptyset$, 
where $J_0^1$ is one of the intervals in $J_0$, and $J_0^2$ the other. In this case we let $I_0$ be the (single) interval 
$$
I_0=J_0^1\cup (J_0^2-\nu_0\omega) 
$$
(see Fig. \ref{fig9}) and define
$$
K_0=[\lambda^{1/(20\tau)}], ~M_0=K_0^2.
$$ 
We note that 
\begin{equation}\label{bc_nu0}
1\leq\nu_0\ll K_0.
\end{equation}
\begin{figure}
\psfrag{t}{$\theta$}
\psfrag{j1}{$J_0^1$}
\psfrag{j2}{$J_0^2$}
\psfrag{j3}{$J_0^2-\nu_0\omega$}
\psfrag{j4}{$I_0$}
\includegraphics[width=9cm]{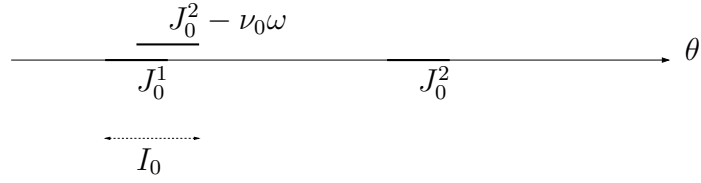}
\caption{Definition of $I_0$ in the resonant case.}\label{fig9}
\end{figure}

By the above definitions, we note that in each of the two cases, the set $I_0$ consists of one or two intervals. Moreover, since $\lambda\gg1$, the
integer $K_0$ is very large (we can make it as large as we want by taking large $\lambda$).

\bigskip
We collect some facts that follow from the above definitions of $I_0$ and $K_0,M_0$:
\begin{lem}\label{BC1} In each case we have
\begin{equation}\label{BC1_eq1}
I_0\cap \bigcup_{0<|m|\leq 2M_0}(I_0+m\omega)=\emptyset.
\end{equation}
Moreover, in the resonant case we also have
$$
J_0\cap \bigcup_{m\in [-2M_0,2M_0]\setminus\{0,\nu_0\}}(I_0+m\omega)=\emptyset, 
$$
$$
I_0\cap J_0^2=\emptyset \text{ and } (I_0+\nu_0\omega)\cap J_0^1=\emptyset,
$$ 
where $1\leq \nu_0\ll K_0$ is as in the definition of case (2).
\end{lem}
\begin{proof}In the non-resonant case we have defined $I_0=J_0$ and $M_0\approx \la^{1/(30\tau)}\ll \la^{1/(25\tau)}$. Thus, in this case 
(\ref{BC1_eq1}) follows from (\ref{bc_check_res}). In the resonant case $I_0$ is a single interval of length $<\co \la^{1/8}$, and
$M_0\approx \la^{1/(10\tau)}\ll \la^{1/(8\tau)}$.
Therefore (\ref{BC1_eq1}) follows from Lemma \ref{BE7}.  

To show the second statement we shall apply Lemma \ref{BE8}. The integer $N$ appearing in the lemma is of the size $\co\la^{1/(8\tau)}$.
Since $1\leq\nu_0\leq [\la^{1/(25\tau)}]$ and $M_0\approx \la^{1/(10\tau)}$, the statement follow from Lemma \ref{BE8}.

\end{proof}

Now we shall verify the base case for the inductive construction (contained in Lemma \ref{BC3} and \ref{BC4}) on which our proof
of Theorem \ref{Maintheorem} is based. The inductive step is the content of Proposition \ref{inductive_lemma} in the next subsection. 
As is often the case, the analysis can be split into two separate
parts: an arithmetic part, and a geometric one. The first lemma, together with Lemma \ref{BC1}, will handle the arithmetics (denoted $\mathcal{A}_0$ in Proposition
\ref{inductive_lemma}), and
the second one will take care of the geometry on the first scale (denoted $\mathcal{B}_0$).

We recall the notation (\ref{notaIter}), that is, given $\theta_0,r_0$ we define
$$
(\theta_n,r_n)=\Phi_E^n(\theta_0,r_0).
$$
We also recall that $E\in\mathcal{E}$ is fixed. To show that a certain configuration (more precisely, that condition $\mathcal{B}_n(ii)$ in Proposition \ref{inductive_lemma} holds for but finitely
many $n\geq0$) happens  when $E$ at the edge of an open gap, we will have to slightly move $E$ when such a configuration appears. For this reason we shall use the notation 
\begin{equation}\label{notaIterE}
(\theta_n,r_n(E'))=\Phi_{E'}^n(\theta_0,r_0).
\end{equation} 
to indicate that we are iterating with $\Phi_{E'}$.

\bigskip
First we note that the definition of the set $J_0$, together with Lemma \ref{BE1}, immediately gives the following result (recall the definition of $\N$ in (\ref{function_N})): 
\begin{lem}\label{BC2}
If $\theta_0\in\T$ and $|r_0|\geq \la^{3/4}$, let $N=\N(\theta_0;J_0)$. Then
$$
\begin{aligned}
|r_k|\geq \la^{3/4} \text{ for all } k\in [0,N].
\end{aligned}
$$
\end{lem}
\begin{proof}Since $\theta_k\notin J_0$ for $0\leq k<N$, we get from Lemma \ref{BE1} that $|r_k|\geq \la^{3/4}$ for all $k=0,1,\ldots N$.
\end{proof}
The same statement holds if we iterate with $\Phi_{E'}$ instead of $\Phi_E$, provided that $E'\approx E$.
\begin{lem}\label{BC2.1}
If $\theta_0\in\T$ and $|r_0|\geq \la^{3/4}$, let $N=\N(\theta_0;J_0)$. Then 
$$
|r_k(E')|\geq \la^{3/4} \text{ for all } k\in [0,N] \text{ and all } |E'-E|<1.
$$
\end{lem}
\begin{proof}
The statement follows from the defintion of the set $J_0=J_0(E)$ in (\ref{setI}) and (the proof of) Lemma \ref{BE1}. Indeed, we have
$|\la f(\theta)-E'|\geq |\la f(\theta)-E|-|E'-E|>2\la^{3/4}-|E'-E|$ for all $\theta\notin J_0$.
\end{proof}

To simplify the statement of the inductive assumption (see Proposition \ref{inductive_lemma}) we only want to use the set $I_0$, and not have to mention $J_0$ (and
similarly on higher steps of the induction). We therefore introduce the set
$$
\Theta_0=\T\setminus \bigcup_{m=0}^{\nu_0}(I_0+m\omega),
$$
where $\nu_0$ is defined as above. Together with Lemma \ref{BC1} the arithmetic part of the base case for the induction follows from:  

\begin{lem}[Base case, Part 1]\label{BC3}
If $\theta_0\in\Theta_0$ and $|r_0|\geq \la^{3/4}$, let $N=\N(\theta_0;I_0)$. Then
$$
\begin{aligned}
|r_k(E')|\geq \la^{3/4}  \text{ for all } k\in [0,N] \text{ and } |E'-E|<1.
\end{aligned}
$$
\end{lem}
\begin{proof}We will prove that $\theta_k\notin J_0$ for all $k\in [0,N-1]$. From this the statement follows immediately from Lemma \ref{BC2.1}. 
If $I_0=J_0$ then, by definition, we have  $\theta_k\notin J_0$ for all $k\in [0,N-1]$.
In the case when $I_0=J_0^1\cup(J_0^2-\nu_0\omega)$ (the resonant case) it is clear that $\theta_k\notin J_0^1$ for all $k\in [0,N-1]$. If there 
was a $k\in[0,N-1]$ such that $\theta_k\in J_0^2\subset I_0+\nu_0\omega$, we would either have $\theta_{k-\nu_0}\in I_0$, if $k\geq \nu_0$, or, if $k<\nu_0$,
$\theta_0\in \bigcup_{k=1}^{\nu_0}(I_0+k\omega)\subset \T\setminus \Theta_0$. Both situations contradict our assumptions. 
\end{proof}

We turn to the geometric part of the base case. In fact, it will be quite similar to the geometric part of the inductive step.
The analysis will be divided into three different cases, depending on the set $J_0$ (and thus on the choice of $E$).

\begin{lem}[Base case, Part 2]\label{BC4}
Let 
$$
\vf(\theta)=\pi_2(\Phi^{M_0+K_0}(\theta-M_0\omega,\infty)),
$$
and let
$$
J_1=\{\theta\in I_0: |\vf(\theta)|\leq 2\la^{-3/4}\}.
$$
Then the set $J_1$ is either empty, consists of one single point, or consists of one or two intervals. Moreover,
\begin{equation}\label{BC_dist}
\text{dist}(J_1,\partial I_0)>\la^{-K_0/10},
\end{equation}
and we have the bounds 
\begin{equation}\label{BC_UB}
\|\vp\|_{C^2(J_1)}<  \la^{6K_0}; \text{ and } 
\end{equation}
$$
\min_{\theta\in J_1}\max\{|\vp'(\theta)|,\vp''(\theta)\}>\la^{K_0/2}  \text{ or } 
\min_{\theta\in J_1}\max\{|\vp'(\theta)|,-\vp''(\theta)\}>\la^{K_0/2}.
$$ 
Furthermore, 
\begin{itemize}
\item[(i)] if $J_1$ consists of two intervals, $J_1^1$ and $J_1^2$, then we have $|\vp'(\theta)|>0$ on $J_1$, $\vp'(\theta)$ having opposite
signs on $J_1^1$ and $J_1^2$, and $$\vp(J_1^i)=[-2\la^{-3/4},2\la^{-3/4}]~ (i=1,2);$$
\item[(ii)] if $J_1=[a,b]$, then $\vp(a)=\vp(b)=\pm2\la^{-3/4}$; 
\end{itemize}

Finally, there is an $e_0\geq \la^{-4M_0}$ such that for all $E'\in [E-e_0,E+e_0]$ we have
$$
\{\theta\in I_0: |\vp(\theta,E')|\leq 1.99\la^{-3/4}\}\subset J_{1}
$$
where $\vf(\theta,E')=\pi_2\left(\Phi_{E'}^{M_0+K_0}(\theta-M_0\omega,\infty)\right)$.
\end{lem}
\begin{rem}
The estimate (\ref{BC_dist}) is of course very rough; $J_1$ is actually "well inside" $I_0$. But this is the only bound we need for the inductive step later.
\end{rem}

\begin{proof}[Proof of Lemma \ref{BC4}]
Let $\vp$ and $J_1$ be defined as in the statement of the lemma.
We shall divide the analysis into three cases, depending on the set $J_0$. 
Recall the definitions we made in the beginning of subsection \ref{base_case}. 

We begin with the last statement. Since we have the (very rough) estimate $|\partial_{E'}\vf(\theta,E')|<\la^{3(M_0+K_0+2)}$ for all $\theta$ such that
$|\vf(\theta,E')|<\la^{5/4}$, by Proposition \ref{A_E}, and since $3(M_0+K_0+2)\ll 4M_0$, the statement indeed follows.

The following result shall be used several times in the proof.

\begin{sublem}\label{BC_SL1}
Let $\sigma_j(\theta)=\pi_2(\Phi^j(\theta+m\omega,\infty))$, where $m$ is some integer, and let $k\geq 0$. Assume that $\theta_0\in \T$ is such that
$\theta_j+m\omega\notin J_0$ for $j\in [0,k-1]$. Then $\sigma_{k+1}(\theta_0)$ is close to $\la f(\theta_k+m\omega)-E$, namely
\begin{equation}\label{BC_SL_eq0}
\begin{aligned}
|\sigma_{k+1}(\theta_0)-(\la f(\theta_k+m\omega)-E)|&\leq \la^{-3/4};
\\|\sigma_{k+1}'(\theta_0)-\la f'(\theta_k+m\omega)|&<\la^{-1/4}; ~\text{ and}\\
|\sigma_{k+1}''(\theta_0)-\la f''(\theta_k+m\omega)|&<\la^{1/2}.
\end{aligned}
\end{equation}
Moreover, given any $u_0\in \PR$, we can write $s_{k+1}=\pi_2(\Phi^{k+1}(\theta_0+m\omega,s_0))$ as
\begin{equation}\label{BC_SL_eq1}
u_{k+1}=\sigma_{k+1}(\theta_0)-\frac{h(\theta_0)}{u_0-w(\theta_0)}
\end{equation}
where $h$ and $w$ satisfy
$$
\la^{-4k}\leq h(\theta_0)\leq \la^{-(4/3)k}, ~|h'(\theta_0)|\leq \sqrt{h(\theta_0)}, |h''(\theta_0)|\leq 1;
$$
and 
$$
|w(\theta_0)|<1.1\la^{-3/4}, |w'(\theta_0)|,|w''(\theta_0)|<1.
$$
\end{sublem}
\begin{proof}
If $k=0$ we have $\sigma_{k+1}(\theta_0)=\sigma_1(\theta_0)=\la f(\theta_0+m\omega)-E$, since $\sigma_0=\infty$. Furthermore, 
$s_1=\la f(\theta_0+m\omega)-E-1/s_0=\sigma_1(\theta_0)-1/s_0$, so (\ref{BC_SL_eq1}) holds with $h=1$ and $w=0$.

We now assume that $k>0$. Since 
$\theta_0+m\omega,\ldots,\theta_{k-1}+m\omega\notin J_0$, and $\sigma_0=\infty$, it follows from Lemma \ref{BC2} that we have 
$$
|\sigma_j(\theta_0)|\geq \la^{3/4} \text{ for all } j\in [0,k].
$$ 
Since $\sigma_{k+1}(\theta_0)=\la f(\theta_0+m\omega)-E-1/\sigma_{k}(\theta_0)$, we immediately get the first estimate in (\ref{BC_SL_eq0}). Moreover, the 
other two follow from Proposition \ref{A1}. The formula (\ref{BC_SL_eq1}), as well as the estimates on $h$ and $w$, follow immediately from Proposition \ref{A5}.  
\end{proof}

We now let
\begin{equation}\label{BC4_r}
r_j=r_j(\theta)=\pi_2(\Phi^j(\theta-M_0\omega,\infty)), \quad j\geq 0.
\end{equation}
In the rest of the proof $r_j$ will always be like this.
Note that we, by definition, have $$\vf(\theta)=r_{M_0+K_0}(\theta).$$
Thus, the aim is to control the iterate $r_{M_0+K_0}(\theta)$ for all $\theta\in I_0$ such that $r_{M_0+K_0}(\theta)$ is close to zero.

We have
\begin{equation}\label{BC_eq99}
\|r_{M_0+1}(\cdot)-(\la f(\cdot)-E)\|_{C^2(I_0)}<\la^{1/2},
\end{equation}
which follows directly from Lemma \ref{BC1} and Sublemma \ref{BC_SL1} (with $m=-M_0$ and $k=M_0$).  
This implies, in particular, that 
$$
\|r_{M_0+1}(\cdot)\|_{C^2(I_0)}<\co\la.
$$
Applying Proposition \ref{A_GB} to $r_{M_0+1}$ gives the upper bound on the derivatives in (\ref{BC_UB}) (since we have
$|\vp(\theta)|\ll 1$ on $J_1$).

Before dividing the analysis into the three cases (depending on $J_0$), we note the following. 
In the case when $J_0$ consists of two intervals, we have that $\lambda |f'(\theta)|$ is large when $\la f(\theta)-E$ is close to zero. More precisely,
let
$$
\widehat{J}_0=\left\{\theta\in J_0: |\la f(\theta)-E|\leq \frac{3}{2}\la^{3/4}\right\}.
$$ 
If $J_0$ consists of two intervals, it follows from Lemmas \ref{BE3}
and \ref{BE9} (with $D=\co \la$ and $\delta=2\la^{3/4}$) that $\widehat{J}_0$ consists of two intervals
$\widehat{J}_0^i\subset J_0^i ~(i=1,2)$, $f'(\theta)$ having opposite sign on the two intervals,
\begin{equation}\label{BC_eq2}
\lambda |f'(\theta)|>\co\lambda^{7/8}\quad \text{on } \widehat{J}_0,
\end{equation}
and
\begin{equation}\label{BC_eq2_1}
\la f(\widehat{J}_0^i)-E=[-(3/2)\la^{3/4},(3/2)\la^{3/4}] \quad (i=1,2).
\end{equation}
We also have, by the definition of $\widehat{J}_0$ and $J_0$, that 
\begin{equation}\label{BC_eq2_2}
|\la f(\theta)-E|>(3/2)\la^{3/4} \text{ for all } \theta\in \T\setminus  \widehat{J}_0;
\end{equation}
see Fig. \ref{fig1}.

With this information we can easily verify (\ref{BC_dist}). We let $\widehat{J}_0\subset J_0$ be defined as above. In the case when $J_0$ consists of a single interval
we can have that $\widehat{J}_0$ is empty. Since $\|f\|_{C^1(\T)}<\co$, we have
\begin{equation}\label{BC_eq2_99}
\text{dist}(\widehat{J}_0,\partial J_0)>\co \la^{-1/4} 
\end{equation}
whenever $\widehat{J}_0\neq \emptyset$. In the non-resonant case we have $I_0=J_0$. If $\theta\in I_0\setminus \widehat{J_0}$ it follows from 
(\ref{BC_eq99}) that $|r_{M_0+1}(\theta)|>\la^{3/4}$. Since (\ref{BC1_eq1}) holds, it thus follows from Lemma \ref{BC2} that $|\vf(\theta)|=|r_{M_0+K_0}(\theta)|\geq \la^{3/4}$.
We conclude that $J_1\subset \widehat{J}_0$ (we have, in particular, that $J_1=\emptyset$ if $\widehat{J}_0=\emptyset$), and therefore (\ref{BC_dist}) follows from (\ref{BC_eq2_99}).
In the resonant case we have defined $I_0=J_0^1\cup(J_0^2-\nu_0\omega)$. It is easy to see that we have 
$\text{dist}(\widehat{J}_0^1\cup (\widehat{J}_0^2-\nu_0\omega),\partial I_0)>\co \la^{-1/4} $. Since 
\begin{equation}\label{BC_eq89}
|r_{M_0+K_0}(\theta)|\geq \la^{3/4} \text{ for all } \theta\in I_0\setminus (\widehat{J}_0^1\cup (\widehat{J}_0^2-\nu_0\omega)),
\end{equation} 
as we will shortly show, we can conclude that (\ref{BC_dist}) also holds in this case.

As in the non-resonant case we get $|r_{M_0+1}(\theta)|>\la^{3/4}$ for all
$\theta\in I_0\setminus \widehat{J}_0^1$. Using Lemmas \ref{BC1} and \ref{BC2} we therefore have $|r_{M_0+\nu_0}(\theta)|\geq \la^{3/4}$ for $\theta\in I_0\setminus \widehat{J}_0^1$.
If $\theta \notin I_0\setminus (\widehat{J}_0^2-\nu_0\omega)$ we have $|\la f(\theta+\nu_0\omega)-E|>(3/2)\la^{3/4}$. Combining these facts we conclude that 
$|r_{M_0+\nu_0+1}(\theta)|\geq |\la f(\theta+\nu_0\omega)-E|-|1/r_{M_0+\nu_0}(\theta)|>\la^{3/4}$ for all $\theta\in I_0\setminus (\widehat{J}_0^1\cup (\widehat{J}_0^2-\nu_0\omega))$.
Lemmas \ref{BC1} and \ref{BC2} now imply that (\ref{BC_eq89}) holds.

\bigskip
Now we divide the analysis into three cases.

\bigskip
\emph{Case I}. We first treat the case when $J_0$ consists of two (disjoint) intervals $J_0^1$, $J_0^2$ 
which are non-resonant (case (1) above). In this case we defined $I_0=J_0$. 
By Lemma \ref{BC1} we have
\begin{equation}\label{BC_eq1_1}
J_0\cap \bigcup_{m\in [-M_0,K_0]\setminus\{0\}}(J_0+m\omega)=\emptyset.
\end{equation}
Combining the above estimates on $\la f-E$ with (\ref{BC_eq99}) we get
$$
\begin{aligned}
r_{M_0+1}&(\widehat{J}_0^i)\supset [-1.4\la^{3/4},1.4\la^{3/4}] \quad (i=1,2); \\
|r_{M_0+1}&(\theta)|\geq 1.4\la^{3/4} \text{ for } \theta\in I_0\setminus \widehat{J}_0; ~\text{ and } \\
|r_{M_0+1}&'(\theta)|>\co \la^{7/8}, \quad \theta\in  \widehat{J}_0.
\end{aligned}
$$
Moreover, $r_{M_0+1}'$ have opposite signs on $\widehat{J}_0^1$ and $\widehat{J}_0^2$.

Now we need to control the iterates $r_{M_0+2},\ldots, r_{M_0+K_0}$. This we will do by comparing (shadowing) them with
some other iterates which we can control. Namely, let
\begin{equation}\label{BC_eq5}
s_j=s_j(\theta)=\pi_2(\Phi^j(\theta+\omega,\infty)), \quad j\geq 0.
\end{equation}
We shall compare $r_{M_0+1+j}(\theta)$ with $s_j(\theta)$ (for $\theta\in I_0$). 
Note that we have, from the definition of the $r_j$ in (\ref{BC4_r}),
$$
\Phi^{M_0+1}(\theta-M_0\omega,\infty)=(\theta+\omega,r_{M_0+1}(\theta)),
$$ and thus
$$
r_{M_0+1+j}(\theta)=\pi_2(\Phi^{j}(\theta+\omega,r_{M_0+1}(\theta))),
$$
which motivates the definition of the $s_j$.
Since (\ref{BC_eq1_1}) holds, we can apply Sublemma \ref{BC_SL1} (with $m=1$ and $k=K_0-2$). We then get
$$
|s_{K_0-1}(\theta)|\geq \la^{3/4} \text{ on } I_0; ~\text{ and}
$$
$$
\|s_{K_0-1}\|_{C^2(I_0)}<\co \la,
$$ 
and we can write $\vp$ as  
\begin{equation}\label{BC_eq3}
\vp(\theta)=r_{M_0+K_0}(\theta)=s_{K_0-1}(\theta)-\frac{h(\theta)}{r_{M_0+1}(\theta)-w(\theta)},
\end{equation}
where the functions $h,w$, for all $\theta\in I_0$, satisfy 
\begin{equation}\label{BC_eq4}
0<h(\theta)\leq \la^{-(4/3)(K_0-2)}, ~|h'(\theta)|\leq \sqrt{h(\theta)}, |h''(\theta)|\leq 1 
\end{equation}
and 
\begin{equation}\label{BC_eq6}
|w(\theta)|<1.1\la^{-3/4}, |w'(\theta)|,|w''(\theta)|<1.
\end{equation}
Since the plan is to apply Lemma \ref{L_B1} to each of the two intervals $\widehat{J}_0^i$ $(i=1,2)$, we let 
\begin{equation}\label{BC_eq7}
g(\theta)=r_{M_0+1}(\theta)-w(\theta)
\end{equation}
denote the denominator in (\ref{BC_eq3}).
From the above estimates on the functions $r_{M_0+1}$ and $w$, we immediately get 
$$
g(\widehat{J}_0^i)\supset [-\la^{3/4},\la^{3/4}] (i=1,2), ~\text{ and }~ |g(\theta)|>\la^{3/4} \text{ for all } \theta\in I_0\setminus \widehat{J}_0,
$$
and 
$$
|g'(\theta)|> \co\la^{7/8} ~\text{ for all } \theta\in \widehat{J}_0,
$$
with opposite signs on $\widehat{J}_0^1$ and $\widehat{J}_0^2$.
Now the statement of the lemma (the lower bound on the derivatives, and the statements in case $(i)$) follows from Lemma \ref{L_B1} (applied to each of the intervals $\widehat{J}_0^1$), since 
$\vp$ can be written as in (\ref{BC_eq3}), and since we have the above estimates on $s_{K_0-1},h$ and $g$.

\bigskip
\emph{Case II}. We now treat the case when $J_0$ consists of one single interval $[a,b]$. 
In this case we defined $I_0=J_0$. The analysis will be very similar to the one we just made for the
non-resonant situation. The difference will mainly be that we shall use Lemma \ref{L_B2} instead of
Lemma \ref{L_B1}. 

We shall assume that 
$$
\min_{\theta\in J_0}\max\{|\la f'(\theta)|,\la f''(\theta)\}\geq \co \la
$$
and
$$
\la f(a)-E=\la f(b)-E=2\la^{3/4},
$$
the other case being analogous; recall Lemma \ref{BE3}.
 
By (\ref{BC_eq99}) we have $\|r_{M_0+1}-(\la f-E)\|_{C^2(I_0)}<\la^{1/2}$.
We let $s_j$ be defined as in (\ref{BC_eq5}), and use them, as before, to shadow $r_{M_0+1+j}$. Since (\ref{BC_eq1_1}) holds also in
this case, we get exactly the same estimates on the $s_j$ as in the previous case. Thus, $\vp$ can be written as in 
(\ref{BC_eq3}), and we get the same estimates on $h$ and $w$ as above, i.e., as in (\ref{BC_eq4}) and (\ref{BC_eq6}).
Therefore, if we let $g(\theta)$ be defined as in (\ref{BC_eq7}) we have the following estimates:
$$
\min_{\theta\in J_0}\max\{|g'(\theta)|, g''(\theta)\}\geq \co \la
$$
and
$$
g(a),g(b)> \la^{3/4}.
$$
Applying Lemma \ref{L_B2} gives that the set $J_1$ is empty, consist of a single point, or one or two intervals. Moreover, we get the
estimates in the lemma.

\bigskip
\emph{Case III}. Finally we consider the resonant case. This is the situation when $J_0$ consists of two intervals, $J_0^1$ and $J_0^2$, and
$(J_0^1+\nu_0\omega)\cap J_0^2\neq \emptyset$ for some $\nu_0\in [1,[\la^{1/(25\tau)}]]$. 
Recall that we in this case have defined the (single) interval $I_0$ by $I_0=J_0^1\cup (J_0^2-\nu_0\omega)$.

What we will have to do in this case is to use the shadowing argument twice; first to get from $I_0+\omega$ to
$I_0+(\nu_0+1)\omega$, and then from $I_0+(\nu_0+1)\omega$ to $I_0+K_0\omega$ (we recall that 
$K_0\gg \nu_0$).

Exactly as in Case I we first get, using (\ref{BC_eq2} -- \ref{BC_eq2_2}), the following estimates on $r_{M_0+1}$: 
$$
\begin{aligned}
r_{M_0+1}&(\widehat{J}_0^1)\supset [-1.4\la^{3/4},1.4\la^{3/4}]; \\
\|r_{M_0+1}&(\theta)\|_{C^2(I_0)}<\co \la ~\text{ and} \\
|r_{M_0+1}'&(\theta)|>\co \la^{7/8} ~\text{ for all } \theta\in \widehat{J}_0^1,
\end{aligned}
$$
where $r_{M_0+1}'$ have the same sign on $\widehat{J}_0^1$ as $f'$.
Moreover, since $I_0\cap J_0^2=\emptyset$ (Lemma \ref{BC1}) we also have, using (\ref{BC_eq2_2}),
$$
|r_{M_0+1}(\theta)|>1.4\la^{3/4} \text{ for all } \theta\in I_0\setminus \widehat{J}_0^1.
$$

Now we want to iterate $r_{M_0+1}$ further (for $\theta\in I_0$). The first step is to get control on $r_{M_0+\nu_0+1}$. 
This will be done by shadowing. 
By Lemma \ref{BC1} we have
\begin{equation}\label{BC_eq8}
J_0\cap \bigcup_{m=1}^{\nu_0-1}(I_0+m\omega)=\emptyset.
\end{equation}
We let $s_k$ be defined as in (\ref{BC_eq5}). Since we have (\ref{BC_eq8}), we can apply Sublemma \ref{BC_SL1} (with $m=1$ and $k=\nu_0-1$) and get
$$
\|s_{\nu_0}(\cdot)-(\la f(\cdot+\nu_0\omega)-E)\|_{C^2(I_0)}<\la^{1/2}.
$$
Since $J_0^2\subset I_0+\nu_0\omega$ and $J_0^1\cap (I_0+\nu_0\omega)=\emptyset$, it follows (exactly as above, using (\ref{BC_eq2} -- \ref{BC_eq2_2})) that 
$$
\begin{aligned}
s_{\nu_0}&(\widehat{J}_0^2-\nu_0\omega)\supset [-1.4\la^{3/4},1.4\la^{3/4}]; \\
\|s_{\nu_0}&(\theta)\|_{C^2(I_0)}<\co \la;  \\
|s_{\nu_0}'&(\theta)|>\co \la^{7/8}  \text{ for all } \theta\in \widehat{J}_0^2-\nu_0\omega,
\end{aligned}
$$
where $s'_{\nu_0}$ has the same sign on $\widehat{J}_0^2$ as $f'$, and
$$
|s_{\nu_0}(\theta)|>1.4\la^{3/4} \text{ for } \theta\in I_0\setminus (\widehat{J}_0^2-\nu_0\omega).
$$
Moreover, by the second part of Sublemma \ref{BC_SL1} we can write $r_{M_0+1+\nu_0}$ as
\begin{equation}\label{BC_eq20}
r_{M_0+1+\nu_0}(\theta)=s_{\nu_0}(\theta)- \frac{h_1(\theta)}{r_{M_0+1}(\theta)-w_1(\theta)},
\end{equation}
where the functions $h_1$ and $w_1$ satisfy
\begin{equation}\label{BC_eq10}
\la^{-4(\nu_0-1)}\leq h_1(\theta)\leq \la^{-(4/3)(\nu_0-1)}, |h_1'(\theta)|\leq \sqrt{h_1(\theta)}, |h_1''(\theta)|\leq 1; ~\text{ and}
\end{equation}
\begin{equation}\label{BC_eq11}
|w_1(\theta)|<1.1\la^{-3/4}, |w_1'(\theta)|,|w_1''(\theta)|<1
\end{equation}
on $I_0$.

To derive information about $r_{M_0+K_0}$, we shall again apply shadowing. To do this, we let
$$
t_j=t_j(\theta)=\pi_2(\Phi^j(\theta+(\nu_0+1)\omega,\infty)), \quad j\geq 0.
$$
By Lemma \ref{BC1} it follows that $\theta_j\notin J_0$ for $j\in [\nu_0+1,K_0]$ if $\theta_0\in I_0$. In particular we have
$$
|\la f(\theta+(K_0-1)\omega)-E|>2\la^{3/4}, \theta\in I_0.
$$
Applying Sublemma \ref{BC_SL1} thus gives us 
the estimates 
\begin{equation}\label{BC_eq15}
\begin{aligned}
\la^{3/4}\leq |&t_{K_0-\nu_0-1}(\theta)|<\co \la;  ~ \text{ and } \\
\|&t_{K_0-\nu_0-1}\|_{C^2(I_0)}<\co \la,
\end{aligned}
\end{equation}
as well as the representation
\begin{equation}\label{BC_eq14}
\varphi(\theta)=r_{M_0+K_0}(\theta)=t_{K_0-\nu_0-1}(\theta)+\frac{h_2(\theta)}{r_{M_0+1+\nu_0}(\theta)-w_2(\theta)}, ~\theta\in I_0,
\end{equation}
where we have the estimates
\begin{equation}\label{BC_eq12}
0<h_2(\theta)\leq \la^{-(4/3)(K_0-\nu_0-2)}, |h_2'(\theta)|\leq \sqrt{h_2(\theta)}, |h_2''(\theta)|\leq 1 
\end{equation}
and 
\begin{equation}\label{BC_eq13}
|w_2(\theta)|<1.1\la^{-3/4}, |w_2'(\theta)|,|w_2''(\theta)|<1
\end{equation}
on $I_0$. Focusing on the denominator in (\ref{BC_eq14}), which we shall denote by $\psi(\theta)$ below, we see that, recalling (\ref{BC_eq20}),
$$
\psi(\theta):=r_{M_0+1+\nu_0}(\theta)-w_2(\theta)=(s_{\nu_0}(\theta)-w_2(\theta))- \frac{h_1(\theta)}{r_{M_0+1}(\theta)-w_1(\theta)}.
$$
The above estimates on $s_{\nu_0}$ and $r_{M_0+1}$ (recall that $s_{\nu_0}'$ on and $r_{M_0+1}'$ have opposite signs, at least when $|s_{\nu_0}|,|r_{M_0+1}|<1.4\la^{3/4}$), together with the ones on $w_1,w_2$ and $h_1$, 
show that we are in the position of applying Lemma \ref{L_B5} 
(with $s=s_{\nu_0}-w_2, h=h_1, \delta=\la^{-4\nu_0}$ and $g=r_{M_0+1}-w_1$). We thus get two intervals $I_0',I_0''\subset I_0$, which we also write $I_0'=[a',b'], I_0''=[a'',b'']$, such that either
\begin{itemize}
\item[(a)] $|\psi'(\theta)|>\sqrt{\la}$ on $I_0'\cup I_0''$, $\psi'(\theta)$ having opposite signs on $I_0'$ and $I_0''$, 
$\psi(I_0')=\psi(I_0'')=[-\sqrt{\la},\sqrt{\la}]$, and $|\psi(\theta)|>\sqrt{\la}$ on $I_0\setminus(I_0'\cup I_0'')$; or 

\bigskip
\item[(b)] $\psi(a')=\psi(b')=\la^{2/3}$ and $\psi(a'')=\psi(b'')=-\la^{2/3}$,
$$
\min_{\theta\in I_0''}\max\{|\psi'(\theta)|,-\psi''(\theta)\}> \sqrt{\la} ~\text{ and }~ \min_{\theta\in I_0'}\max\{|\psi'(\theta)|,\psi''(\theta)\}>\sqrt{\la};
$$
\begin{equation}\label{BC_eq16}
\min_{\theta\in I_0'}\psi(\theta)-\max_{\theta\in I_0''}\psi(\theta)>\co\frac{\sqrt{\la^{-4\nu_0}}}{\la^{1/4}}>\la^{-2(\nu_0+1)};
\end{equation}
and $|\psi(\theta)|>\la^{2/3}$ for all $\theta\in I_0\setminus(I_0'\cup I_0'')$;
\end{itemize}
see Fig. \ref{fig3}.
Note that, with the above definition of $\psi$ it follows from (\ref{BC_eq14}) that we can write
$$
\vp(\theta)=t_{K_0-\nu_0-1}(\theta)+\frac{h_2(\theta)}{\psi(\theta)}.
$$
We recall that we have the estimates (\ref{BC_eq15}) and (\ref{BC_eq12}).

If we are in case $(a)$, then the statement of the lemma follows from Lemma \ref{L_B1} ($J_1$ consists of two intervals, and we have situation $(i)$). 
If $(b)$ holds, we first note the following crucial fact, namely that the "gap" in (\ref{BC_eq16}) is relatively big. 
More precisely, if $|\psi(\theta)|\geq\la^{-(4/3)(K_0-\nu_0-2)}$, then we see that $|\vp(\theta)|>1$ (using the estimates on $h_2$ and $t_{K_0-\nu_0-1}$).
Since $K_0\gg \nu_0$, we have $\la^{-2(\nu_0+1)}\gg \la^{-(4/3)(K_0-\nu_0-2)}$. This implies, by noticing the size of the "gap" in (\ref{BC_eq16}), that
we must have that $$\min_{\theta\in I_0'}|\psi(\theta)|> \la^{-(4/3)(K_0-\nu_0-2)} ~\text{ or }~ \min_{\theta\in I_0''}|\psi(\theta)|> \la^{-(4/3)(K_0-\nu_0-2)}$$ (it is possible that both hold).
Since we are looking for $\theta\in I_0$ such that $|\vp(\theta)|\leq 2\la^{-3/4}$, we see that this can happen on at most one of the intervals $I_0', I_0''$. We can therefore 
apply Lemma \ref{L_B2} to this single interval, and thus obtaining the statements in the lemma. This finishes the proof.    
\end{proof}
\end{subsection}


\begin{subsection}{The inductive step}\label{induction}

In this section we shall inductively define sets $I_1,I_2,\ldots$, and choose integers $M_0<M_1<M_2,\ldots$ and $K_0<K_1<K_2<\ldots$, and successively gain control
on longer and longer orbits of "typical" initial conditions $(\theta_0,r_0)$.
We stress that we sometimes will stop the inductive construction after a finite number of steps (depending on the parameter $E$).
As we shall see later, this happens exactly when the cocycle $(\omega,A_E)$ is uniformly hyperbolic (or, equivalently, $E$ is in the resolvent set of the 
Schr\"odinger operator \ref{operator}).

\bigskip
In the formulation of the inductive assumption we will use the following notation: given sets $I_0,I_1,\ldots,I_n\subset \T$ and integers
$M_j,K_j, \nu_j (j=0,\ldots, n)$ we define the sets (recall that $\Theta_0$ was defined in the previous section).
\begin{equation}\label{setsTGS}
\begin{aligned}
\Theta_n&=\T\setminus\left(\left(\bigcup_{j=0}^{n-1}\bigcup_{m=-M_j+1}^{\nu_j}(I_j+m\om)\right)\cup \bigcup_{m=0}^{\nu_n}(I_n+m\omega)\right);\\
G_n&=\bigcup_{j=0}^n\bigcup_{m=1}^{K_j}(I_j+m\om);\\
S_n&=\T\setminus\bigcup_{j=0}^n\bigcup_{m=1}^{M_j}(I_j+m\om).
\end{aligned}
\end{equation}
Furthermore, if $I\subset \T$ is an interval, 
$$I=[a-\ve,a+\ve],$$ 
and $k>0$, then we denote by $kI$ the 
interval 
$$kI=[a-k\ve,a+k\ve].$$ 
That is, the intervals $I$ and $kI$ have the same center, and $kI$ is $k$ times the size of $I$.
Moreover, if $I=\bigcup_{i=1}^m I^i$, where each $I^i$ is an interval, then we use the notation $kI=\bigcup_{i=1}^m kI^i$.

In order to get consistent notation in the inductive assumption below, we let
$$
S_{-1}=\T, ~G_{-1}=\emptyset \quad\text{and}\quad e_{-1}^\pm=K_{-1}=1.
$$
We also recall that the integers $M_0$ and $K_0$, and the set $I_0\subset \T$, where defined (depending on properties of the set $J_0$) in 
section \ref{base_case}. 

\bigskip
We are now ready to state the inductive step. We again stress that $\lambda>0$ is assumed to be large, depending only on $f,\kappa$ and $\tau$.
\begin{prop}[The inductive step]
\label{inductive_lemma}

For any $n\geq 0$ the following holds: 

\medskip
\noindent\underline{Assumptions:} Assume that the integers $K_j, M_j$ and $\nu_j$  $(0\leq j \leq n)$ have been chosen so that
\begin{equation}\label{il_eq1}
K_j\in \left[[\la^{K_{j-1}/(60\tau)}], [\la^{K_{j-1}/(20\tau)}]\right],~ M_j\in [K_j^2,2K_j^2] \text{ and } \nu_j\in [0,K_j/2], \quad j=0,\ldots, n,
\end{equation}
and that the nested sequence of sets $I_0\supset I_1\supset \cdots \supset I_n$ have been constructed,
where each $I_j$ ($j\in [0,n]$) consists of one or two closed intervals, each of length  $<(1/2)\la^{-K_{j-1}/9}$. 
Moreover, assume that the numbers $0<e_n^+\leq e_{n-1}^+\leq \ldots \leq e_0^+\leq 1$, 
$0<e_n^-\leq e_{n-1}^-\leq \ldots \leq e_0^-\leq 1$, where $e_k^\pm\geq \la^{-4M_k}$ ($k=0,\ldots, n$) have been chosen.
Furthermore, assume that the following two conditions, denoted $\mathcal{A}_n$ and $\mathcal{B}_n$, hold:
\begin{itemize}
\item[$\mathcal{A}_n)$]
\begin{itemize}
\item[$(1)$]
$\displaystyle
I_j\cap\bigcup_{0<|m|\leq 2M_j}(I_j+m\omega)=\emptyset \text{ for each } j\in [0,n].
$
\item[$(2)$]
$\displaystyle
I_n+K_n\omega, I_n-M_n\omega \subset \T\setminus\bigcup_{j=0}^{n-1}\bigcup_{m=-2M_j}^{2M_j}(2I_j+m\om).
$

\item[$(3)$]If $\theta_0\in\Theta_{n}$ and $|r_0|\geq \la^{3/4}$, and if $N=\mathcal{N}(\theta_0;I_n)$ then the following hold:

\medskip
\begin{itemize}
\item[$(a)$] for any $k\in [0,N]$ such that $|r_k|\geq \la^{-2}$, we have \\
$
|\rho_0|^{\alpha_0}\cdots |\rho_l|^{\alpha_l}\geq\la^{(2/3+\left(1/12\right)^{n+1})(k+1)}~ 
\text{for any choices of } \alpha_i\in [1,6]
$\\
where $r_0\cdots r_k=\rho_0\cdots \rho_l$;

\medskip
\item[$(b)$] if $T\in [0,N]$
is such that $\theta_T\in S_{n-1}$, then $|r_T|\geq \la^{3/4}$, and for any $l\in [0,L]$ we have \\
$
|\rho_l|^{\alpha_l}\cdots |\rho_L|^{\alpha_L}\geq \la^{(2/3+\left(1/12\right)^{n+1})(L-l+1)}~\text{for any choices of } 
\alpha_i\in [1,6],
$\\
where $r_0\cdots r_T=\rho_0\cdots \rho_L$;

\medskip
\item[$(c)$] if  $|r_k|<\la^{3/4}$ for some  $k\in [0,N]$,  then  $\theta_k\in G_{n-1}$.

\end{itemize}

\item[$(4)$] For all $E'\in [E-e_{n-1}^-,E+e_{n-1}^+]$ we have: if $\theta_0\in\Theta_{n}$ and $|r_0|\geq \la^{3/4}$, 
and if $N=\mathcal{N}(\theta_0;I_n)$, then the statements in (3)(a-c) hold with all the $r_k$ replaced by $r_k(E')$. 
\end{itemize}

\medskip
\item[$\mathcal{B}_n)$]If we let
$$
\vf_n(\theta)=\pi_2\left(\Phi^{M_n+K_n}(\theta-M_n\omega,\infty)\right),
$$
and define the set $J_{n+1}$ by
$$
J_{n+1}= \{\theta\in I_n: |\vp_n(\theta)|\leq 2\la^{-3/4}\},
$$
then $J_{n+1}$ is either empty, consists of one single point, or consists of one or two intervals. Moreover, 
\begin{equation}\label{IS_eq1}
\text{dist}(J_{n+1},\partial I_n)>\la^{-K_n/10} 
\end{equation}
and we have the bounds
\begin{equation}\label{IS_eq2}
\min_{\theta\in J_{n+1}}\max\{|\vp_n'(\theta)|,\vp_n''(\theta)\}>\la^{K_n/2}  \text{ or } 
\min_{\theta\in J_{n+1}}\max\{|\vp_n'(\theta)|,-\vp_n''(\theta)\}>\la^{K_n/2} .
\end{equation}
Furthermore, 
\begin{itemize}
\item[$(i)_n$] if $J_{n+1}$ consists of two intervals, $J_{n+1}^1$ and $J_{n+1}^2$, then we have $|\vp_n'(\theta)|>0$ on $J_{n+1}$, $\vp_n'(\theta)$ having opposite
signs on $J_{n+1}^1$ and $J_{n+1}^2$, and $$\vp_n(J_{n+1}^i)=[-2\la^{-3/4},2\la^{-3/4}]~ (i=1,2);$$
\item[$(ii)_n$] if $J_{n+1}=[a,b]$, then $\vp_n(a)=\vp_n(b)=\pm2\la^{-3/4}$. 
\end{itemize}

\medskip
Finally, for all $E'\in [E-e_{n}^-,E+e_n^+]$ we have
\begin{equation}\label{IS_eq99}
\{\theta\in I_n: |\vp_n(\theta,E')|\leq 1.99\la^{-3/4}\}\subset J_{n+1}
\end{equation}
where $\vf_n(\theta,E')=\pi_2\left(\Phi_{E'}^{M_n+K_n}(\theta-M_n\omega,\infty)\right)$.
\end{itemize}

\medskip

\noindent\underline{Conclusions:} 
\begin{itemize}
\item If the set $J_{n+1}$ (in condition $\mathcal{B}_{n}$) is empty or consists of one single point, then 
$(3)_{n+1}(a)$ holds with $I_{n+1}=\emptyset$. In this case we stop the induction.

\item Otherwise there exists a set $I_{n+1}\subset I_n$, consisting of one or two (non-degenerate) closed intervals, of length $<(1/2)\la^{-K_n/9}$, 
and there are integers $$K_{n+1}\in \left[[\la^{K_n/(60\tau)}],[\la^{K_n/(20\tau)}]\right], ~M_{n+1}\in [K_{n+1}^2,2K_{n+1}^2], \text{ and } \nu_{n+1}\in [0,K_{n+1}/2]$$ and numbers
$\la^{-4M_n}\leq e_{n+1}^\pm\leq e_n^{\pm}$ such that $\mathcal{A}_{n+1}$ 
and $\mathcal{B}_{n+1}$ hold. 

Moreover, if the set $J_{n+1}$ (in $\mathcal{B}_{n}$) consists of a single interval, then $I_{n+1}=J_{n+1}$ and $\nu_{n+1}=0$.

Finally, if the sets $J_{n+1}$ and $J_{n+2}$ (in $\mathcal{B}_{n}$ and $\mathcal{B}_{n+1}$, respectively) both consists of one single interval 
then either we have that the first condition in $(\ref{IS_eq2})_n$ and $(\ref{IS_eq2})_{n+1}$ both hold, or the second condition holds in both cases 
(i.e., $\vf_n$ and $\vf_{n+1}$ are bent in the same direction on $J_{n+1}$ and $J_{n+2}$, respectively). 
Moreover, if the first condition in $(\ref{IS_eq2})_n$ holds, then we can take $e_{n+1}^-=e_n^-$. Furthermore, there is a $0<\widetilde{e}_{n+1}^-<\la^{-K_{n+1}/2}$ such that
the set in $(\ref{IS_eq99})_{n+1}$ is empty for all $E'\in [E-e_n^-,E-\widetilde{e}_{n+1}]$. Similarly, 
if the second condition in $(\ref{IS_eq2})_n$ holds, then we can take $e_{n+1}^+=e_n^+$, and there is a $0<\widetilde{e}_{n+1}^+< \la^{-K_{n+1}/2}$ such that
the set in $(\ref{IS_eq99})_{n+1}$ is empty for all $E'\in [E+\widetilde{e}_{n+1}^+,E+e_n^+]$.

\end{itemize}
\end{prop}

\begin{figure}
\psfrag{a}{$-2\la^{-3/4}$}
\psfrag{b}{$2\la^{-3/4}$}
\psfrag{j1}{$J_{n+1}^1$}
\psfrag{j2}{$J_{n+1}^2$}
\psfrag{f1}{$\vf_{n+1}$}
\psfrag{f2}{$\vf_{n+1}$}
\includegraphics[width=10cm]{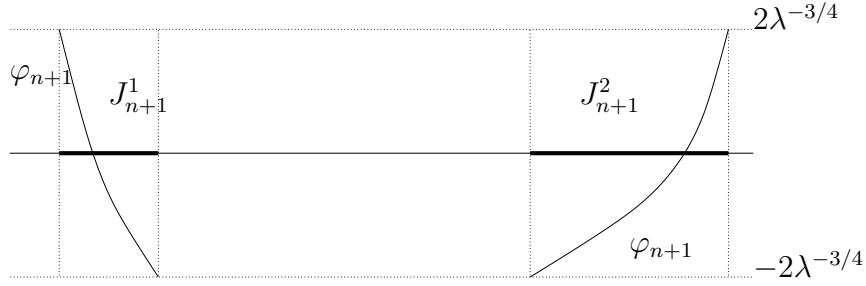}
\caption{Condition $\mathcal{B}(i)_n$}\label{fig10}
\end{figure}

\begin{figure}
\psfrag{a}{$-2\la^{-3/4}$}
\psfrag{b}{$2\la^{-3/4}$}
\psfrag{j}{$J_{n+1}$}
\psfrag{f}{$\vf_{n+1}$}
\includegraphics[width=10cm]{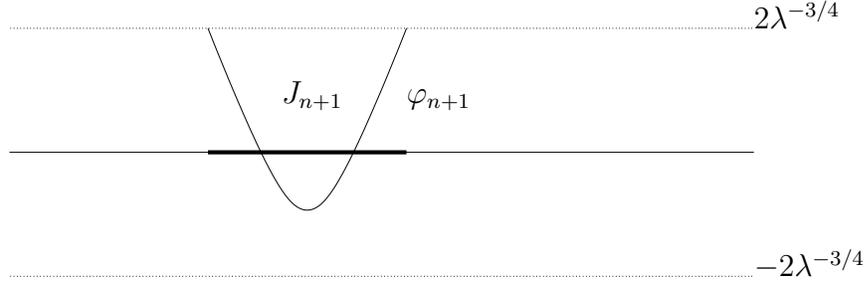}
\caption{Condition $\mathcal{B}(ii)_n$}\label{fig11}
\end{figure}

\begin{rem}a) The proof of this proposition is the heart of our construction. Essentially the proof
consists of two main parts: an arithmetic part ($\mathcal{A}_n$), and a geometric one ($\mathcal{B}_n$). 
These two parts could have been divided into
two separate propositions. However, we think the global picture becomes more transparent having only one inductive statement, 
instead of having to combine two propositions. In the proof we will clearly mark the different steps.

b) We shall see, in the next section, that the induction stops, i.e., $I_n=\emptyset$ for some $n$, exactly when the cocycle $(\omega,A_E)$ is
uniformly hyperbolic.

c) Concerning statement $\mathcal{A}_n$(4), recall the notation (\ref{notaIterE}), and the discussion preceding the notation.

d) The last part of the conclusions will tell us that we are close to an open gap in the case that both $\mathcal{B}_{n}(ii)$ and $\mathcal{B}_{n+1}(ii)$ hold.
\end{rem}

\begin{rem}\label{start_induction}
Note that $\mathcal{A}_0$ follows immediately from Lemmas \ref{BC1} and \ref{BC3} (condition $(2)_0$ is void), and $\mathcal{B}_0$ is verified in Lemma \ref{BC4}. Note also that $M_0,K_0, \nu_0$, 
as defined in section (\ref{base_case}), satisfy condition (\ref{il_eq1}), and the set $I_0$ satisfy the length assumptions. 
Thus the induction starts. It is now possible to go to Section \ref{the_proof} to see how the information gained from the inductive construction
is used to finish the proof of the main theorems.
\end{rem}

\begin{proof}[Proof of Proposition \ref{inductive_lemma}]
We assume that $n\geq0$ is fixed, and that all the inductive assumptions hold. 
The proof consists of several steps. We once again stress that the size of $\lambda$ does not depend on $n$.

Before we begin, it can be instructive to reflect on the sizes of the integers $M_j$ and $K_j$. Let 
$$
Z=[\la^{1/(60\tau)}].
$$
By taking $\lambda$ big, we can make $Z$ as large as we want. 
From (\ref{il_eq1}) we see that
$$
K_0\geq Z,
$$
and
$$
K_j\geq \la^{K_{j-1}/(60\tau)}\geq Z^{K_{j-1}} \text{ for } j=1,2,\ldots,n.
$$
Thus, $K_1\geq Z^Z, K_2\geq Z^{Z^{Z}}$, and so on.

\bigskip
\underline{Part 1 (verification of $\mathcal{A}_{n+1}$)}. As we mentioned above, this is the arithmetic part of the proof.
\bigskip

\noindent\emph{Step 1 (defining the set $I_{n+1}$ and choosing $M_{n+1},K_{n+1}$; verifying $(1-2)_{n+1}$)}. 
We first derive the following result.
\begin{sublem}\label{SL1}
Given any  $I\subset \T$, consisting of one or two intervals of length $\leq \la^{-K_{n-1}/9}$,
there exists an integer $T\in [0,\la^{K_{n-1}/(9\tau)}]$ such that
\begin{equation}\label{SL1_eq1}
I+T\omega\subset \T\setminus \bigcup_{j=0}^{n}\bigcup_{m=-2M_j}^{2M_j}(2I_j+m\omega).
\end{equation}
\end{sublem}
\begin{proof}
Let 
$$
\varepsilon_j=\la^{-K_{j-1}/9}.
$$
We shall consider the worst case, when $I$ and all the $I_j$ ($j=0, 1,\ldots, n$) consists of two intervals, i.e., $I_j=I_j^1\cup I_j^2$ and
$I=I^1\cup I^2$. 
By assumption we have $|I_j^i|\leq \varepsilon_j$, and $|I^i|\leq \ve_n$.

For each $j\in [0,n]$ we have the following.
Let $N_j=(\kappa)^{1/\tau} (1/(4\ve_j))^{1/\tau}$. By Lemma \ref{BE7} we then have
$$
(4I_j^i)\cup (4I_j^i+m\omega)=\emptyset\quad \text{for all } 1\leq |m|\leq N_j ~(i=1,2).
$$
Since $|I^\iota|\leq \ve_n\leq \ve_j$ ($\iota=1,2$), we know that if $(I^\iota+t\omega)\cap 2I_j^i\neq \emptyset$ for some $t\in\mathbb{Z}$, then $I^\iota+t\omega\subset 4I_j^i$.
Hence, for any integers $\nu$ and $m$, there is at most one $p$ in the interval $[\nu,\nu+N_j]$ such that $(I^\iota+p\omega)\cap(2I_j^i+m\omega)$.
Thus it follows that for any integer $\nu$, there are at most $4(4M_j+1)$ integers $p\in [\nu,N_j+\nu]$ such that
$$
(I+p\omega)\cap\bigcup_{m=-2M_j}^{2M_j}(2I_j+m\omega)\neq \emptyset,
$$
and, therefore, at most $4(4M_j+1)(N_n/N_j+1)<40N_n(M_j/N_j)$ integers $p\in [1,N_n]$ such this relation holds. 

Consequently, there
are at most 
$40N_n(M_0/N_0+\ldots +M_n/N_n)$
integers in $[1,N_n]$ such that (\ref{SL1_eq1}) does not hold. Since each $N_j\approx \la^{K_{j-1}/(9\tau)}$, and, by assumption, $M_j\leq 2\la^{K_{j-1}/(10\tau)}$, we see that the
$$
\begin{aligned}
40N_n(M_0/N_0+\ldots +M_n/N_n)<&\co N_n\left(\frac{1}{\la^{1/(90\tau)}}+\frac{1}{\la^{K_0/(90\tau)}}+\ldots+\frac{1}{\la^{K_{n-1}/(90\tau)}}\right) \\& \ll \la^{K_{n-1}/(9\tau)}.
\end{aligned}
$$
Thus we can find (extremely many) $T\in [1,\la^{K_{n-1}/(9\tau)}]$ such that (\ref{SL1_eq1}) holds.
\end{proof}

We can now define $I_{n+1},M_{n+1}$ and $K_{n+1}$, of the right sizes, such that $(1)_{n+1}$, $(2)_{n+1}$ and $(\ref{il_eq1})_{n+1}$ hold. 

(0) First, 
if the set $J_{n+1}$ (in condition $\mathcal{B}_n$) is empty, or contains one single point, then we let $$I_{n+1}=\emptyset.$$ 
If not, that is, $J_{n+1}$ consist of one or two intervals, we have to check for resonances. 
By the derivative assumptions (\ref{IS_eq2}) on the function $\vp(\theta)$ in $\mathcal{B}_n$, and the definition of $J_{n+1}$, we get,
using Lemma \ref{BE9}, that 
\begin{equation}\label{P1S1_eq2}
|J_{n+1}|\leq 4\sqrt{2\la^{-3/4}/\la^{K_n/2}}\ll \la^{-K_n/9}.
\end{equation}
 
We now have two cases:

(1) If
\begin{equation}\label{P1S1_eq1}
J_{n+1}\cap\bigcup_{0<|m|\leq [\la^{K_n/(25\tau)}]}(J_{n+1}+m\omega)=\emptyset,
\end{equation}
we let $$I_{n+1}=J_{n+1}$$
and
$$
\nu_{n+1}=0.
$$ 
In this case we automatically have $I_{n+1}\subset I_n$ since $J_{n+1}\subset I_n$ by definition. 
Moreover, applying Sublemma \ref{SL1} with $I=I_{n+1}+[\la^{K_n/(60\tau)}]\omega$ shows that there exists an
$$K_{n+1}\in [\la^{K_n/(60\tau)},2\la^{K_n/(60\tau)}]$$ such that the first condition in $(2)_{n+1}$ holds. Applying
Sublemma \ref{SL1} with $I=I_{n+1}-2K_{n+1}^2\omega$ shows that there exists an
$M_{n+1}\in [K_{n+1}^2,2K_{n+1}^2]$ such that the second condition in $(2)_{n+1}$ holds. We note that condition $(1)_{n+1}$
follows from (\ref{P1S1_eq1}), since $M_{n+1}\approx \la^{K_n/(30\tau)}\ll \la^{K_n/(25\tau)}$.

\medskip
(2) If (\ref{P1S1_eq1}) does not hold, then $J_{n+1}$ must consists of two intervals (because of the estimate (\ref{P1S1_eq2}) and Lemma \ref{BE7}), 
$J_{n+1}^1$ and $J_{n+1}^2$, and there must be an integer
$$\nu_{n+1}\in [1,\la^{K_n/(25\tau)}]$$ such that $J_{n+1}^1\cap (J_{n+1}^2-\nu_{n+1}\omega)\neq \emptyset$. In this case we let 
$$I_{n+1}=J_{n+1}^1\cup (J_{n+1}^2-\nu_{n+1}\omega),$$ which, recalling (\ref{P1S1_eq2}), is an interval of length $\ll \la^{-K_n/9}$. Using assumption (\ref{IS_eq1}) shows that 
\begin{equation}\label{Step1_In}
I_{n+1}\subset I_n \text{ and } I_{n+1}+\nu_{n+1}\omega\subset I_n.
\end{equation} 
We also note that since $J_{n+1}^i\subset I_{n}$ ($i=1,2$) it follows immediately from $\mathcal{A}_n(1)_n$ that we have the lower bound
\begin{equation}\label{Step1_nu}
\nu_{n+1}>2M_n.
\end{equation}

Applying Sublemma \ref{SL1} with $I=I_{n+1}+[0.5\la^{K_n/(20\tau)}]\omega$
gives us an $$K_{n+1}\in [(1/2)\la^{K_n/(20\tau)},\la^{K_n/(20\tau)}]$$ such that the first condition in $(2)_{n+1}$ holds. Applying 
the Sublemma with $I=I_{n+1}-2K_{n+1}^2\omega$ gives us $M_{n+1}\in [K_{n+1}^2,2K_{n+1}^2]$ such that the second condition in $(2)_{n+1}$ holds. Condition $(1)_{n+1}$ follows
from Lemma \ref{BE7} and the length estimate $|I_{n+1}|\ll \la^{-K_n/9}$ (recall that $M_{n+1}\approx \la^{K_n/(10\tau)}$).

\medskip

From the above definitions we note the we always have the following:
\begin{sublem}\label{SL1.1}
If $\theta_0\in \Theta_{n+1}$ and $I_{n+1}\neq \emptyset$, then $\N(\theta_0;I_{n+1})\leq \N(\theta_0;J_{n+1})$.
\end{sublem}
\begin{proof}In the case when $I_{n+1}=J_{n+1}$ there is nothing to prove. We therefore focus on the resonant case, i.e., the case
when $I_{n+1}=J_{n+1}^1\cup (J_{n+1}^2-\nu_{n+1}\omega)$. Assume that $\theta_k\in J_{n+1}^2$. Then $\theta_{k-\nu_{n+1}}\in I_{n+1}$.
If $k-\nu_{n+1}\geq 0$ we have $\N(\theta_0;I_{n+1})<k$. If $k-\nu_{n+1}<0$, we have that $\theta_0\in \bigcup_{m=1}^{\nu_{n+1}}(I_{n+1}+m\omega)$,
which means that $\theta_0\notin \Theta_{n+1}$ (recall the definition of $\Theta_{n+1}$ in (\ref{setsTGS})), contradicting the assumptions.
\end{proof}

\bigskip
\noindent\emph{Step 2 (verifying $(3)_{n+1}$)}:
To obtain $(3)_{n+1}$ we shall use the assumptions $\mathcal{A}_n$ and $\mathcal{B}_n$. First we derive some facts which
we will need. 

The first three sublemmas will be used repeatedly when we verify $(a), (b)$ and $(c)$ in $(3)_{n+1}$.

\begin{sublem}\label{SL2}
Assume that $|r_0|\geq \la^{3/4}$ and that $\theta_0\in\Theta_{n}$ is such that $\N(\theta_0;I_n)\geq M_n/2$ . 
For every $k\in [0, \N(\theta_0;I_n)+2K_n]$ such that $|r_k|\geq\la^{-2}$
we have the estimate 
$$
|\rho_0|^{\alpha_0}\cdots |\rho_l|^{\alpha_l}\geq 
\la^{(2/3+\left(1/12\right)^{n+2})(k+1)}~\text{for any choices of } \alpha_i\in [1,6],  
$$
where $r_0\cdots r_k=\rho_0\cdots \rho_l$.  
\end{sublem}
\begin{proof}
Let $R=\N(\theta_0;I_n)$. From the assumption on $\theta_0$ we have that $R\geq M_n/2$. If $k\leq R$, then
the statement follows directly from $(3)_n(a)$. We therefore assume that $k\in [R+1,R+2K_n]$. Since not both
$|r_{R}|$ and $|r_{R-1}|$ can be smaller than $\lambda^{-2}$, let us assume that $|r_{R-1}|\geq \la^{-2}$. Then we
can write $r_0\cdots r_{R-1}$ as $\rho_0\cdots \rho_L$, and $r_{R}\cdots r_k$ as $\rho_{L+1}\cdots \rho_l$.  
Since we know that each $|\rho_i|\geq \la^{-2}$ (Lemma \ref{BE4}), we have that 
$$
|\rho_0|^{\alpha_0}\cdots |\rho_L|^{\alpha_L}|\rho_{L+1}|^{\alpha_{L+1}}\cdots |\rho_l|^{\alpha_l}\geq 
\la^{(2/3+\left(1/12\right)^{n+1})R}\cdot \la^{-6\cdot2(k-R+1)}
$$
where we have used $(3)_n(a)$ to estimate the first factor.
An elementary computation, using the fact that $R\geq M_n/2\geq K_n^2/2$ and $0<k-R\leq 2K_n$ shows that
$$(2/3+\left(1/12\right)^{n+1})R -12(k-R+1)>(2/3+\left(1/12\right)^{n+2})(k+1).$$ The crucial fact here is that
$(1/12)^{n+2}M_n/4-13K_n\geq ((1/12)^{n+1}K_n/4-13)K_n\gg 0$
(recall the extreme growth of the sequence $(K_j)_{j\geq 0}$, as discussed in the beginning of the proof).
\end{proof}

\begin{sublem}\label{SL2.5}
Assume that for some $(\theta_0,r_0)$ we have $\theta_{K_n+1}\in \Theta_{n}$ and $|r_{K_n+1}|\geq \la^{3/4}$, and 
there is an integer $T$ such that
$M_n/2\leq T\leq \N(\theta_{K_n+1};I_n)$ and $\theta_T\in S_{n-1}$. Then $|r_T|\geq \la^{3/4}$ and for any $l\in [0,L]$ we have
$$
|\rho_l|^{\alpha_l}\cdots |\rho_L|^{\alpha_L}\geq \la^{(2/3+\left(1/12\right)^{n+2})(L-l+1)}~\text{for any choices of } 
\alpha_i\in [1,6], 
$$
where $r_0\cdots r_T=\rho_0\cdots \rho_L$.
\end{sublem}
\begin{proof}The proof is similar to the previous one. 
Applying $(3)_n(b)$ to the point $(\theta_{K_n+1},r_{K_n+1})$, using the assumption on $T$, we 
get $|r_T|\geq \la^{3/4}$. Thus we can write $r_0\cdots r_T$ as $\rho_0\cdots \rho_L$. Let $\rho_p$ ($p\leq K_n+1$)
be the $\rho$ containing $r_{K_n+1}$. By  $(3)_n(b)$ we have for any $l\in[p+1,L]$
$$
|\rho_l|^{\alpha_l}\cdots |\rho_L|^{\alpha_L}\geq\la^{(2/3+\left(1/12\right)^{n+1})(L-l+1)}~\text{for any choices of } 
\alpha_i\in [1,6]. 
$$
Since $|\rho_0|^{\alpha_0}\cdots |\rho_p|^{\alpha_p}\geq \la^{-12(p+1)}$ if $\alpha_i\in[1,6]$, and since $p\leq K_n+1$ and $L\geq T/2\geq M_n/4$,
the estimate in the statement follows, as in the proof of the previous lemma
\end{proof}

\begin{sublem}\label{SL3}
Assume that $|r_0|\geq \la^{3/4}$ and $\theta_0\in \Theta_{n}$. Then
$$
|r_k|<\la^{3/4}, k\in [0,\N(\theta_0;I_n)+K_n] \quad\Rightarrow \quad \theta_k\in G_n.
$$
\end{sublem}
\begin{proof}For $k\in [0,\N(\theta_0;I_n)]$, it follows from $(3)_n(c)$ that if $|r_k|<\la^{3/4}$, then $\theta_k\in 
G_{n-1}\subset G_n$. If $k\in [\N(\theta_0;I_n)+1,\N(\theta_0;I_n)+K_n]$, then we have, by the definition of $\N(\theta_0;I_n)$, that 
$\theta_k\in (I_n+\omega)\cup \cdots \cup (I+K_n\omega)\subset G_n$. Thus the statement holds.
\end{proof}

\bigskip

In the following two sublemmas we shall make use of the definition of the set $J_{n+1}$ in $\mathcal{B}_n$. Here we are using the idea of "shadowing".  

\begin{sublem}\label{SL5}
If $\theta_0\in (I_n\setminus J_{n+1})-M_n\omega$ (or if $\theta_0\in I_n-M_n\omega$ and the set $J_{n+1}$ consists of one single point) and 
$|r_0|\geq \la^{3/4}$, then $|r_{M_n+K_n+1}|,|r_{M_n+K_n+2}|\geq \la^{3/4}$.
\end{sublem}
\begin{proof}
Using the definition of the function $\vp_n(\theta)$ in $\mathcal{B}_n$, we shall let $(\theta_0,r_0)$ ``shadow'' the point 
$(\theta_0,\infty)$ to obtain control on $r_{M_n+K_n+1}$. 

First we note that we have 
\begin{equation}\label{SL5_eq1}
\begin{aligned}
I_n+K_n\omega, I_n+(K_n+1)\omega, I_n-M_n\omega \subset \T\setminus J_0; \quad \text{and} \\
I_n-(M_n-1)\omega\subset \Theta_{n}. 
\end{aligned}
\end{equation}
Indeed, when $n=0$ the statements follow from Lemma \ref{BC1} (recall that $I_0=J_0$ in the non-resonant case). 
If $n\geq 1$ both statements follow from conditions $(1-2)_n$ in $\mathcal{A}_n$ (recall the definition of $\Theta_{n}$ in (\ref{setsTGS}), and 
the fact that we have $J_0=I_0$ or $J_0\subset I_0\cup (I_0+\nu_0\omega)$, where $0<\nu_0\ll K_0\ll M_0$).

Next we note that for all $|t_0|\geq \la^{3/4}$ we have
\begin{equation}\label{SL5_eq3}
|t_1\cdots t_{M_n+K_n}|\geq \la^{(2/3)(M_n+K_n)} \text { if } |t_{M_n+K_n}|\geq \la^{-2}.
\end{equation}
To see this, we first note that $|t_1|>\la^{3/4}$, since $\theta_0\in \T\setminus J_0$ (Lemma \ref{BE1}). Furthermore, by assumption
$\theta_1\in I_n-(M_n-1)\omega\subset \Theta_n$, and thus $\N(\theta_1;I_n)=M_n-1$ (by $(1)_n$). Hence the statement follows by applying 
Sublemma \ref{SL2} to $(\theta_1,t_1)$.

To continue we let $s_0=\infty$, and define $s_k=\pi_2(\Phi^k(\theta_0,s_0))$. The idea is to compare $r_{M_n+K_n+1}$ with $s_{M_n+K_n+1}$. By the definition of 
$\vp_n(\theta)$ in $\mathcal{B}_n$ we have $s_{M_n+K_n}=\vp_n(\theta_0+M_n\omega)$. Since we have assumed that $\theta_0\notin J_{n+1}-M_n\omega$ 
(or that $J_{n+1}$ consists of one single point), we therefore 
have $|s_{M_n+K_n}|\geq 2\la^{-3/4}$. Moreover, $\theta_{K_n+M_n}\in I_n+K_n\omega\subset \T\setminus J_0$ (by (\ref{SL5_eq1})). Thus we have
\begin{equation}\label{SL5_eq2}
\la^{3/2}>|s_{M_n+K_n+1}|=|\la f(\theta_{K_n+M_n})-1/s_{M_n+K_n}|\geq (3/2)\la^{3/4}.
\end{equation}
Since (\ref{SL5_eq3}) holds, we can apply Lemma \ref{BE20} and conclude that $|s_{M_n+K_n+1}-r_{M_n+K_n+1}|\ll 1$, and hence $|r_{M_n+K_n+1}|>\la^{3/4}$.
Finally, $\theta_{K_n+M_n+1}\in I_n+(K_n+1)\omega\subset \T\setminus J_0$, it follows 
that $|r_{K_n+M_n+2}|\geq \la^{3/4}$.
\end{proof}

We also derive the following result, which will be used in the proof of Theorems \ref{Maintheorem} and \ref{thm2} in the next section.
\begin{sublem}\label{SL6}
If $\theta_0\in J_{n+1}-M_n\omega$ and $|r_0|\geq \la^{3/4}$, then $|r_{M_n+K_n}-\vp_n(\theta_0+M_n)|<\la^{-(4/3)M_n}.$
\end{sublem}
\begin{proof} 
Let $s_k$ be defined as in the proof of the previous sublemma. By definition we have $J_{n+1}\subset I_n$, so $\theta_0\in I_n-M_n\omega$.
Since $s_{M_n+K_n}=\vp_n(\theta_0+M_n)$ and $\theta_0\in J_{n+1}-M_n\omega$,
we have, by the definition of $J_{n+1}$, that $|s_{M_n+K_n}|\leq 2\la^{-3/4}$. 
Therefore we must also have $|s_{M_n+K_n-1}|>\la^{-3/2}$. 
Exactly as in the proof of Sublemma \ref{SL5} we get that for all $|t_0|\geq \la^{3/4}$
$$
|t_1\cdots t_{M_n+K_n-1}|\geq \la^{(2/3)(M_n+K_n-1)} \text { if } |t_{M_n+K_n-1}|\geq \la^{-2}.
$$
Applying Lemma \ref{BE20} to $(\theta_0,r_0)$ and $(\theta_0,s_0)$ shows that
$$
|s_{M_n+K_n}-r_{M_n+K_n}|\leq 2\la^{-((4/3)(M_n+K_n-1)+3/4)}\ll \la^{-(4/3)M_n}.
$$ 
\end{proof}

\bigskip
Now we are ready to show that $(3)_{n+1}$ in $\mathcal{A}_{n+1}$ holds. In fact, this will follow from the next sublemma.

\begin{sublem}\label{SL7}
Assume that $\theta_0\in \Theta_{n}\setminus \bigcup_{m=-M_n+1}^0(I_n+m\omega)$ and $|r_0|\geq \la^{3/4}$. Let $N=\N(\theta_0;J_{n+1})$ if $I_{n+1}\neq \emptyset$; 
let $N=\infty$ if $I_{n+1}= \emptyset$. Then the statements $(a-c)_{n+1}$ in $(3)_{n+1}$ hold. Moreover, $(a)_{n+1}$ holds with $[0,N]$ replaced by $[0,N+K_n]$ (if $N<\infty$). 
\end{sublem}
Indeed, we note, recalling the definition of $\Theta_{n+1}$ in (\ref{setsTGS}), 
that $$\Theta_{n+1}\subset \Theta_{n}\setminus \bigcup_{m=-M_n+1}^0(I_n+m\omega).$$
Moreover, by Sublemma \ref{SL1.1} we 
have $\N(\theta_0;I_{n+1})\leq \N(\theta_0;J_{n+1})$ if $\theta_0\in \Theta_{n+1}$ and $I_{n+1}\neq \emptyset $. Thus we can use the above 
sublemma to conclude that $(3)_{n+1}$ holds.

\begin{proof}[Proof of Sublemma \ref{SL7}]Let $\theta_0, r_0$ and $N$ be as in the statement of the sublemma. By the assumption on $\theta_0$
 we have $N\geq M_n$ (since $J_{n+1}\subset I_n$). 

We first treat the case when $N=\N(\theta_0;I_n)$.
In this case $(3)_n(a - c)$ immediately imply the weaker conditions $(3)_{n+1}(a - c)$ (since $G_n\supset G_{n-1}$ and $S_n\subset S_{n-1}$).  

From now on we assume that $N>\N(\theta_0;I_n)$. Let $0<R_1<R_2<\ldots\leq N$ be the times $k$ in $[0,N]$ when $\theta_k\in I_n$. We then have $R_1\geq M_n$,
and from $(1)_n$ we have $R_{j+1}-R_j>2M_n$. Below we shall prove that
\begin{equation}\label{SL7_eq1}
|r_{R_j-M_n}|\geq \la^{3/4} \text{ for all } j\geq 1.
\end{equation}
From this it follows, using Sublemma \ref{SL5}, that 
\begin{equation}\label{SL7_eq2}
|r_{R_j+K_n+1}|, |r_{R_j+K_n+2}|\geq \la^{3/4} \text{ for all } j\geq 1 \text{ such that } R_j<N.
\end{equation}
Note that we have, by the definition of the $R_j$ and $(1-2)_n$, 
\begin{equation}\label{SL7_eq3}
\theta_{R_j+K_n+i}\in I_n+(K_n+i)\omega\subset \Theta_n~ (i=1,2).
\end{equation}
Applying Sublemmas \ref{SL2} and \ref{SL3} to the points $(\theta_0,r_0)$ and $(\theta_{R_j+K_n+2},r_{R_j+K_n+2}), j\geq 1,$  gives us $(a)_{n+1}$ and $(c)_{n+1}$
(we get control on iterates in $[0, R_1+K_n+1], [R_1+K_n+2,R_2+K_n+1], \ldots$, and thus on the whole interval $[0,N]$; in fact on the interval $[0,N+K_n]$ if $N<\infty$).

To verify $(b)_{n+1}$ we do as follows (the exponents $\alpha_i$ are always assumed to be in $[1,6]$). 
We assume that $T\in [0,N]$ is such that $\theta_T\in S_n$. Note that $G_n\cap S_n=\emptyset$, so
it follows from $(c)_{n+1}$ that $$|r_T|\geq \la^{3/4}.$$ If $T\in [0,R_1]$, the statement about the product in $(b)_{n+1}$ follows from $(b)_{n}$.
We now consider the case when $T\in[R_1,R_2]$. Since $\theta_{R_1}\in I_n\subset S_n\subset S_{n-1}$ (that $I_n\subset S_n$ follows from $(1)_n$), we can use $(b)_n$ to derive
$$
|\rho_l|^{\alpha_l}\cdots |\rho_{L_1}|^{\alpha_{L_1}}\geq \la^{(2/3+(1/12)^{n+1})(L_1-l+1)} \text{ for all } l\in [0,L_1],
$$
where $\rho_0\cdots \rho_{L_1}=r_0\cdots r_{R_1}$.
Applying Sublemma \ref{SL2.5} to $(\theta_{R_1+1},r_{R_1+1})$, recalling (\ref{SL7_eq2}) and (\ref{SL7_eq3}), and noticing that $T-R_1>M_n$ since $\theta_{R_1}\in I_n$,
we get
$$
|\rho_l|^{\alpha_l}\cdots |\rho_{L}|^{\alpha_{L}}\geq \la^{(2/3+(1/12)^{n+2})(L-l+1)} \text{ for all } l\in [L_1+1,L],
$$
where $\rho_{L_1+1}\cdots \rho_{L}=r_{R_1+1}\cdots r_{T}$. Combining the two last estimates yields the estimate in $(b)_{n+1}$.
Now we can proceed analogously, by successively treating $T$ in the intervals $[R_2,R_3], [R_3,R_4], \ldots$. We conclude that $(b)_{n+1}$ holds.

We turn to (\ref{SL7_eq1}). From $(3)_n(c)$ it follows directly (using our assumptions on $(\theta_0,r_0)$) that 
$$
\text{if } |r_k|< \la^{3/4} \text{ for some } k\in [0,R_1], \text{ then } \theta_k\in G_{n-1}. 
$$ 
Combining this with the fact that $(I_n-M_n\omega)\cap G_{n-1}=\emptyset$ (which follows from $(2)_n$; recall that $G_{-1}=\emptyset$) shows that
$|r_{R_1-M_n}|\geq \la^{3/4}$ (recall that $R_1-M_n\geq0$). From Sublemma \ref{SL5} we thus get $|r_{R_1+K_n+2}|\geq \la^{3/4}$. Applying
$(3)_n(c)$ to $(\theta_{R_1+K_n+2},r_{R_1+K_n+2})$, which is possible because of (\ref{SL7_eq3}), gives:
$$
\text{if } |r_k|< \la^{3/4} \text{ for some } k\in [R_1+K_n+2,R_2], \text{ then } \theta_k\in G_{n-1} 
$$ 
(recall that $R_2-R_1>2M_n\gg K_n$). Since, again, $(I_n-M_n\omega)\cap G_{n-1}=\emptyset$ we get $|r_{R_2-M_n}|\geq \la^{3/4}$. Repeating the argument gives us
(\ref{SL7_eq1}) for all $j\geq 1$.
\end{proof} 

\emph{Step 3 (verifying $(4)_{n+1}$)} The verification of $(4)_{n+1}$ is exactly the same as the one for $(3)_{n+1}$, using $(4)_{n}$ instead of $(3)_n$ everywhere.
Indeed, we have $(\ref{IS_eq99})_{n}$ so we can use shadowing as usual. Note that the analysis works for all $E'\in [E-e_n^-,E+e_n^+]$. 
Moreover, as in Sublemma \ref{SL7}, we get that
\begin{sublem}\label{Case4} 
Condition $(4)_{n+1}$ holds with $N=\infty$ if $E'\in [E-e_n^-,E+e_n^+]$ is such that
$$
\{\theta\in I_n: |\vp_n(\theta,E')|\leq 1.99\la^{-3/4}\}=\emptyset.
$$
\end{sublem}

\bigskip
\underline{Part 2 (verification of $\mathcal{B}_{n+1}$)}.

\bigskip
The first thing we note is that since $|\partial_{E'}\vf_{n+1}(\theta,E')|<\la^{3(M_{n+1}+K_{n+1}+2)}$ (by Proposition \ref{A_E}), it clearly follows
that $(\ref{IS_eq99})_{n+1}$ holds for $|E'-E|<\la^{-4M_{n+1}}$. Thus in any case we can always take $e_{n+1}^\pm=\la^{-4M_{n+1}}$.

Now we shall derive a bound on the function $\vf_n$, as defined in $\mathcal{B}_n$.

\begin{sublem}\label{SL10a}
If $J=\{\theta\in I_n: |\vf_n(\theta)|<\la^{5/4}\}$, then
\begin{equation}\label{IS_eq3}
\|\vf_n\|_{C^2(J)}<\la^{6K_n},
\end{equation}
where $\vf_n(\theta)=\pi_2\left(\Phi^{M_n+K_n}(\theta-M_{n}\omega,\infty)\right)$.
\end{sublem}
\begin{proof}Let $r_j(\theta)=\pi_2\left(\Phi^{j}(\theta-M_{n}\omega,\infty)\right)$. Take any $\theta_0\in I_n-M_n\omega\subset \Theta_n$ ($r_0=\infty$). From $(1-2)_n$ it follows that
$\N(\theta_0;I_n)=M_n$. Since $I_n\subset S_{n-1}$, it follows from $(3)_n(b)$ that $|r_{M_n}|\geq \la^{3/4}$, and writing $r_0\cdots r_{M_n}=\rho_0\cdots \rho_L$, we have
$$
|\rho_l|^{\alpha_l}\cdots |\rho_L|^{\alpha_L}\geq \la^{(2/3)(L-l+1)}\text { for all } l\in [1,L]
$$ 
and any choices of $\alpha_i\in [1,6]$. Applying Proposition \ref{A1} therefore gives us 
$$
\|r_{M_n+1}\|_{C^2(I_n)}<\la^2.
$$
To get the statement of the sublemma we now just apply Proposition \ref{A_GB}.
\end{proof}

We shall assume that $I_{n+1}\neq \emptyset$ (if it is the empty set we stop the induction).
As in the proof of Lemma \ref{BC4} (the geometric base case) we will get three different cases, depending on
the set $J_{n+1}$. The following three sublemmas will be used in each of the three cases.


\begin{sublem}\label{SL10}
If we let
$
r_j(\theta)=\pi_2\left(\Phi^{j}(\theta-M_{n+1}\omega,\infty)\right), 
$
then we have the estimates
\begin{equation}\label{SL10_eq2}
|r_{M_{n+1}-M_n}(\theta)|\geq \la^{3/4} \text{ for all } \theta\in I_{n+1} \text{ and } \|r_{M_{n+1}-M_n}\|_{C^2(I_{n+1})}<\la^2.
\end{equation}
\end{sublem}
\begin{proof}We note that, using $\mathcal{A}_{n+1}(1-2)$, we have
\begin{equation}\label{SL10_eq1}
\begin{aligned}
I_{n+1}-(M_{n+1}-1)\omega&\subset \Theta_n\setminus \bigcup_{m=-M_n+1}^0(I_n+m\omega); 
\\ I_{n+1}-(M_n+1)\omega&\subset S_n\cap (\T\setminus G_n) \text{ and } I_{n+1}-M_n\omega\subset \T\setminus G_{n}.
\end{aligned}
\end{equation}
Moreover, $\N(\theta; J_{n+1})\geq M_{n+1}$ for all $\theta\in I_{n+1}-M_{n+1}\omega$ (recall Sublemma \ref{SL1.1}).

Take $\theta_0\in I_{n+1}-M_{n+1}\omega$ (we have $r_0=\infty$). Then $|r_1|\geq \la^{3/4}$, since $I_{n+1}-M_{n+1}\omega\in \T\setminus J_0$. 
We now apply Sublemma \ref{SL7} to the point $(\theta_1,r_1)$. 
Together with the relations in (\ref{SL10_eq1}) we get the following. First,
$$
|r_{M_{n+1}-M_n-1}|,|r_{M_{n+1}-M_n}|\geq \la^{3/4}.
$$
Writing $r_1\cdots r_{M_{n+1}-M_n-1}=\rho_1\cdots \rho_L$ we also get, for any choices of $\alpha_i\in [1,6]$, 
$$
\begin{aligned}
|\rho_l|^{\alpha_l}\cdots |\rho_L|^{\alpha_L}\geq \la^{(2/3)(L-l+1)}\text { for all } l\in [1,L].
\end{aligned}
$$
Thus we can apply Proposition \ref{A1} and obtain $\|r_{M_{n+1}-M_n}\|_{C^2(I_{n+1})}<\la^2$. This shows (\ref{SL10_eq2}).
\end{proof}

The following sublemma can be viewed as an extension of Sublemma \ref{SL6}.
\begin{sublem}\label{SL11}
Assume that $r_0(\theta)$ satisfies $|r_0(\theta)|\geq \la^{3/4}$ on $I_n$ and $\|r_0\|_{C^2(I_{n})}<\la^2$. 
If we let
$r_j(\theta)=\pi_2\left(\Phi^{j}(\theta-M_{n}\omega,r_0(\theta))\right)$, and if $\vp_n(\theta)$ is the function in $\mathcal{B}_n$, then we have
$$
\|r_{M_n+K_n}-\vp_n\|_{C^0(J_{n+1})}<\la^{-(4/3)M_n} \text{ and } \|r_{M_n+K_n}-\vp_n\|_{C^2(J_{n+1})}<1.
$$
Moreover,
$$
|r_{M_n+K_n}(\theta)|\geq 1.9\la^{-3/4} \text{ on } I_n\setminus J_{n+1}. 
$$
\end{sublem}
\begin{proof}The $C^0-$estimate is exactly the same as in Sublemma \ref{SL6}. We turn to the $C^2$-norm. Our estimate is extremely rough (depending on the
way we formulated Proposition \ref{A5}), but it is enough for our purposes.
Let $s_k(\theta)=\pi_2(\Phi^k(\theta-M_n\omega,\infty))$. Then we have, by definition, that $s_{M_n+K_n}(\theta)=\vp_{n}(\theta)$.
Since $|\vp_{n}(\theta)|\leq 2\la^{-3/4}$ on $J_{n+1}$, it follows that $|s_{M_n+K_n-1}(\theta)|>\la^{-2}$ for $\theta\in J_{n+1}$. Thus we can write 
$s_1\cdots s_{M_n+K_n-1}=\rho_1\cdots \rho_L$.
Proceeding as in the proof of Sublemmas \ref{SL5} and \ref{SL6}, we get for any choice of $\alpha_i\in [1,6]$, and for each $\theta\in J_{n+1}$,
$$
|\rho_1|^{\alpha_1}\cdots |\rho_l|^{\alpha_l}\geq \la^{(2/3)l}  \text{ for all } l\in [1,L].
$$
In particular we get $|r_1(\theta)|\geq \la^{3/4}$ on $J_{n+1}$, and thus the estimates on $r_0$ implies that $\|r_1\|_{C^2(J_{n+1})}<\la^2$.

By applying Proposition \ref{A5} we can write $r_{M_n+K_n}(\theta)-\vp_n(\theta)=h(\theta)/(r_1(\theta)-w(\theta))$. By using the estimates
on $h$ and $w$ which we get in Proposition \ref{A5}, in combination with the ones on $r_1$, the desired norm estimate follows.

The last statement is proved by the same approach as the one used in Sublemma \ref{SL5} ($r_{M_n+K_n}$ is very close to $\vp_{n+1}$; if $|\vp(\theta)|\gg \la^{3/2}$ for some $\theta\in I_n$, we must also 
have that $|r_{M_n+K_n}|$ is very large).
\end{proof}

\begin{sublem}\label{SL12}
Let $I=I_{n+1}+K_n\omega$ in the non-resonance case, and $I=I_{n+1}+(\nu_{n+1}+K_n)\omega$ in the resonant case, and let
$$
r_j(\theta)=\pi_2\left(\Phi^{j}(\theta,\infty)\right). 
$$
Then we have
\begin{equation}\label{SL12_eq0}
|r_{K_{n+1}-K_n}(\theta)|\geq \la^{3/4} \text{ for all } \theta\in I \text{ and } \|r_{K_{n+1}-K_n}\|_{C^2(I)}<\la^2.
\end{equation}
Moreover, given any $s_0(\theta)$, we have
$$
\pi_2\left(\Phi^{K_{n+1}-K_n}(\theta,s_0(\theta))\right)=:s_{K_{n+1}-K_n}(\theta)=r_{K_{n+1}-K_n}(\theta)-\frac{h(\theta)}{s_0(\theta)-w(\theta)}
$$
where $h$ and $w$ satisfy the following estimates on $I$:
$$
0<h(\theta)<\la^{-K_{n+1}/2}, ~|h'(\theta)|\leq \sqrt{h(\theta)}, ~|h''(\theta)|<1; \text{ and }
$$
$$
|w(\theta)|<1.1\la^{-3/4} ~ \text{ and } ~|w'(\theta)|, |w''(\theta)|<1.
$$
\end{sublem}
\begin{proof}The proof is similar to the previous two. First we note that in both cases we have (using $\mathcal{A}_{n+1}(1-2)$, and the fact that $I_{n+1}+\nu_{n+1}\omega\subset I_n$;
recall (\ref{Step1_In})) 
\begin{equation}\label{SL12_eq1}
\begin{aligned}
I+\omega&\subset \Theta_n\setminus \bigcup_{m=-M_n+1}^0(I_n+m\omega); 
\\ I+(K_{n+1}-K_n-1)\omega&\subset S_n\cap (\T\setminus G_n) \text{ and } I+(K_{n+1}-K_n)\omega\subset \T\setminus G_{n}.
\end{aligned}
\end{equation}
We also have $\N(\theta;J_{n+1}) \gg K_{n+1}-K_n$ for all $\theta\in I$. Proceeding as in the proof of Sublemma \ref{SL10}, that is,  applying Sublemma \ref{SL7}, we get
(\ref{SL12_eq0}). The second statement follow from Proposition \ref{A5} (we get the estimate on the product by applying $(3)_{n+1}$ to points $(\theta+\omega,r_1(\theta)), \theta\in I$).
\end{proof}

We now divide the further analysis into three cases.

\bigskip
\emph{Case I}: Here we assume that $I_{n+1}$ consists of two intervals. This means, by the definitions in step 1 above, that the set $J_{n+1}$ consists of two
non-resonant intervals, and $I_{n+1}=J_{n+1}$. In this case we thus have  $\mathcal{B}_n(i)$. Since the two intervals are non-resonant, we can focus on each 
of them separately.

Let, as usual, 
$$r_j(\theta)=\pi_2(\Phi^j(\theta-M_{n+1}\omega,\infty)),$$
so that $\vp_{n+1}(\theta)=r_{M_{n+1}+K_{n+1}}(\theta)$. Thus, recalling the definition of the set $J_{n+1}$, we need to 
control $r_{M_{n+1}+K_{n+1}}(\theta)$ on $I_{n+1}$. This will be obtained by taking the following route: 
$$
I_{n+1}-M_{n+1}\omega \to I_{n+1}-M_{n}\omega \to I_{n+1}+K_n\omega \to I_{n+1}+K_{n+1}\omega.
$$
Combining Sublemma \ref{SL10} with Sublemma \ref{SL11}, and recalling that $J_{n+1}=I_{n+1}$ in this case,
give us
\begin{equation}\label{CI_eq3}
\|r_{M_{n+1}+K_{n}}-\vp_n\|_{C^0(I_{n+1})}<\la^{-(4/3)M_n} \text{ and } \|r_{M_{n+1}+K_{n}}-\vp_n\|_{C^2(I_{n+1})}<1
\end{equation}
where $\vp_n$ is as in $\mathcal{B}_n$. Since $\mathcal{B}_n(i)$ holds in the present case, we can thus apply Lemma \ref{BE9}(1) to $\vp_n$ on each of the two intervals in $I_{n+1}$ 
(with $\delta=2\la^{-3/4}$ and
$D=\la^{K_n/2}$), and derive 
the following estimates for 
$r_{M_{n+1}+K_{n}}$:
Let 
$$
I'=\{\theta\in I_{n+1}: |r_{M_{n+1}+K_{n}}(\theta)|\leq 1.5\la^{-3/4}\}.
$$ 
Then $I'$ consists of two intervals, one in each interval in $I_{n+1}$, and 
$$
|r_{M_{n+1}+K_{n}}'(\theta)|>\la^{K_n/4-1} \text{ on } I',
$$
the derivative having opposite signs on the two intervals (recall Fig. \ref{fig10}).
Moreover, using the upper bound (\ref{IS_eq3}), it follows that
\begin{equation}\label{CI_eq2}
\text{dist}(I',\partial I_{n+1})>\la^{-6K_n-1}.
\end{equation}
 
To get control on $r_{M_{n+1}+K_{n+1}}(\theta)$ on $I_{n+1}$ we use shadowing, as usual. Therefore we let
\begin{equation}\label{CI_eq1}
s_j(\theta)=\pi_2(\Phi^j(\theta+K_n\omega,\infty)).
\end{equation}
From Sublemma \ref{SL12} we then get (note that $s$ and $r$ have opposite roles there, compared to the present situation)
$$
r_{M_{n+1}+K_{n+1}}(\theta)=s_{K_{n+1}-K_n}(\theta)-\frac{h(\theta)}{r_{M_{n+1}+K_n}(\theta)-w(\theta)}
$$ 
where $s_{K_{n+1}-K_n}$, $h$ and $w$ satisfy the estimates in Sublemma \ref{SL12}. This implies, using Lemma \ref{L_B1}, that
the estimates in $\mathcal{B}_{n+1}(i)$ hold. In particular, since $|w(\theta)|<1.1\la^{-3/4}$, we must have $J_{n+2}\subset I'$
(it is only when $r_{M_{n+1}+K_n}(\theta)-w(\theta)\approx 0$ that $\theta\in J_{n+1}$; cf Lemma \ref{L_B1}).
Thus $(\ref{IS_eq1})_{n+1}$ follows from (\ref{CI_eq2}).

\bigskip
\emph{Case II}: If $J_{n+1}$ consists of a single interval (and thus $I_{n+1}=J_{n+1}$ by definition), we can proceed exactly as in the previous case. 
The only difference is that $\mathcal{B}_n(ii)$ holds in this case. Since (\ref{CI_eq3}) will hold, the estimates on $\vp_n$ in $\mathcal{B}_n(ii)$ give us 
control on $r_{M_{n+1}+K_n}$. By using shadowing as
above (i.e., letting $s_j$ be as in (\ref{CI_eq1}) and applying Sublemma \ref{SL12} to compare $r_{M_{n+1}+K_{n+1}}$ and $s_{K_{n+1}-K_n}$), the statements in 
 $\mathcal{B}_{n+1}$ follows from Lemma \ref{L_B2}. We note that if we have situation $\mathcal{B}_{n+1}(ii)$, then, by Lemma \ref{L_B2}, 
$\vf_n$ on $J_{n+1}$ and $\vf_{n+1}$ on $J_{n+2}$ are both "bent" in the same direction (see Fig. \ref{fig11}). Note that this is a part of the last statement in the conclusions 
in Proposition \ref{inductive_lemma}. To get the other part we do as follows. We handle the case when $\vf_n, \vf_{n+1}$ are both bent upward. As in Case I we get
\begin{equation}\label{CII_eq99}
\vf_{n+1}(\theta,E')=r_{M_{n+1}+K_{n+1}}(\theta,E')=s_{K_{n+1}-K_n}(\theta,E')-\frac{h(\theta,E')}{r_{M_{n+1}+K_n}(\theta,E')-w(\theta,E')}
\end{equation}
where we for each $E'\in [E-e_{n}^-,E+e_n^{+}]$ and $\theta\in I_{n+1}$ have the estimates 
$$0<h(\theta,E')<\la^{-K_{n+1}/2}, |\partial_{E'} h(\theta,E')|\leq \sqrt{h(\theta,E')};$$ 
$$|w(\theta,E')|<1.1\la^{-3/4}, |\partial_{E'}w(\theta,E')|<1/10;$$
and 
$$|s_{K_{n+1}-K_n}(\theta,E')|\geq \la^{3/4}, |\partial_{E'}s_{K_{n+1}-K_n}(\theta,E')|<2.$$ 
Indeed, this follows as in the proof of Sublemma \ref{SL12}, using $(4)_{n+1}$ instead of only $(3)_{n+1}$, and making use of Proposition \ref{A_E}. 
Moreover, the first part of (\ref{CI_eq3}) holds for all $E'\in [E-e_{n}^-,E+e_n^{+}]$, that is,
$$
\|r_{M_{n+1}+K_n}(\cdot,E')-\vf_n(\cdot,E')\|_{C^0(I_{n+1})}<\la^{-(4/3)M_n}.
$$
A trivial fact which follows immediately from formula (\ref{DerE}) is that 
$$
\partial_{E'} r_{K_{n+1}-K_n}(\theta,E')<-1
$$
(in fact it moves much faster, but we do not need it). This implies, using the representation (\ref{CII_eq99}), that if $|\vf_{n+1}(\theta,E')|<1$ (so the absolute value of the last term in 
(\ref{CII_eq99}) must be $>\sqrt{\la}$)
for some $\theta\in I_{n+1}$
and $E'\in  [E-e_{n}^-,E+e_n^{+}]$, then $\partial_{E'}\vf_{n+1}(\theta,E')<-1/h<-\la^{K_{n+1}/2}$. That is, the bend in $\mathcal(B)_{n+1}(ii)$ moves very fast upwards 
when we change $E'<E$.  

Since we have $(\ref{IS_eq99})_n$ it thus follows that 
$$
\{\theta\in I_n: |\vp_n(\theta,E')|\leq 1.99\la^{-3/4}\}=\emptyset
$$
for all $E\in [E-e_n^-, E-\widetilde{e}_{n+1}^-]$, where $\widetilde{e}_{n+1}^-=\la^{-K_{n+1}/2}$. Thus we let $e_{n+1}^-=e_n^-$ and $e_{n+1}^+=\la^{-4M_{n+1}}$ in this case.

\bigskip
\emph{Case III}: It remains to consider the resonant case. In this case
we defined $$I_{n+1}=J^1_{n+1}\cup (J^2_{n+1}-\nu_{n+1}\omega).$$ We note that $\mathcal{B}_n(ii)$ holds, since $J_{n+1}$ consists of two intervals. 
We also recall (\ref{Step1_In}) and (\ref{Step1_nu}), i.e., that 
$$
\nu_{n+1}>2M_n
$$
and
$$
I_{n+1},I_{n+1}+\nu_{n+1}\omega\subset I_n
$$

To simplify the argument slightly, we shall assume that
\begin{equation}\label{CIII_eq1}
|\varphi_n(\theta)|<\la^{5/4} \text{ on } I_{n+1}\cup (I_{n+1}+\nu_{n+1}\omega),
\end{equation}
where $\vf_n$ is the function in $\mathcal{B}_n$ (recall the definition of $J_{n+1}$ there). (If this would not be the case, we would just divide 
$I_{n+1}\cup (I_{n+1}+\nu_{n+1}\omega)$ into parts where $|\varphi_n(\theta)|$ is smaller than $\la^{5/4}$, and parts where $|\varphi_n(\theta)|$ is larger that 
$\la^{5/4}$. For the latter parts we can iterate one more step to have something smaller that $\la^{5/4}$; to avoid this slightly more cumbersome notation, 
we make the above assumption.)  

If (\ref{CIII_eq1}) holds, it follows from Sublemma \ref{SL10a} that
\begin{equation}\label{CIII_eq3}
\|\vf_n(\theta)\|_{C^2(I_{n+1}\cup (I_{n+1}+\nu_{n+1}\omega)}<\la^{6K_n}.
\end{equation}
Moreover, a version of Sublemma \ref{SL11} can be formulated:   
\begin{sublem}\label{SL11p}
 Let $I=I_{n+1}\cup (I_{n+1}+\nu_{n+1}\omega)$, and assume that  (\ref{CIII_eq1}) holds. Assume that $r_0(\theta)$ satisfies $|r_0(\theta)|\geq \la^{3/4}$ on $I$ 
and $\|r_0\|_{C^2(I)}<\la^2$. 
If we let
$r_j(\theta)=\pi_2\left(\Phi^{j}(\theta-M_{n}\omega,r_0(\theta))\right)$, then we have
$$
\|r_{M_n+K_n}-\vp_n\|_{C^0(I)}<\la^{-(4/3)M_n} \text{ and } \|r_{M_n+K_n}-\vp_n\|_{C^2(I)}<1.
$$
\end{sublem}
\begin{proof}The proof is exactly the same as that of Sublemma \ref{SL11}.
\end{proof}

As in the above cases we first fix
$$
r_j(\theta)=\pi_2(\Phi^j(\theta-M_{n+1}\omega,\infty)).
$$
To control $r_{M_{n+1}+K_{n+1}}$ on $I_{n+1}$ in this case, we will have to go through both $J_{n+1}^1\subset I_{n+1}$ and $J_{n+1}^2\subset I_{n+1}+\nu_{n+1}\omega$, that is, we have to take the rout: 
$$
\begin{aligned}
I_{n+1}-M_{n+1}\omega \to& I_{n+1}-M_{n}\omega \to I_{n+1}+K_n\omega \to I_{n+1}+(\nu_{n+1}-M_n)\omega\to  \\ \to& I_{n+1}+(\nu_{n+1}+K_n)\omega \to I_{n+1}+K_{n+1}\omega
\end{aligned}
$$
(note that $I_{n+1}+(\nu_{n+1}-M_n)\omega\subset I_n-K_n\omega$, so we can use the inductive construction to get to $I_{n+1}+(\nu_{n+1}+K_n)\omega\subset I_n+K_n\omega$).

First, by combining Sublemmas \ref{SL10} and \ref{SL11p}, we get
$$
\|r_{M_{n+1}+K_n}-\vp_n\|_{C^0(I_{n+1})}<\la^{-(4/3)M_n} \text{ and } \|r_{M_{n+1}+K_n}-\vp_n\|_{C^2(I_{n+1})}<1.
$$
From the defintion of $J_{n+1}$ (in $\mathcal{B}_n$) we have, in particular,
$$
r_{M_{n+1}+K_n}(I_{n+1})\supset [-1.99\la^{-3/4},1.99\la^{-3/4}].
$$ 
As in Case I we get the estimate
\begin{equation}\label{CIII_eq4}
|r'_{M_{n+1}+K_n}|>\la^{K_n/4-1} \text{ on } I',
\end{equation}
where $I'=\{\theta\in I_{n+1}: |r_{M_{n+1}+K_n}|\leq 1.5\la^{-3/4}\}$. We note that we have
$I'\subset J_{n+1}^1$.

To get further, we shall, as usual, shadow $r_{M_{n+1}+K_n}$ with $s_{0}=\infty$. Thus we let
$$
s_j(\theta)=\pi_2(\Phi^j(\theta+K_n\omega,\infty)).
$$
This shall give control on $r_{M_{n+1}+\nu_{n+1}+K_n}$ on $I_{n+1}$.
We now state Claim 1, which will be shown below:  
\begin{equation}\label{CIII_eq2}
|s_{\nu_{n+1}-M_n-K_n}(\theta)|\geq \la^{3/4} \text{ for all } \theta\in I_{n+1} \text{ and } \|s_{\nu_{n+1}-M_n-K_n}\|_{C^2(I_{n+1})}<\la^2.
\end{equation}

Combining (\ref{CIII_eq2}) with Sublemma \ref{SL11p} give us
\begin{equation}\label{CIII_eq5}
\begin{aligned}
\|&s_{\nu_{n+1}}(\cdot)-\vp_n(\cdot+\nu_{n+1}\omega)\|_{C^0(I_{n+1})}<\la^{-(4/3)M_n}; \text{ and } \\ \|&s_{\nu_{n+1}}(\cdot)-\vp_n(\cdot+\nu_{n+1}\omega)\|_{C^2(I_{n+1})}<1. 
\end{aligned}
\end{equation}
Thus we have 
$$
s_{\nu_{n+1}}(I_{n+1})\supset [-1.99\la^{-3/4},1.99\la^{-3/4}]
$$ 
and
\begin{equation}\label{CIII_eq8}
|s'_{\nu_{n+1}}|>\la^{K_n/4-1} \text{ on } I'',
\end{equation}
where $I''=\{\theta\in I_{n+1}: |s_{\nu_{n+1}}|\leq 1.5\la^{-3/4}\}$. We also note that $I''\subset J_{n+1}^2-\nu_{n+1}\omega$.

Since $r_{M_{n+1}+K_n}(\theta)\approx \vf_{n+1}$ on $I'\subset J_{n+1}^1$ and $s_{\nu_{n+1}}(\theta)\approx \vp(\theta+\nu_{n+1}\omega)$ on $I''\subset J_{n+1}^2-\nu_{n+1}$, it follows 
from the assumptions in $\mathcal{B}_n$ that $r_{M_{n+1}+K_n}'$ on $I'$ and $s_{\nu_{n+1}}'$ on $I''$ have opposite signs.

To continue, we shall use Claim 2: 
\begin{equation}\label{CIII_eq6}
r_{M_{n+1}+\nu_{n+1}+K_n}(\theta)=s_{\nu_{n+1}}(\theta)-\frac{h_1(\theta)}{r_{M_{n+1}+K_n}(\theta)-w_1(\theta)},
\end{equation}
where $h$ and $w$ satisfy the following bounds on $I_{n+1}$:
$$
\la^{-4\nu_{n+1}}<h_1(\theta)<\la^{-\nu_{n+1}/2}, ~|h_1'(\theta)|\leq \sqrt{h_1(\theta)}, ~|h_1''(\theta)|<1; \text{ and }
$$
$$
|w_1(\theta)|<1.1\la^{-3/4} ~ \text{ and } ~|w_1'(\theta)|, |w_1''(\theta)|<1.
$$
This statement will be proved below.

Finally we use Sublemma \ref{SL12} to get control on $r_{M_{n+1}+K_{n+1}}$. Applying this sublemma we see that we can write
\begin{equation}\label{CIII_eq7}
r_{M_{n+1}+K_{n+1}}(\theta)=t(\theta)- \frac{h_2(\theta)}{r_{M_{n+1}+\nu_{n+1}+K_n}(\theta)-w_2(\theta)}, ~\theta\in I_{n+1},
\end{equation}
where 
$$
|t(\theta)|\geq \la^{3/4} \text{ for all } \theta\in I_{n+1} \text{ and } \|t\|_{C^2(I_{n+1})}<\la^2;
$$
$$
0<h_2(\theta)<\la^{-K_{n+1}/2}, ~|h_2'(\theta)|\leq \sqrt{h_2(\theta)}, ~|h_2''(\theta)|<1; \text{ and }
$$
$$
|w_2(\theta)|<1.1\la^{-3/4} ~ \text{ and } ~|w_2'(\theta)|, |w_2''(\theta)|<1.
$$
Using (\ref{CIII_eq6}), we write the denominator $\psi(\theta):=r_{M_{n+1}+\nu_{n+1}+K_n}(\theta)-w_2(\theta)$ in (\ref{CIII_eq7}) as
$$
r_{M_{n+1}+\nu_{n+1}+K_n}(\theta)-w_2(\theta)=s_{\nu_{n+1}}(\theta)-w_2(\theta)-\frac{h_1(\theta)}{r_{M_{n+1}+K_n}(\theta)-w_1(\theta)},
$$ 
The plan now is to apply Lemma \ref{L_B6} to control  $\psi$ on $I_{n+1}$. To make the presentation more transparent, 
we interpret the above functions in terms of the notation of Lemma \ref{L_B6}. We let 
$$s=s_{\nu_{n+1}}-w_2,~ 
g=r_{M_{n+1}+K_n}-w_1  \text{ and } h=h_1,
$$ 
and $\psi=s-h/g$.
We also let $D=\la^{7K_n}$, $d=\la^{K_n/5}$ and $\delta=\la^{-4\nu_{n+1}}$. From the estimates (\ref{CIII_eq3}), (\ref{CIII_eq4}) and (\ref{CIII_eq8}), and the estimates on $w_1, w_2$,
it follows that $s$ and $g$ satisfy the assumptions in Lemma \ref{L_B6} with $d$ and $D$ as above, and $I=I_{n+1}$. Moreover, by the above estimates on $h_1$ we have 
$$
\delta<h_1<\la^{-\nu_{n+1}/2}\ll D^{-4}
$$
since $\nu_{n+1}>2M_n\geq 2K_n^2$. Thus $h$ also satisfies the assumptions in Lemma \ref{L_B6}. 

Now we apply Lemma \ref{L_B6} to finish the verification of $\mathcal{B}_{n+1}$. From the lemma we get two cases: 

(1) There are two intervals $I^1, I^2 \subset I=I_{n+1}$ such that
$$
|\psi'(\theta)|>d/2=\la^{K_n/5}/2 \text{ on } I^1\cup I^2
$$
$\psi'$ having opposite sign on $I^1$ and $I^2$, $\psi(I^i)=[-1/(3\la),1/(3\la)]$ and $|\psi(\theta)|>1/(3\la)$ outside $I^1\cup I^2$. Since 
$$r_{M_{n+1}+K_{n+1}}=t-h_2/\psi,$$ where $h_2$ and $t$ satisfy the above estimates, the statement of $\mathcal{B}_{n+1}(i)$ follows by applying Lemma \ref{L_B1}.

(2) In the second case, i.e., Lemma \ref{L_B6} (2), we first note that from the fact that the "gap size" $>d\sqrt{\delta}/D^{3/2}$ (which is bigger than $\la^{-3\nu_{n+1}}$), 
it follows that we can have $|\psi(\theta)|\leq \la^{-K_{n+1}/2}$ on at most one of the intervals $I^1,I^2$. Indeed, note that 
$\la^{-K_{n+1}/2}\ll \la^{-3\nu_{n+1}}$ because $K_{n+1}\geq (1/2)\la^{K_n/(20\tau)}$ and $\nu_{n+1}\leq \la^{K_n/(25\tau)}$ (see case (2) in Step 1 of the verification of $\mathcal{A}_{n+1}$).  
Thus, in this case the statement of $\mathcal{B}_{n+1}$ follows by applying Lemma \ref{L_B2} (recall that $|h_2|<\la^{-K_{n+1}/2}$). 

Finally, since the set $J_{n+2}$ must be in $I^1\cup I^2$, which in turn must be in $I'\cup I''$, condition $(\ref{IS_eq2})_{n+1}$ follows (as in Case I).

It remains to verify Claim 1 and 2. To obtain Claim 1, we note that 
$\nu_{n+1}-M_n\gg K_n$, by (\ref{Step1_nu}), and $\N(\theta_0;J_{n+1})\geq \nu_{n+1}-K_n$ for all $\theta_0\in I_{n+1}+K_n\omega$. 
Moreover, by $\mathcal{A}_n$ we have
\begin{equation}\label{CIII_eq20}
I_{n+1}+(K_n+1)\omega\subset I_{n}+(K_n+1)\omega\subset \Theta_n\setminus \bigcup_{m=-M_n+1}^0(I_n+m\omega)
\end{equation}
and, using (\ref{Step1_In}), 
$$
\begin{aligned}
I_{n+1}&+(\nu_{n+1}-M_n-1)\omega\subset I_n-(M_n+1)\omega\subset S_n; \\
I_{n+1}&+(\nu_{n+1}-M_n)\omega\subset I_n-M_n\omega\subset \T\setminus G_n.
\end{aligned}
$$
Thus we can proceed as in the proof of Sublemma \ref{SL10}.

We verify Claim 2. Since (\ref{CIII_eq20}) holds, we can apply Sublemma \ref{SL7} to points $(\theta_1,s_1)$ for $\theta_1\in I_{n+1}+(K_n+1)\omega$.
Thus we get (from the last part of  Sublemma \ref{SL7}), using the fact that $|s_{\nu_{n+1}-1}|>\la^{-2}$ since (\ref{CIII_eq5}) holds, that
$$
|\sigma_1|^{\alpha_1}\cdots |\sigma_l|^{\alpha_l}\geq \la^{(2/3)l}
$$
for all $l\in [1,L]$ and $\alpha_i\in[1,6]$, where $s_1\cdots s_{\nu_{n+1}-1}=\sigma_1\cdots \sigma_L$. Recall that 
$\N(\theta_0;J_{n+1})\geq \nu_{n+1}-K_n$ for all $\theta_0\in I_{n+1}+K_n\omega$.
Now the statement of the claim follows from Proposition \ref{A5}.

\bigskip
This completes the proof of Proposition \ref{inductive_lemma}.
\end{proof}

\end{subsection}

\end{section}
\begin{section}{Proof of Theorems \ref{Maintheorem} and \ref{thm2}}\label{the_proof}
We are finally ready for the proof of Theorems \ref{Maintheorem} and \ref{thm2}. In fact the statements of the theorems will follow easily from the construction in the previous section.
As in that section we assume that $\omega\in \T$ satisfies the 
Diophantine condition (\ref{DC}), for some $\kappa>0$ and $\tau\geq1$, and that the function $f$ is fixed. We also assume that $\lambda$ is sufficiently large, depending only on 
$f,\kappa$ and $\tau$.

Fix any $E\in\mathcal{E}$. We now turn on the inductive machinery from Section \ref{the_induction}, that is, the base case in Lemmas \ref{BC1}, \ref{BC3} and \ref{BC4}, and
the inductive step in Proposition \ref{inductive_lemma} (recall Remark \ref{start_induction} after the statement of the inductive step). 
We recall that if for some $n\geq 0$ the set $J_{n+1}$ (in condition $\mathcal{B}_n$ in Proposition \ref{inductive_lemma}) is empty or 
consists of a single point, then we stop the machine.
As we shall see below, this happens iff the cocycle $(\omega,A_E)$ is uniformly hyperbolic. 

\bigskip
We first deal with the case when the inductive construction stops after a finite number of steps. Thus, we assume that there is an $n\geq 0$ such that 
$I_n\neq \emptyset$ and $I_{n+1}=\emptyset$. By Proposition \ref{inductive_lemma} we then know that condition $(3)_{n+1}(a)$ holds with $N=\infty$.
If we verify that the set $\Theta_{n+1}$, defined as in (\ref{setsTGS}), contains an interval, then it will thus follow from Lemma \ref{BE_UH} that the cocycle $(\omega,A_E)$ is uniformly hyperbolic.
Furthermore, from the growth rate in $(3)_{n+1}(a)$, recalling (\ref{kingman}) and  (\ref{matrix_r}), we will get
$$
\gamma(E)\geq \frac{2\log \lambda}{3}.
$$
We now verify that $\Theta_{n+1}$ contains an interval. Note that  $\Theta_{n+1}$ is defined by removing a finite number of intervals (since each of the sets $I_0, \ldots, I_n$ 
consists of one or two intervals). It is therefore enough to show that the measure of $\Theta_{n+1}$ is positive. From the definition of $I_0$ in Section
\ref{base_case}, and the choice of $K_0,M_0, \nu_0$ there, together with the inductive construction in Proposition \ref{inductive_lemma}, we get (inductively) the following estimates for each
$j\in [0,n]$: 
$$
|I_j|\leq \la^{-K_{j-1}/9},  
$$
where the integers $K_i$ (and $M_i, \nu_i$) satisfy
$$
K_i\in [[\la^{K_{i-1}/(60\tau)}],[\la^{K_{i-1}/(20\tau)}]], ~ M_i\in [K_i^2,2K_i^2] \text{ and } \nu_i\in [0,K_i/2].
$$
Here, as in Proposition \ref{inductive_lemma}, $K_{-1}=1$.
We therefore have 
$$
\begin{aligned}
|\T\setminus \Theta_{n+1}|< \sum_{j=0}^{n+1}(M_j+K_j)|I_j|< 4\sum_{j=0}^{n+1}\la^{K_{j-1}/(10\tau)-K_{j-1}/9} < \co \la^{-1/90}.
\end{aligned}
$$
Thus the set $\Theta_{n+1}$ has measure $\approx 1$.

\bigskip
In the rest of this section we shall assume that the inductive construction goes on forever. In this case we let
$$
\Theta_\infty=\bigcap_{n\geq 0} \Theta_n=\T\setminus \bigcup_{j=0}^{\infty}\bigcup_{m=-M_j+1}^{\nu_j}(I_j+m\omega).
$$
As above, making use of the estimates on $M_j,K_j, \nu_j, |I_j|$ ($j\geq 0$), we get that the measure of the set $\Theta_\infty$ is $\approx 1$.

That we have 
$$
\gamma(E)\geq \frac{2\log \lambda}{3}
$$
also in this case follows easily from $(3)_n(a)$ in Proposition \ref{inductive_lemma}, which we know holds for all $n\geq 0$. 
Indeed, fix any $\theta_0\in \Theta_\infty$ (which is of positive measure) and $|r_0|\geq \la^{3/4}$. Since $\Theta_\infty\subset \Theta_n$ for all $n\geq 0$,
we know that $\N(\theta_0; I_n)\geq M_n$ for each $n\geq 0$. Since $M_n\to\infty$ as $n\to \infty$, it therefore follows from $(3)_n(a)$
that $|r_0\cdots r_k|\geq \la^{(2/3)(k+1)}$ for all $k\geq 0$ such that $r_k\geq \la^{-2}$, and thus 
$$
\max\{|r_0\cdots r_k|,|r_0\cdots r_{k+1}|\}\geq \la^{2/3(k+1)} \text{ for all } k\geq0.
$$ 
We can also state this result as follows:
\begin{equation}\label{Pf_eq0}
\theta_0\in \Theta_\infty, |r_0|\geq \la^{3/4} \Longrightarrow |\rho_0\cdots \rho_l|\geq \la^{(2/3)(k+1)}, l\geq 0,
\end{equation}
where $\rho_0\cdots \rho_l$ is $r_0\cdots r_k$ or $r_0\cdots r_{k+1}$.

\bigskip
We turn to the finer dynamical properties of the map $\Phi=\Phi_E$. Since we have assumed that the induction never stops, we know, in particular, that condition
$\mathcal{B}_n$ holds for all $n\geq 0$ (the set $J_{n+1}$ contains one or two non-degenerate intervals). This means that for each $n\geq 0$ we either have
$\mathcal{B}_n(i)_n$ or $\mathcal{B}_n(ii)_n$. Now the dynamics of $\Phi$ will depend on whether we have $\mathcal{B}_n(i)_n$ for finitely or infinitely many $n$.
We will see that if there are only finitely many such $n$ (i.e., if $\mathcal{B}_n(ii)_n$ holds for all large enough $n$), then $E$ belongs to the edge of an open gap
in the spectrum of the associated Schr\"odinger operator (\ref{operator}) (or, equivalently, the edge of an open interval in the set of parameters $E'$ for which the
cocycle $(\omega,A_{E'})$ is uniformly hyperbolic). 

\bigskip
\emph{Case 1}. We assume here that we have situation $\mathcal{B}_n(i)_n$ for infinitely many $n\geq 0$. The aim is to show that the map $\Phi$ is minimal in this case.
We divide the proof of this fact into a few pieces.

\begin{sublem}\label{Pf_SL1}
Assume that $\theta_0\in \Theta_\infty$ and $|r_0|\geq \la^{3/4}$. If $n\geq 0$ is such that $\mathcal{B}_n(i)_n$ holds, and if $J_{n+1}^1$ is one of the intervals in $J_{n+1}$ such that
$J_{n+1}^1\subset I_{n+1}$ (recall from the definition of $I_{n+1}$ in step 1 of the proof of Proposition \ref{inductive_lemma} that we either have
$I_{n+1}=J_{n+1}$ or $I_{n+1}=J_{n+1}^1\cup(J_{n+1}^2-\nu_{n+1}\omega)$), then the set
$$
\{r_k:\theta_k\in J_{n+1}^1+K_n\omega, k\geq 0\} 
$$  
is $\la^{-M_n}-$dense in the interval $[-2\la^{-3/4},2\la^{-3/4}]$.
\end{sublem}
\begin{proof}
Assume that $(\theta_0,r_0)$ are fixed as in the assumptions of the lemma.
We let 
$$
G_\infty=\bigcup_{k\geq 0}G_k=\bigcup_{j=0}^\infty\bigcup_{m=1}^{K_j}(I_j+m\omega).
$$
Since $\mathcal{A}_p(3)_p(c)$ holds for all $p\geq 1$ (recall that $\N(\theta_0; I_p)\geq M_p$ for all $p$), we conclude that
\begin{equation}\label{Pf_SL1_eq1}
|r_k|<\la^{3/4}, k\geq 0 \Longrightarrow \theta_k\in G_\infty.
\end{equation}

By definition we have $J_{n+1}\subset I_n$.
From $\mathcal{A}_n(2)_n$ it follows that $(I_n-M_n\omega)\cap G_{n-1}=\emptyset$. Moreover, $\mathcal{A}_n(1)_n$ implies that
$$
(I_n-M_n\omega)\cap\bigcup_{m=1}^{K_n}(I_n+m\omega)=\emptyset,
$$
and thus $(I_n-M_n\omega)\cap G_{n}=\emptyset$. Since we have $J_{n+1}^1\subset I_{n+1}$ it also follows from $\mathcal{A}_{n+1}(1)_{n+1}$ that
$$
(J_{n+1}^1-M_n\omega)\cap\bigcup_{m=1}^{K_{n+1}}(I_{n+1}+m\omega)=\emptyset.
$$ 
Combining these facts now give us
$$
(J_{n+1}^1-M_n\omega)\cap G_{n+1}=\emptyset.
$$
From (\ref{Pf_SL1_eq1}) it therefore follows that if $|r_k|<\la^{3/4}$ for some $k\geq 0$ such that $\theta_k\in J_{n+1}^1-M_n\omega$, then we must in fact have
$\theta_k\in \bigcup_{j=n+2}^\infty\bigcup_{m=1}^{K_j}(I_j+m\omega)$. Using the upper bounds on the derivative of $\vp_n$ in $(\ref{IS_eq3})_n$, and recalling the properties of $J_{n+1}$ in 
$\mathcal{B}_n(i)_n$, we see that $|J_{n+1}^1|>\la^{-6K_n}$. From the estimates on the $K_j,I_j$ we thus get
$$
\left|\bigcup_{j=n+2}^\infty\bigcup_{m=1}^{K_j}(I_j+m\omega)\right|<2 K_{n+2}|I_{n+2}|\leq 2\la^{-K_{n+1}(1/9-1/(20\tau))}.
$$
Therefore we conclude that the set $\{\theta_k: |r_k|\geq \la^{3/4}, k\geq0\}$ is  $2\la^{-K_{n+1}(1/9-1/(20\tau))}-$dense in the interval $J_{n+1}^1-M_n\omega$.
Applying Sublemma \ref{SL6} (where the function  $\vp_n$ is as in $\mathcal{B}_n$), and using that $\mathcal{B}_n(ii)_{n}$ holds, i.e., $\vp_n(J_{n+1}^1)=[-2\la^{-3/4},2\la^{-3/4}]$, together
with the upper bound $|\vp'_n(\theta)|<\la^{6K_n}$ on $J_{n+1}$, now gives that
the set $\{r_k:\theta_k\in J_{n+1}^1+K_n\omega, k\geq 0\}$  
is $\left(\la^{6K_n}\cdot2\la^{-K_{n+1}(1/9-1/(20\tau))}+2\la^{-(4/3)M_n}\right)-$dense in the interval $[-2\la^{-3/4},2\la^{-3/4}]$. Since this last expression is much smaller than $\la^{-M_n}$, 
the statement of the lemma holds.
\end{proof}

\begin{sublem}\label{Pf_SL2}
Let $n_p$ $(p\in\mathbb{N})$ be any subsequence of $\mathbb{N}$ and assume that the points $b_p\in I_{n_p}+K_{n_p}\omega$. If 
$b\in \T$ is an accumulation point of $(b_p)_{p\geq 1}$, then $b, b+\omega\in \Theta_\infty$.  
\end{sublem}
\begin{proof}By definition the sets $I_k$ $(k\geq 0)$ are all closed, and thus the sets $\Theta_n$ are open. This causes some (artificial) problems when taking limits.
We therefore take open sets $U_k$ such that $I_k\subset U_k \subset 2I_k$, $k\geq 0$, and such that $\mathcal{A}_n(1)_n$ in Proposition \ref{inductive_lemma}, 
with $I_j$ replaced by $U_j$, holds for all $n\geq 0$. 
This is indeed possible, since each $I_k$ consists of one or two non-degenerate closed intervals.
Let
$$
\Theta'_n=\T\setminus\left(\left(\bigcup_{j=0}^{n-1}\bigcup_{m=-M_j+1}^{\nu_j}(U_j+m\om)\right)\cup \bigcup_{m=1}^{\nu_n}(U_n+m\omega)\right).
$$ 
Then $\Theta_n'$ is closed and (recall the definition of $\Theta_n$ in (\ref{setsTGS})) $\Theta'_n\subset \Theta_n$. Hence the set $\Theta'_\infty=\bigcap_{n\geq0}\Theta'_n$ is closed,
and it contains $\Theta_\infty$.
By $\mathcal{A}_n(2)_n$ 
(and $\mathcal{A}_n(1)_n$ with $I_j$ replaced by $U_j$) it follows that
$$
I_n+K_n\omega, I_n+(K_n+1)\omega \subset \Theta_n' \text{ for all } n\geq 0.
$$
Since $\Theta_0'\supset\Theta_1'\supset \Theta_2'\supset \cdots$, it thus follows that $b,b+\omega\in \Theta'_\infty\subset \Theta_\infty$.
\end{proof}

\begin{sublem}\label{Pf_SL3}
If $\theta_0\in \T\setminus J_0$ and $\theta_1\in \Theta_\infty$, and if $R$ is the vertical segment $$R=\{(\theta_0,r):|r|\leq 2\la^{-3/4}\},$$ then
\begin{equation}\label{Pf_SL3_eq1}
\overline{\bigcup_{k=0}^\infty \Phi^k(R)}=\T\times \widehat{\mathbb{R}}.
\end{equation}
\end{sublem}
\begin{proof}Let $L=\{(\theta_0,r):|r|> 2\la^{-3/4}\}$ (so $L\cup R=\{\theta_0\}\times \PR$). Then $\pi_2(\Phi(L))$ is the interval
$$A=(\la f(\theta_0)-E-\la^{3/4}/2, \la f(\theta_0)-E+\la^{3/4}/2)\subset \widehat{\mathbb{R}}.$$ 
We note that $|A|=\la^{3/4}$. Since $\theta_0\in \T\setminus J_0$ we have that $|r|\geq \la^{3/4}$ for all $r\in A$.

We have assumed that $\theta_1\in \Theta_\infty$. Therefore it follows from (\ref{Pf_eq0}) that
if $r_1\geq \la^{3/4}$, then $|r_1\cdots r_k|\geq \la^{2/3k}$ for all $k\geq 1$ such that $|r_k|\geq \la^{-2}$.

Fix any $t_1\in A$ (in particular we have $|t_1|\geq \la^{3/4}$). The orbit of $(\theta_k,t_k)_{k\geq 1}$ will be used as a reference. As in the proof of Lemma \ref{BE20} we easily get 
\begin{equation}\label{Pf_SL3_eq2}
|\pi_2(\Phi^{k}(\theta_1,A))|\leq \la^{3/4-(4/3)k}\text { for all } k\geq 1 \text{ such that } |t_k|\geq 2\la^{-2}.
\end{equation}
Since the orbit $(\theta_1+j\omega)_{j\geq1}$ is dense in $\T$, and since $|t_{k+1}|\gg \la$ if $|t_k|< 2\la^{-2}$, this result will give (\ref{Pf_SL3_eq1}).
\end{proof}

Now we have the tools needed to prove that the map $\Phi$ is minimal. Let $(n_p)$ be the (infinite) subsequence of $\mathbb{N}$ such that
$\mathcal{B}_{n_p}(i)_{n_p}$ holds for all $p\geq 1$. For each such $n_p$ we choose one of the two intervals, $J_{n_p+1}^1$, in the set $J_{n_p+1}$ such that $J_{n_p+1}^1\subset I_{n_p+1}$.
Since $|J_{n_p+1}^1|\to 0$ as $n\to\infty$, there must be a point $\theta^*\in \T$ and a subsequence of $(n_p)$ (which we also denote by $(n_p)$ for simplicity)
such that $\theta^*$ is accumulated by the intervals $J_{n_p+1}^1+K_{n_p}\omega$ (that is, $J_{n_p+1}^1+K_{n_p}\omega\to \{\theta^*\}$ in the Hausdorff metric). 
Thus it follows from Sublemma \ref{Pf_SL1} (applied for each $n=n_p$) and the definition of $\theta^*$ that 
$$
\theta_0\in \Theta_\infty, |r_0|\geq \la^{3/4} \Longrightarrow \overline{\{(\theta_k,r_k):k\geq 0\}}\supset \{\theta^*\}\times
[-2\la^{-3/4},2\la^{-3/4}].
$$
By Sublemma \ref{Pf_SL2} it follows that $\theta^*, \theta^*+\omega\in \Theta_\infty$. We can therefore apply Sublemma \ref{Pf_SL3} to the segment $\{\theta^*\}\times
[-2\la^{-3/4},2\la^{-3/4}]$ and conclude that we also have
$$
\theta_0\in \Theta_\infty, |r_0|\geq \la^{3/4} \Longrightarrow \overline{\{(\theta_k,r_k):k\geq 0\}}=
\T\times\widehat{\mathbb{R}}.
$$
Combining this information with the estimate (\ref{Pf_eq0}) and Lemma \ref{A_Osc}, shows that there for each $\theta\in \Theta_\infty$ exists a $w(\theta)$ 
(one of the Oseledets' directions) such that
$$
\theta_0\in \Theta_\infty, r_0\neq w(\theta_0) \Longrightarrow \overline{\{(\theta_k,r_k):k\geq 0\}}=
\T\times\widehat{\mathbb{R}}.
$$ 
Furthermore, since the set $\Theta^*=\{\theta\in\T: \theta_k\in \Theta_\infty \text{ for some } j\geq 0\}$ has full measure, we conclude that for all $\theta\in\Theta^*$ there is $w(\theta)\in 
\widehat{\mathbb{R}}$ such that
\begin{equation}\label{Pf_minimal}
\theta_0\in \Theta^*, r_0\neq w(\theta_0) \Longrightarrow \overline{\{(\theta_k,r_k):k\geq 0\}}=
\T\times\widehat{\mathbb{R}}.
\end{equation}
It is clear that the cocycle cannot be uniformly hyperbolic (recall the definition in the introduction; if the cocycle is uniformly hyperbolic, then there 
exists continuous functions $w^{\pm}:\T\to \PR$ (the directions of the subspaces $W^\pm$) such that $\Phi(\theta,w^{\pm}(\theta))=(\theta+\omega,w^{\pm}(\theta+\omega))$).
Since $\gamma(E)>0$ we thus know that $\Phi_E$ has a unique minimal set $M$ (recall the discussion in Remark \ref{rem_proj} after Theorem \ref{thm2}).
From (\ref{Pf_minimal}) it thus follows that we must have $M=\T\times\widehat{\mathbb{R}}$.

\bigskip
\emph{Case 2}. Now we treat the case when there is an $n_0\geq 0$ such that we have $\mathcal{B}_n(ii)_n$ for all $n\geq n_0$. When this holds, $E$ must in fact
be a boundary point of an open gap in the spectrum, that is, there exists some $\ve>0$ such that $(\omega,A_{E'})$ is uniformly hyperbolic for all $E'\in (E-\ve,E)$
or for all $E'\in (E,E+\ve)$. In \cite{Bj2} we had a situation when $\mathcal{B}_n(ii)_n$ holds for all $n\geq 0$; in that case $E$ was the lowest energy in the spectrum, i.e.,
$(\omega,A_{E'})$ is uniformly hyperbolic for all $E'<E$.

The first thing we shall show is that there exists a phase $\theta^*\in \T$ and a vector $u\in l^2(\Z)$, exponentially decaying at $\pm\infty$, such that 
\begin{equation}\label{Pf_op}
H_{\theta^*}u=Eu
\end{equation}
where $H_{\theta}$ is defined as in (\ref{operator}). 

The $\theta^*$ is defined in the following way. 
Since $I_{n+1}\subset I_n$ and $|I_n|<\la^{K_{n-1}/9}$ for all $n\geq 0$ (we have defined $K_{-1}=1$), and 
since $I_{n}=J_n$ consists of one single interval for all $n\geq n_0+1$, it follows that 
$$
\bigcap_{n\geq0} I_n=\{\theta^*\}
$$
for some $\theta^*\in\T$. 

We shall now construct the vector $u$. In the construction we will need the following result:   
\begin{sublem}\label{Pf_SL6}
There exists a $w\in \PR$ such that if $\theta_0=\theta^*$ and $r_0\neq w$, then there is a constant $C>0$, depending on $r_0$, such that
$$
|\rho_0\cdots \rho_j|\geq C\la^{2j/3} \text{ for all } j\geq 0.
$$ 
Furthermore,
$$
|r_{K_n}|\geq \la^{3/4}/2 \text{ for all sufficiently large } n \text{ (depending on $r_0$ and $n_0$)}.
$$
\end{sublem}
\begin{proof}
Since $J_{n}$ consists of one single interval for all $n\geq n_0+1$, we have, from Proposition \ref{inductive_lemma}, that $\nu_n=0$ for all $n\geq n_0+1$. Thus 
(recalling the definitions in (\ref{setsTGS}))
$$
\begin{aligned}
\Theta_\infty=\bigcap_{n\geq 0} \Theta_n=&\T\setminus \bigcup_{j=0}^{\infty}\bigcup_{m=-M_j+1}^{\nu_j}(I_j+m\omega) =\\=& 
\T\setminus \left(\left(\bigcup_{j=0}^{n_0}\bigcup_{m=-M_j+1}^{\nu_j}(I_j+m\omega)\right)\cup \left(\bigcup_{j=n_0+1}^{\infty}\bigcup_{m=-M_j+1}^{0}(I_j+m\omega)\right)\right).
\end{aligned}
$$
From $\mathcal{A}_n(1)$, which holds for all $n\geq0$, it follows that  $\theta^*+K_{n_0+1}\omega$ does not belong to the last of the two unions above. Moreover, from  $\mathcal{A}_{n_0+1}(2)$ it follows
that $\theta^*+K_{n_0+1}\omega$ does not belong to the first union. We conclude that
\begin{equation}\label{Pf_eq21}
\theta^*+K_{n_0+1}\omega\in \Theta_\infty.
\end{equation}
We let $\vartheta_0=\theta^*+K_{n_0+1}\omega$ and take any $|s_0|\geq \la^{3/4}$. From (\ref{Pf_eq0}) we get the estimate
$$
|\sigma_0\cdots \sigma_l|\geq \la^{(2/3){l+1}} \text{ for all } l\geq 0.
$$
We also claim that 
$$
|s_{K_n-{K_{n_0+1}}}|\geq \la^{3/4} ~ \text{ for all } n\geq n_0+2.
$$
From these two properties the statements in the sublemma follows by an application of Lemma \ref{A_Osc}.
To prove the claim we proceed as follows. For each $n\geq 0$ we can apply $\mathcal{A}_n(3)$ to the point $(\vartheta_0,s_0)$. Note that
$\N(\vartheta_0; I_n)>2M_n-K_{n_0+1}$ for all $n\geq n_0+1$. Since 
$$
I_{n}+K_{n}\omega\subset \T\setminus G_{n-1},
$$
it follows from $\mathcal{A}_n(3)(c)$ that $|s_{K_n-{K_{n_0+1}}}|\geq \la^{3/4}$ for all $n\geq n_0+2$.
\end{proof}

The plan now is to show that there exists an $r^*\in \PR$ such that letting $u_1/u_0=r^*$, the solution $(u_j)_{j\in \Z}$ to (\ref{Pf_op}) is exponentially decaying at $\pm\infty$. 
We will use the following sublemma to define $r^*$. 
\begin{sublem}\label{Pf_SL4}
For all $n\geq 0$ the following holds: If $\theta_0\in I_{n}-M_n\omega$ and $|r_0|,|s_0|\geq\la^{3/4}$, then
$$
|s_{M_n}-r_{M_n}|\leq 2\la^{-4M_n/3}.
$$
Moreover, if $n\geq 1$ we have $|r_{M_{n}-M_{n-1}}|\geq \la^{3/4}$.
\end{sublem}
\begin{proof}We have $I_n-M_n\omega \subset \T\setminus J_0$, $I_n-(M_n-1)\omega\subset \Theta_n$ and $(I_n-\omega)\cap G_{n-1}=\emptyset$. We also have $\N(\theta_0;I_n)=M_n$.
This gives us that $|r_1|,|s_1|\geq\la^{3/4}$ and $|r_1-s_1|=|1/s_1-1/r_1|\leq 2\la^{-3/4}$. Applying $\mathcal{A}_n(3)$ to $(\theta_1,r_1)$ and  $(\theta_1,s_1)$ gives
$|r_{M_n-1}|,|s_{M_n-1}|\geq \la^{3/4}$ (from $(3)_n(c)$) and $|r_1\cdots r_{M_n-1}|, |s_1\cdots s_{M_n-1}|\geq \la^{(2/3)(M_n-1)}$. Thus it follows from formula (\ref{contrF}) that
$|s_{M_n}-r_{M_n}|\leq 2\la^{-4M_n/3}$.
The last statement follows from $(3)_n(c)$ since $(I_{n-1}-M_{n-1}\omega)\cap G_{n-1}=\emptyset$.
\end{proof}
Now we have all the things needed to construct $r^*$. Let 
$$
A_n=\{(\theta^*-M_n\omega,r): |r|\geq \la^{3/4}\}
$$
and 
$$
B_n=\pi_2(\Phi^{M_n}(A_n)).
$$
From Sublemma \ref{Pf_SL4} it follows that $B_{n+1}\subset B_n$ for all $n\geq 0$ and $|B_n|\to 0$ as $n\to\infty$. We conclude that
$$
\bigcap_{n\geq0} B_n=\{r^*\}
$$
for some $r^*\in\PR$. Letting
$$
(\theta^*_j,r^*_j)=\Phi(\theta^*,r^*),~ j\in \Z,
$$
it follows from the construction that
$$
(\theta^*_{-M_n},r^*_{-M_n})\in (I_{n+1}-M_n\omega)\times \PR\setminus (-\la^{3/4},\la^{3/4}) \text { for all } n\geq 0.
$$
Applying $\mathcal{A}_n(3)(b)$ to these points (recalling that $I_{n}-M_n\omega\subset \Theta_n$ and $I_n\subset S_{n-1}$) gives $|r^*|\geq \la^{3/4}$ and
\begin{equation}\label{Pf_eq20}
|r_{-{j+1}}^*\cdots r_0^*|\geq \la^{j/3} \text{ for all } j\geq 0.
\end{equation}
Moreover, since $I_n=J_n$ for all $n\geq n_0+1$, it follows from Sublemma \ref{SL6}, together with the definition of the set $J_n$, that
\begin{equation}\label{Pf_eq30}
(\theta^*_{K_n},r^*_{K_n})\in (I_{n+1}+K_n\omega)\times [-1,1] \text{ for all } n\geq n_0+1.
\end{equation}

If we take $(u_0,u_1)\in \R^2\setminus\{0\}$ such that $u_1/u_0=r^*$, and notice that $r_{-j+1}^*\cdots r_0^*=u_1/u_{-j+1}$, it follows immediately from (\ref{Pf_eq20}) that
$(u_j)_{j\in \Z}$ is exponentially decaying at $-\infty$. Moreover, since (\ref{Pf_eq30}) holds, we conclude that we in fact must have $r^*=w$, where $w\in\PR$ is as in Sublemma
\ref{Pf_SL6}. This shows that $u_j$ decays exponentially fast at $+\infty$.

\bigskip
It remains to show that $E$ is the endpoint of an open gap. To do this we shall use the last part of the conclusions in Proposition \ref{inductive_lemma} (the "Finally" part). 
Since $\mathcal{B}_n(ii)_n$ holds for all $n\geq n_0$, we thus get the functions $\vf_n$ on $I_{n+1}$ ($n\geq n_0$) are all "bent upward" or all "bent downward". We treat the first case. Inductively we
then get that $e_n^-=e_{n_0}^-$ for all $n\geq n_0$. Moreover, there are $\widetilde{e}_{n}^-\to 0$ such that
condition $(4)_n$ in Proposition \ref{inductive_lemma} holds with $N=\infty$ for all $E'\in [E-e_{n_0}^-,E-\widetilde{e}_{n}^-]$ ($n>n_0+1$). This follows from Sublemma \ref{Case4}, which holds for each $n$. 
Thus, as in the case when the induction
stops after a finite number of steps which was treated above, it follows that the cocycle $(\omega,A_{E'})$ is uniformly hyperbolic for all  $E'\in [E-e_{n_0}^-,E-\widetilde{e}_{n}^-]$ and all $n>n_0$.
Since  $\widetilde{e}_{n}^-\to 0$ as $n\to \infty$ we conclude that  $(\omega,A_{E'})$ is uniformly hyperbolic for all  $E'\in [E-e_{n_0}^-,E)$, which is what we wanted to show.
\end{section}

\appendix
\section{}

This appendix contains various formulae which will be needed to keep control on the iterates of the map $\Phi$. 
We shall also establish certain derivate estimates under various assumptions on the iterates.
The main results, which are used frequently in Section \ref{the_induction}, are Propositions \ref{A_GB}, \ref{A1} and \ref{A5}. 

\bigskip
Assume that $r_0=r_0(\theta)$ is a given $C^2$-function. We define, in accordance with (\ref{notaIter}),
$$
(\theta_k,r_k)=(\theta_k,r_k(\theta))=\Phi^{k}(\theta,r_0(\theta)),\quad k\in \mathbb{Z}.
$$  
We shall assume that $\lambda$ is large, depending only on $f$. We also assume that 
$E\in\mathcal{E}$ and $\omega$ are fixed. We shall only consider the dependence of the iterates $r_k=r_k(\theta)$ 
on $\theta$. 
For easier notation, let 
$$
v(\theta)=\lambda f(\theta)-E.
$$ 
Then $\Phi$ can be written
$$
\Phi(\theta,r)=(\theta+\omega,v(\theta)-1/r).
$$
Since $E\in\mathcal{E}$ (recall the definition of $\mathcal{E}$ in (\ref{setE})) we have the bound
\begin{equation}\label{A_v}
\|v\|_{C^2(\T)}<\co \lambda.
\end{equation}

In this section we will denote the supremum norm over $\T$ by $\|\cdot\|$, i.e., $\|\phi\|=\sup_{\theta\in\T}|\phi(\theta)|$.
\begin{subsection}{Formulae} 

\emph{First derivative:} We first derive a formula for $r_k'(\theta)$ ($k>0$), given that $r_0=r_0(\theta)=\infty$.
By definition we have 
$$
r_1(\theta)=v(\theta), \quad \text{and} \quad r_2(\theta)=v(\theta_1)-1/r_1(\theta).
$$
Thus
$$
r'_2(\theta)=v'(\theta)+r_1'(\theta)/r_1(\theta)^2.
$$
Continuing inductively we get
\begin{equation}\label{firstDerF}
r'_k(\theta)=v'(\theta_{k-1})+\sum_{j=1}^{k-1}\frac{v'(\theta_{j-1})}{r_j(\theta)^2\cdots r_{k-1}(\theta)^2},\quad k>0.
\end{equation}
Below we shall often use this formula without explicitly specifying at which points we evaluate
the function $v'(\theta)$, that is, we just write it as
$$
r'_k=v'+\frac{v'}{r_{k-1}^2}+\ldots + \frac{v'}{r_1^2\cdots r_{k-1}^2}.
$$

\emph{Contraction:} Let $(\theta_0,r_0)$ and $(\theta_0,s_0)$ be two given points (over the same fiber $\theta_0\in \T$). 
We want to compare $r_k$ and $s_k$. First we see that
$$
r_1-s_1=-1/r_0+1/s_0=\frac{r_0-s_0}{r_0s_0}.
$$
Inductively we get
\begin{equation}\label{contrF}
r_k-s_k=\frac{r_0-s_0}{(s_0\cdots s_{k-1})(r_0\cdots r_{k-1})},\quad k\geq 1.
\end{equation}

\emph{Shadowing:} We shall use a version of the above formula for contraction when we want to derive information about iterates $s_k$, knowing
what happens with the iterates $r_k$ (in the same fiber). That is, given $(\theta_0,r_0)$ and $(\theta_0,s_0)$, we seek to express 
$s_{k}$ ($k\geq 1$) in terms of $s_0$ and $r_1,\ldots,r_{k}$.

Assume that $s_0$ is given, and let $r_0=\infty$. As above we have
$$
r_1-s_1=\frac{1}{s_0} \quad\Rightarrow \quad s_1=r_1-\frac{1}{s_0};
$$ 
$$
r_2-s_2=\frac{r_1-s_1}{r_1s_1}=\frac{1/s_0}{r_1(r_1-1/s_0)}=\frac{1}{r_1^2(s_0-1/r_1)};
$$
$$
r_3-s_3=\frac{r_2-s_2}{r_2s_2}=\frac{r_2-s_2}{r_2(r_2-(r_2-s_2))}=\frac{1}{r_2^2(r_2-s_2)^{-1}-r_2}=\frac{1}{r_2^2r_1^2(s_0-1/r_1-1/(r_1^2r_2))}.
$$
Continuing inductively we obtain the expression
\begin{equation}\label{A_form1}
r_{k+1}-s_{k+1}=\left(r_{1}^2\cdots r_{k}^2\left(s_0-\frac{1}{r_1}-\frac{1}{r_1^2r_2}-\frac{1}{r_1^2r_2^2r_3}-\ldots 
- \frac{1}{r_1^2\cdots r_{k-1}^2r_{k}}\right)\right)^{-1}.
\end{equation}

\bigskip
We will need to control derivatives (w.r.t. $\theta$) of the terms in the RHS of the previous expression. Therefore we note that we have the 
following general estimates (which we obtain by simple differentiation). 
Given any $C^2$-functions $g_1(\theta),\ldots, g_k(\theta)$, we have
\begin{equation}\label{prod_form1}
\begin{aligned}
\left|\frac{\partial}{\partial\theta}\left(\frac{1}{g_1^2\cdots g_k^2}\right)\right|&\leq 
\frac{2k}{g_1^2\cdots g_k^2}\cdot\max_{1\leq j\leq k}\left\{\frac{|g_j'|}{|g_j|}\right\}\\ 
\left|\frac{\partial}{\partial\theta}\left(\frac{1}{g_1^2\cdots g_{k-1}^2g_k}\right)\right|&\leq 
\frac{2k}{g_1^2\cdots g_{k-1}^2|g_k|}\cdot\max_{1\leq j\leq k}\left\{\frac{|g_j'|}{|g_j|}\right\}
\end{aligned}
\end{equation}
and
\begin{equation}\label{prod_form2}
\begin{aligned}
\left|\frac{\partial^2}{\partial\theta^2}\left(\frac{1}{g_1^2\cdots g_k^2}\right)\right|&\leq
\frac{8k^2}{g_1^2\cdots g_k^2}\left(\max_{1\leq j\leq k}\left\{\frac{|g_j'|}{|g_j|}\right\}^2+
\max_{1\leq j\leq k}\left\{\frac{|g_j''|}{|g_j|}\right\}\right) \\
\left|\frac{\partial^2}{\partial\theta^2}\left(\frac{1}{g_1^2\cdots g_{k-1}^2g_k}\right)\right|&\leq
\frac{8k^2}{g_1^2\cdots g_{k-1}^2|g_k|}\left(\max_{1\leq j\leq k}\left\{\frac{|g_j'|}{|g_j|}\right\}^2+
\max_{1\leq j\leq k}\left\{\frac{|g_j''|}{|g_j|}\right\}\right).
\end{aligned}
\end{equation}

Recall that there was the (artificial) problem that the iterate $r_{k+1}$ could be very large if $r_{k}$ is close to zero. 
This will, of course, also affect the derivatives. We solved this problem by
pairing the iterate $r_k$ with $r_{k+1}$, writing $\rho_i=r_kr_{k+1}=vr_k-1$, if $|r_k|<\la^{-2}$. 
This was the content of Lemma \ref{BE4}. In the following lemma we derive an expression of $r_k'$ in terms of
the $\rho_i$:s, provided that $|r_{k-1}|\geq \la^{-2}$ (if $|r_{k-1}|$ is very small, there need not be any bound on $|r_k'|$).

\begin{lem}\label{A_L1}
Assume that $r_0=\infty$ and that $|r_{k-1}|\geq\la^{-2}$ for some $k\geq 1$. Then
we have
$$
r_{k}'=v'+\frac{a_{l-1}}{\rho_{l-1}^2}+\ldots +\frac{a_1}{\rho_{l-1}^2\cdots \rho_{1}^2}
$$
where the $a_i$ ($i\in [1,l-1]$) satisfy $|a_i|<(3/2)\|v'\|$. More precisely, we either have $a_i=v'$, or $a_i=v'r_j^2+v'$ and $\rho_i=r_jr_{j+1}$.
Furthermore, we have the bounds
$$
|a_i'|< \co\la\left(1+\frac{1}{\rho_{i-1}^2}+\ldots +\frac{1}{\rho_{i-1}^2\cdots \rho_1^2}\right),\quad i\in [1,l-1].
$$
\end{lem}
\begin{proof}Since $|r_{k-1}|\geq \la^{-2}$ we can use Lemma \ref{BE4} and write $r_1\cdots r_{k-1}=\rho_1\cdots \rho_{l-1}$.
We now use formula (\ref{firstDerF}). If $|r_j|<\la^{-2}$ (for some $1\leq j\leq k-2$), so $\rho_i=r_jr_{j+1}$, then we pair the two terms
$$
\frac{v'}{r_{k-1}^2\cdots r_{j+1}^2}+\frac{v'}{r_{k-1}^2\cdots r_{j}^2}=\frac{v'r_j^2+v'}{r_{k-1}^2\cdots r_{j}^2}.
$$
From this the first statement follows. To get the second one we note that if $a_i=v'r_j^2+v'$ (so $|r_j|<\la^{-2}$, and thus $r_{j-1}\gg\la^{-2}$), then 
$$
|a_i'|\leq \|v''\|(1+\la^{-4})+2\|v'\||r_j||r_j'|<\co\la (1+1/\rho_{i-1}^2+\ldots+1/\rho_{i-1}^2\cdots \rho_1^2)
$$
since $|r_j'|\leq 2\|v'\|(1+1/\rho_{i-1}^2+\ldots+1/\rho_{i-1}^2\cdots \rho_1^2)$.
\end{proof}

We can use this lemma to derive estimates on the $\rho_j'$.
\begin{lem}\label{A_L2}
Assume that $r_0=\infty$ (so $r_1=v$) and define the $\rho_j$ from the $r_k$, as in Lemma \ref{BE4}. We then have
$$
|\rho_j'|\leq 2\|v'\|\left(1+\frac{1}{\rho_{j-1}^2}+\ldots +\frac{1}{\rho_{1}^2\cdots\rho_{j-1}^2}\right),\quad \text{if } 
\rho_j=r_k 
$$
$$
|\rho_j'|\leq 2\|v\|\|v'\|\left(1+\frac{1}{\rho_{j-1}^2}+\ldots +\frac{1}{\rho_{1}^2\cdots\rho_{j-1}^2}\right),\quad \text{if } 
\rho_j=r_kr_{k-1} 
$$
\end{lem}
\begin{proof}This follows immediately from the previous lemma, since $r_kr_{k-1}=vr_{k-1}-1$.
\end{proof}

\end{subsection}
\begin{subsection}{Global upper bounds}
Here we derive global upper bounds on $r_k'$ and $r_k''$, provided that $|r_k|$ is not too large. They are formulated in the way we shall need them.
\end{subsection}
\begin{prop}\label{A_GB}
Assume that $|r_0'|,|r_0''|<\la^2$. Then for all $k\geq 1$ such that $|r_{k}|<  \la^{5/4}$ we have the bounds
$$
|r'_k|< \la^{3(k+1)} \quad \text{and} \quad |r_k''|<\la^{6(k+1)}.
$$
\end{prop}
\begin{proof}Assume that for some $k\geq 0$ we have $|r_k'|<(1/2)(\la^3+\la^6+\ldots+\la^{3(k+1)})$
and $|r_k''|<\la^{6(k+1)}$ (this holds, in particular, for $k=0$). We get two cases:

If $|r_k|\geq \la^{-3/2}$, then, since $|v'|<\co\la\ll\la^3/2$, 
$$
|r_{k+1}'|=|v'+r_k'/r_k^2|<(1/2)(\la^3+\la^6+\ldots+\la^{3(k+2)})
$$
and, since $|r_k'|<\la^{3(k+1)}$, 
$$
|r_{k+1}''|=|v''+r_k''/r_k^2-2(r_k')^2/r_k^3|<\la^2+\la^{6(k+1)+3}+2\la^{6(k+1)+9/2}<\la^{6(k+2)}.
$$

If $|r_k|<\la^{-3/2}$, then $|r_{k+1}|>\la^{3/2}/2>\la^{5/4}$. Moreover, since $\|v\|_{C^0(\T)}<\co \la$, we have $|r_kv|<\co \la^{-1/2}$. Writing $r_{k+2}$ as 
$$
r_{k+2}=v-\frac{1}{v-1/r_k}=v-\frac{r_k}{r_kv-1},
$$
easy calculations, using the bounds on $|r_k'|,|r_k''|$, show that $|r_{k+2}'|<(1/2)(\la^3+\la^6+\ldots+\la^{3(k+3)})$ and $|r_{k+2}''|<\la^{6(k+3)}$.
\end{proof}

\begin{subsection}{Abstract estimates}
Now we shall use the above formulae to 
derive estimates under the assumption that we have a certain control on the iterates of $r_0=\infty$. 

In the first lemma we assume that all the iterates $r_1,r_2,\ldots, r_l$ are large.

\begin{lem}\label{A_01}
Assume that $r_0=\infty$ and that we for some $\theta\in \T$ have $|r_k|\geq \la^{3/4}$ for all $k\in [1,l]$, some $l\geq 1$. 
Let 
$$
h(\theta)=\frac{1}{r_1^2\cdots r_l^2} \text{ and } g(\theta)=\frac{1}{r_1^2\cdots r_{l-1}^2 r_l}.
$$
Then have the
following estimates:
$$
|g'(\theta)|,|g''(\theta)|<\frac{1}{\la^{l/5}}, \quad\text{and} \quad
|h'(\theta)|<\sqrt{h(\theta)}, ~|h''(\theta)|<\frac{1}{\la^{l/5}}.
$$
\end{lem}
\begin{proof}We will show that we have 
$$
|r'_k|\leq 2\|v'\|
\text{ and } |r''_k|\leq 2\|v''\| \text{ for all } k\in [1,l].
$$
From this all the estimates follow easily from the bounds (\ref{A_v}) and formulae (\ref{prod_form1}) and (\ref{prod_form2}).

Using  formula (\ref{firstDerF}) and the given bounds $|r_j|\geq \la^{3/4}$ ($j\in [1,k-1]$), we  
see that $|r'_k|\leq 2\|v'\|$ for all $k\in [1,l]$.

To obtain the bound on the second derivative we proceed by induction. By definition we have $r_1=v$, so $r_1''=v''$.
Let $a=\|v''\|+8\|v'\|^2/\la^{9/4}$. Note that $|r_1''|\leq a$.
Assume that we for some $k\geq 1$ have shown that $|r_k''|\leq a(1+1/\la^{3/2}+\ldots +1/\la^{(3/2)(k-1)})$. Then 
$$
\begin{aligned}
|r''_{k+1}|=|v''+r''_k/r_k^2-2(r_k')^2/r_k^3|\leq a+|r_k''|/\la^{3/2}\leq \\
\leq a(1+/\la^{3/2}+\ldots +1/\la^{(3/2)(k)}).
\end{aligned}
$$
Thus, for each $k\geq 1$ we have $|r''_k|\leq 
a(1+1/\la^{3/2}+\ldots)<2\|v''\|$.
\end{proof}

In the following lemma we shall derive similar estimates as in the previous one, but under weaker conditions on the iterates.

\begin{lem}\label{A02}
Assume that $r_0=\infty$  
and that for some $\theta\in \mathbb{T}$ and $l\geq 10$ 
we have the estimate
$$
|\rho_1|^{\alpha_i}\cdots |\rho_l|^{\alpha_l}\geq \la^{2l/3}\quad \text{ for any choices of } 
\alpha_1,\ldots, \alpha_l\in [1,6]. 
$$ 
Let
$$
h(\theta)=\frac{1}{\rho_1^2\cdots \rho_{l}^2}  \text{ and } g(\theta)=\frac{1}{\rho_1^2\cdots \rho_{l-1}^2\rho_l}.
$$
Then 
$$
|g'(\theta)|,~|g''(\theta)|<\frac{1}{\la^{l/5}} \quad\text{and}
\quad
|h'(\theta)|<\sqrt{h(\theta)}, ~|h''(\theta)|<\frac{1}{\la^{l/5}}.
$$
\end{lem}
\begin{proof}Let 
$$
m=\max\left\{1, 
\max_{1\leq q\leq l}\left\{\frac{1}{|\rho_q|}\right\},
\max_{1\leq p\leq q\leq l-1}\left\{\frac{1}{\rho_p^2\cdots \rho_{q}^2|\rho_{q+1}|}\right\}
\right\}.
$$
We will show that we always have the bounds
\begin{equation}\label{A02_eq1}
\frac{|\rho_i'|}{|\rho_i|}<\co \la^2lm, \quad i=1,\ldots, l
\end{equation}
\begin{equation}\label{A02_eq2}
\frac{|\rho_i''|}{|\rho_i|}<\co \la^4l^3m^2, \quad i=1,\ldots, l.
\end{equation}
Applying the estimate (\ref{prod_form1}) we therefore have
$$
\frac{1}{\sqrt{h(\theta)}}|h'(\theta)|<\frac{2l}{|\rho_1\cdots \rho_l|}(\co \la^2lm)<
\frac{\co l^2\la^2}{\rho_1^{\alpha_1}\cdots \rho_l^{\alpha_l}}
$$
and
$$
|g'(\theta)|<\frac{\co l^2\la^2}{\rho_1^{\alpha_1}\cdots \rho_l^{\alpha_l}}
$$
for some choice of $\alpha_i\in [1,4]$ ($i=1,\ldots, l$). Since $l\geq 10$, the assumptions on the products imply that
$$
\frac{\co l^2\la^2}{\rho_1^{\alpha_1}\cdots \rho_l^{\alpha_l}}<\la^{-l/5}.
$$
Similarly, using estimate (\ref{prod_form2}), we get
$$
|h''(\theta)|, |g'(\theta)|<\frac{\co l^5\la^4m^2}{\rho_1^2\cdots \rho_{l-1}^2|\rho_l|^{1,2}}<\frac{\co l^5\la^4}{\rho_1^{\alpha_1}\cdots \rho_l^{\alpha_l}}<\la^{-l/5}
$$
for some $\alpha_i\in [1,6]$ ($i=1,\ldots, l$).

It remains to verify (\ref{A02_eq1}) and (\ref{A02_eq2}). First, (\ref{A02_eq1}) follows immediately from Lemma \ref{A_L2}, since the estimate for $\rho_i'$ contains at most
$l$ terms. To get (\ref{A02_eq2}) we first note that (\ref{A02_eq1}), Lemmas \ref{A_L1} and \ref{A_L2}, and estimate (\ref{prod_form2}) imply
that for all $1\leq p\leq q\leq l-1$ we have
$$
\frac{1}{|\rho_{q+1}|}\left|\frac{\partial}{\partial \theta}\left(\frac{a_p}{\rho_q^2\cdots\rho_p^2}\right)\right|<\co\la^3l^2m^2.
$$
If $\rho_i=r_j$, it follows, by differentiating the expression for $r_j'$ in Lemma \ref{A_L1}, that 
$$
\frac{|\rho_i''|}{|\rho_i|}<\co \la^3l^3m^2.
$$
If $\rho_i=r_jr_{j+1}=r_jv-1$, where $|r_j|<\la^{-2}$, then we get
$$
\frac{|\rho_i''|}{|\rho_i|}<\co \la^4l^3m^2.
$$
Thus, in either case we we have the estimate (\ref{A02_eq2}).
\end{proof}

\bigskip
The next proposition in this section gives us control on the first and second derivative (w.r.t. $\theta$) of the $k$th iterate $r_k$, $k>0$, 
under suitable assumptions on the previous iterates $r_0,\ldots,r_{k-1}$. Roughly it says that  
$r_k(\theta)$ is close to $\la f(\theta+(k-1)\omega)-E$ in $C^2$-norm under some conditions.

\begin{prop}\label{A1}
Assume that $r_0=\infty$ and that we for some $\theta\in \T$ 
and $k\geq 1$ have $|r_{k-1}|\geq \la^{-2}$, so  $r_1\cdots r_{k-1}$ can be written as $\rho_1\cdots \rho_{l-1}$, 
and for each $i\in [1,l-1]$ we have
$$
\begin{aligned}
|\rho_i|^{\alpha_i}\cdots |\rho_{l-1}|^{\alpha_{l-1}}\geq\lambda^{2(l-i)/3}\quad \text{ for any choices of } 
\alpha_i,\ldots, \alpha_{l-1}\in [1,2]. 
\end{aligned}
$$
Then 
$$
|r_{k}'(\theta)-\lambda f'(\theta_{k-1})|<\lambda^{-1/4}\quad\text{and}
\quad |r_k''(\theta)-\lambda f''(\theta_{k-1})|<\la^{1/2}.
$$
\end{prop}
\begin{proof}
By Lemma \ref{A_L1}, applying the given estimates on the products, we immediately get
$$
|r_{k}'(\theta)-\la f'(\tht_{k-1})|\leq 2\|v'\|
\left(\frac{1}{\rho_{l-1}^2}+\ldots \frac{1}{\rho_{l-1}^2\cdots \rho_{1}^2}\right)<
\frac{\co\la}{\la^{4/3}}=\frac{\co}{\la^{1/3}}.
$$

We now turn to the estimate of the second derivative. Using Lemma \ref{A_L1}, we see that differentiation of $r_k'$ yields
$$
r_{k}''-v''=\sum_{i=1}^{l-1}\frac{\partial}{\partial\theta}\left(\frac{a_i}{\rho_{l-1}^2\cdots \rho_{i}^2}\right).
$$
The terms in the RHS can be estimated as follows. In the case where $i=l-1$, we have by assumption that 
$|\rho_{l-1}|\geq \la^{2/3}$. Thus $\rho_{l-1}=r_{k-1}$ and $a_{l-1}=v'$. We therefore have, using Lemma \ref{A_L2},
$$
\left|\frac{\partial}{\partial\theta}\left(\frac{a_{l-1}}{\rho_{l-1}^2}\right)\right|\leq
\frac{\|v''\|}{\la^{4/3}}+\frac{2\|v'\|^2}{\rho_{l-1}^3}
\left(1+\frac{1}{\rho_{l-2}^2}+\ldots +\frac{1}{\rho_{l-2}^2\cdots\rho_1^2}\right)<\co 
$$
For $i<l-1$ we do as follows. We have
$$
\left|\frac{\partial}{\partial\theta}\left(\frac{a_i}{\rho_{l-1}^2\cdots \rho_{i}^2}\right)\right|\leq
\frac{|a_i'|}{\rho_{l-1}^2\cdots \rho_{i}^2}+
|a_i|\left|\frac{\partial}{\partial\theta}\left(\frac{1}{\rho_{l-1}^2\cdots \rho_{i}^2}\right)\right|.
$$
By Lemma \ref{A_L1} the first term can be estimated by
$$
\co\la \left(\frac{1}{\rho_{l-1}^2\cdots \rho_i^2}+\ldots +\frac{1}{\rho_{l-1}^2\cdots \rho_1^2}\right)<\co 
\frac{\la}{\la^{4/3(l-i)}}.
$$
Next, by Lemma \ref{A_L2} we see that 
$$
|\rho_i'|<\co \la^2\left(1+\frac{1}{\rho_{i-1}^2}+\ldots +
\frac{1}{\rho_{i-1}^2\cdots \rho_{1}^2}
\right),
$$
and for $\iota\in [i+1,l-1]$ we have
$$
|\rho_\iota'|<\co \la^2\left(1+\frac{1}{\rho_{\iota-1}^2}+\ldots +
\frac{1}{\rho_{\iota-1}^2\cdots \rho_i^2}+\frac{1}{\rho_{\iota-1}^2\cdots \rho_{i-1}^2}+\ldots+\frac{1}{\rho_{\iota-1}^2\cdots \rho_{1}^2}
\right).
$$
Thus, using the assumtion on the products, we can conclude that for any  $\iota\in [i,l-1]$
$$
\frac{1}{\rho_{l-1}^2\cdots \rho_i^2}\frac{|\rho_\iota'|}{|\rho_\iota|}<\co\frac{\la^2(l-i)}{\la^{(4/3)(l-i)}}.
$$
It therefore follows from the estimate (\ref{prod_form1}) that the second term can be bounded from above by
$$
\co |a_i| \frac{(l-i)^2\la^2}{\la^{(4/3)(l-i)}}<\co \frac{(l-i)^2\la^3}{\la^{(4/3)(l-i)}}.
$$ 
Putting these estimates together shows that $|r_k''-v''|<\la^{1/2}$.

\end{proof}

\bigskip
The last propostion in this section will be used when we want to ``shadow'' orbits. In these cases we will 
control the iterates of $s_0=s_0(\theta)$ by comparing them with the iterates of the constant  $r_0(\theta)=\infty$. Our settings
will be the following. 

\begin{prop}\label{A5}
Assume that $s_0=s_0(\theta)$ is any given function, and let 
$r_0=r_0(\theta)=\infty$. Assume that for some $\theta\in \T$ and some $k\geq 1$ we have the following: 
$r_1\cdots r_k$ can be written as $\rho_1\cdots \rho_l$ (i.e., $|r_k|\geq \la^{-2}$), and
$$
|r_j|\geq \la^{3/4},\quad j\leq 10,\quad\text{and}
$$ 
$$
|\rho_1|^{\alpha_1}\cdots |\rho_i|^{\alpha_i}\geq \la^{2i/3} \text{ for any } i\leq l \text{ and any choices of } 
\alpha_1,\ldots \alpha_i\in [1,6].
$$
Then we can write $s_{k+1}$ as
\begin{equation}\label{A_eq6}
s_{k+1}=r_{k+1}-\frac{h(\theta)}{(s_0-w(\theta))}
\end{equation}
where the functions $h$ and $w$ satisfy
$$
\begin{aligned}
\la^{-4l}< h(\theta)\leq\la^{-4l/3}, ~&|h'(\theta)|<\sqrt{h(\theta)}, ~|h''(\theta)|< 1; \text{and} \\
 |w(\theta)|< \frac{1.1}{\la^{3/4}},  ~&|w'(\theta)|,|w''(\theta)|<1/10.
\end{aligned}
$$
\end{prop}
\begin{proof}
Let 
$$
h(\theta)=\frac{1}{r_{1}^2\cdots r_{k}^2}=\frac{1}{(\rho_1\cdots\rho_l)^2}
$$
and 
$$
w(\theta)=\frac{1}{r_1}+\frac{1}{r_1^2r_2}+\frac{1}{r_1^2r_2^2r_3}+\ldots 
+ \frac{1}{r_1^2\cdots r_{k-1}^2r_{k}}.
$$
From formula (\ref{A_form1}) we then see that indeed $r_{k+1}$ can be written as in (\ref{A_eq6}).

By the assumption on the products, the estimates on $h$ immediately follow from Lemmas (\ref{A_01}) and (\ref{A02}) 
(by Lemma \ref{BE4} we always have the trivial upper bound $|\rho_1\cdots \rho_l|<\la^{2l}$).

Turning to the function $w$, we claim that it can be written as
\begin{equation}\label{A5_eq1}
w(\theta)=\frac{1}{\rho_1}+\frac{1}{\rho_1^2\rho_2}+\ldots+\frac{1}{\rho_1^2\cdots \rho_9^2\rho_{10}}+
\frac{a_{11}}{\rho_1^2\cdots \rho_{10}^2\rho_{11}}+\ldots
+\frac{a_l}{\rho_1^2\cdots\rho_{l-1}^2\rho_l},
\end{equation}
where the $a_i=1$ or $a_i=v$ (evaluated at some $\theta$).
Indeed, let $j\geq 1$ be the first $j$ such that $|r_j|<\la^{-2}$ (if there is one). We have assumed that 
$|r_k|\geq \la^{-2}$ so $j<k$. Hence we can write $r_1\cdots r_{j+1}=\rho_1\cdots \rho_i$. We note that we have
$$
\begin{aligned}
\frac{1}{r_1^2\cdots r_{j-1}^2r_j}+\frac{1}{r_1^2\cdots r_{j-1}^2r_j^2r_{j+1}}=
\frac{r_jr_{j+1}+1}{r_1^2\cdots r_{j-1}^2r_j^2r_{j+1}}=&\frac{r_jv}{r_1^2\cdots r_{j-1}^2r_j^2r_{j+1}}=\\
=&\frac{v}{\rho_1^2\cdots \rho_{i-1}^2\rho_{i}}.
\end{aligned}
$$
Continuing, pairing terms like above, gives the result.

Using the estimates on the products, we get $|w(\theta)|<1.1/\la^{3/4}$. 
Moreover, the bounds on $w',w''$ follow by differentiation of each term in (\ref{A5_eq1}), and applying  Lemmas (\ref{A_01}) and (\ref{A02}).
\end{proof}

We end this subsection with a well-known result which we prove using the formulae in the previous proposition.

\begin{lem}\label{A_Osc}
Let $r_0=\infty$. Assume that for some $\theta\in \T$ we have the estimate
$$
|\rho_1\cdots \rho_l|\to\infty \text{ as } l\to \infty.
$$ 
Then there exists $w\in \mathbb{R}$ such that for all $s_0\neq w$ we have
$$
\min\{|r_k-s_k|,|r_{k+1}-s_{k+1}|\}\to 0 \text{ as } k\to \infty.
$$
\end{lem}
\begin{proof}We let $(k_j)_{j\geq 1}$ be the indices such that $|r_{k_j}|\geq \la^{-2}$. We also let
$$
w_j=\frac{1}{r_1}+\frac{1}{r_1^2r_2}+\frac{1}{r_1^2r_2^2r_3}+\ldots 
+ \frac{1}{r_1^2\cdots r_{k_j-1}^2r_{k_j}} \text{ and } h_j=\frac{1}{r_1^2\cdots r_{k_j}^2}.
$$
As in the proof of the previous proposition we can write
$$
w_j=\frac{a_1}{\rho_1}+\frac{a_2}{\rho_1^2\rho_2}+\ldots
+\frac{a_{l_j}}{\rho_1^2\cdots\rho_{l_{j}-1}^2\rho_{l_j}},
$$
where each $|a_i|<\co\la$. Since, by definition, all the $|\rho_i|\geq \la^{-2}$ (see Lemma \ref{BE4}), we can use the assumptions and derive
$|w_{j+1}-w_j|=\frac{|a_{l_{j+1}}|}{\rho_1^2\cdots\rho_{l_{j+1}-1}^2|\rho_{l_{j+1}}|}\to 0$ as $j\to\infty$.
Thus it follows that there exists $w\in\mathbb{R}$ such that
$w_j\to w \text{ as } j\to\infty$. Furthermore, $h_{j}\to 0$ as $j\to\infty$.

Now, if $s_0\neq w$, then there exists $\ve>0$ and a $j_0>0$ such that $|s_0-w_j|>\ve$ for all $j\geq j_0$. We conclude that, using formula (\ref{A_form1}),
we have
$$
|r_{k_j+1}-s_{k_j+1}|=\frac{h_{j}}{s_0-w_j}\to 0 \text{ as } j\to \infty.
$$
\end{proof}


\end{subsection}

\begin{subsection}{Derivative with respect to the parameter $E$} To show that a certain configuration can only happen at a gap edge (more precisely, Case 2 in Section 
\ref{the_proof}), we will have do differentiate with respect to the parameter $E$.

Since $v=\la f-E$ we have $\partial_E v=-1$. Therefore, if $r_0=\infty$, we get, exactly as in (\ref{firstDerF}), that
\begin{equation}\label{DerE}
\partial_E r_k=-\left(1+\frac{1}{r_{k-1}^2}+\ldots +\frac{1}{r_{k-1}^2\cdots r_1^2}\right), ~k>0.
\end{equation}
\begin{prop}\label{A_E}
The statements of Proportions \ref{A_GB} and \ref{A5} concerning the first derivatives hold 
if we replace $\theta$-derivative by $E$-derivative, i.e., if we mean $'=\partial_E$.
Moreover, in Proposition \ref{A1} we also have the conclusion $|\partial_E r_k(\theta,E)-(-1)|<1$.
\end{prop}
\begin{proof} By using the above formula, all the estimates are the same. Of course the statements also hold for the second derivative (with respect to $E$), 
but we do not need them.
\end{proof}
\end{subsection}



\section{} Here we collect some elementary results of a more geometrical nature. 
We will focus on functions $\psi$, defined on some interval $I$, of the form 
\begin{equation}\label{B_eq0}
\psi(\theta)=s(\theta)-\frac{h(\theta)}{g(\theta)}, ~\theta\in I,
\end{equation}
(under various conditions on $s,h,g$; recall the shadowing formula (\ref{A_eq6})). The aim is to derive information about sets of the form
$$
\{\theta\in I: |\psi(\theta)|\leq C\},
$$
as well as estimates on $\psi$ on these sets. We shall in this appendix always assume that $I$ is an interval, and that the functions $s,h$ and $g$ are $C^2$ on $I$.

First note that we have the following formulae: Given any $\varepsilon\in \mathbb{R}$, let 
\begin{equation}\label{B_eq1}
\varphi_\varepsilon(\theta)=\frac{h(\theta)}{s(\theta)-\varepsilon}.
\end{equation}
A straightforward computation shows that if $t\in I$ is such that 
$$\varphi_\varepsilon(t)=g(t),$$ 
then we have 
$$\psi(t)=\ve,$$
\begin{equation}\label{B_eq2}
\psi'(t)=\frac{(s(t)-\ve)^2}{h(t)}\left(g'(t)-\varphi_\ve'(t)\right); \quad\text{and}
\end{equation}
\begin{equation}\label{B_eq3}
\psi''(t)=\frac{(s(t)-\ve)^2}{h(t)}\left(g''(t)-\varphi_\ve''(t) +2\left(\frac{s'(t)}{s(t)-\ve}-\frac{s(t)-\ve}{h(t)}\cdot g'(t)\right)
(g'(t)-\varphi_\ve'(t))\right).
\end{equation}

We shall assume that $\lambda$ is large (larger than some numerical constant).

\bigskip
In the first two lemmas the functions $h$ and $s$ are assumed to satisfy the following conditions on $I$: there is a constant
$\delta_2$, satisfying $0< \delta_2\leq 1$, such that 
\begin{equation}\label{B_cond1}
\begin{aligned}
|s(\theta)|\geq \la^{3/4} \text{ and } \|s\|_{C^2(I)}<\co\la; \\
0< h(\theta)\leq\delta_2, ~|h'(\theta)|\leq \sqrt{h(\theta)} \text{ and } ~|h''(\theta)|<1.
\end{aligned}
\end{equation}
An easy calculation shows that under these conditions we have the following estimates for the function $\vp_\ve$ defined in (\ref{B_eq0}): 
\begin{equation}\label{B_est1}
|\vp_\ve(\theta)|<\frac{h(\theta)}{\sqrt{\la}}, ~|\vp'_{\ve}(\theta)|<\sqrt{h(\theta)/s(\theta)} \text{ and } 
~|\vp_\ve''(\theta)|<1 \quad \text{for all } |\ve|\leq \sqrt{\la}.
\end{equation}

\begin{lem}\label{L_B1} 
Assume that $s$ and $h$ satisfy condition (\ref{B_cond1}), and
that $g$ satisfies 
$$
g(I)\supset [-\delta_2,\delta_2], \text{ and } ~ g'(\theta)\geq D_1 \text{ on } I
$$
for some constant $D_1\geq \sqrt{\la}$. If we let $\psi(\theta)$ be defined as in (\ref{B_eq0}),
then the set
$$
J=\{\theta\in I: |\psi(\theta)|\leq 2\la^{-3/4}\}
$$
is a non-degenerate interval, $\psi(J)=[-2\la^{-3/4},2\la^{-3/4}]$, and for all $\theta\in J$ we have 
$$
\psi'(\theta)>\frac{\la D_1}{\delta_2}.
$$
\end{lem}
\begin{rem}Note that if we would have the assumption $g'(\theta)\leq -D_1$ instead (this situation  we will indeed have when we apply the lemma in Section \ref{the_induction}), 
then we can apply the lemma
to $\tilde{\psi}(\theta)=-\psi(\theta)=(-s(\theta))-h(\theta)/(-g(\theta))$, and obtain $\psi'(\theta)<-(\la D_1)/\delta_2$. 

The same comment applies to all the lemmas below; we formulate them only in one case, for the sake of readability, and then the "mirror image case" is obtained by applying them to $-\psi$. 
\end{rem}

\begin{proof} Fix $\ve\in [-2\la^{-3/4},2\la^{-3/4}]$ and let $\vp_\ve(\theta)$ be defined as in (\ref{B_eq1}).
We have the estimates (\ref{B_est1}).
From the assumptions on $g$ it therefore follows that there must be a unique $t\in I$ such that
$\vp_\ve(t)=g(t)$ (and hence $\psi(t)=\ve$). Using formulae (\ref{B_eq2})  yields the estimates.

\end{proof}

\begin{lem}\label{L_B2} 
Assume that $s,h$ satisfy condition (\ref{B_cond1}), and that $g$ satisfies 
$$
\begin{cases}
g(a),g(b)\geq \delta_2 ; ~\text{and} \\ 
\min_{\theta\in I}\max\{|g'(\theta)|, g''(\theta)\}\geq D_1;
\end{cases}
$$ 
where $I=[a,b]$ and  $D_1\geq \sqrt{\la}$.
If we let $\psi(\theta)$ be defined as in (\ref{B_eq0}) and let $J$ be the set
$$
J=\{\theta\in I: |\psi(\theta)|\leq 2\la^{-3/4}\},
$$
then we have the estimates
$$
\min_{\theta\in J} \max\{|\psi'(\theta)|, \psi''(\theta)\}>\sqrt{\frac{D_1}{\delta_2}}. 
$$
Moreover, the set $J$ is either 
\begin{itemize}
\item[$(1)$] empty;  
\item[$(2)$] a single point; 
\item[$(3)$] a single interval $[a',b']$ and $\psi(a')=\psi(b')=2\la^{-3/4}$; or
\item[$(4)$] consists of two intervals, $J^1, J^2$, and $\psi(J^i)=[-2\la^{-3/4},2\la^{-3/4}]$ ~$(i=1,2)$, and
 $\psi'(\theta)>0$ on $J^1$, $\psi'(\theta)<0$ on $J^2$.
\end{itemize}
\end{lem}
\begin{proof}Set $\ve_0=2\la^{-3/4}$. We shall focus on $\ve\in [-\ve_0,\ve_0]$, and we shall by $\vp_\ve$ denote the function defined in (\ref{B_eq1}). 
We first notice that we have the following monotonicity property (since $h(\theta)>0$ on $I$): if $\ve,\ve'\in [-\ve_0,\ve_0]$ and $\ve>\ve'$, then
$\vp_\ve(\theta)>\vp_{\ve'}(\theta)$ for all $\theta\in I$.

From the assumptions on $g$ and the estimates in (\ref{B_est1}) we have
$$\min_{\theta\in I}\max\{|g'(\theta)-\vp_\ve'(\theta)|, g''(\theta)-\vp_\ve''(\theta)\}> D_1/2\geq \sqrt{\la}/2.$$
and $g(a)-\vp_\ve(a), g(b)-\vp_\ve(b) >\delta_2/2$ for all $\ve\in [-\ve_0,\ve_0]$.
Thus, the equation $g(\theta)=\vp_\ve(\theta)$ can have $0, 1$ or $2$ solutions.

By using the monotonicity property we thus get the following: if $g(\theta)>\vp_{\ve_0}(\theta)$ for all $\theta\in I$, then $J=\emptyset$;
if $g(t)=\vp_{\ve_0}(t)$ for some $t\in I$, and $g(\theta)>\vp_{\ve}(\theta)$ for all $\theta\in I$ and all $\ve<\ve_0$ then $J$ consists of one single point;
if $g(t)<\vp_{\ve_0}(t)$ for some $t\in I$, and $g(\theta)\geq\vp_{-\ve_0}(\theta)$ for all $\theta$, then $J$ consists of one interval;
if $g(t)<\vp_{-\ve_0}(t)$ for some $t\in I$, then $J$ consists two intervals.

Now we derive the estimates on the derivatives. If there is a $t\in I$ such that $\vp_\ve(t)=g(t)$ (for some $\ve\in [-\ve_0,\ve_0]$), then, by using formulae
(\ref{B_eq2}) and (\ref{B_eq3}), together with the estimates on $\vp_\ve$ in (\ref{B_est1}), we get the following:

(1) if $|g'(t)|\geq \frac{1}{10}\sqrt{D_1h(t)/s(t)}$, then $|\psi'(t)|>\sqrt{D_1}/\sqrt{\delta_2}$, and $\psi'(t)$ has the same sign
as $g'(t)$; 

(2) if $|g'(t)|< \frac{1}{10}\sqrt{D_1h(t)/s(t)} (\ll D_1)$, we must, by our assumptions on $g$, have $g''(t)\geq D_1$, which implies $\psi''(t)>(\la D_1)/\delta_2$. 
This gives us the lower bounds.

\end{proof}


The subsequent two lemmas are to be used in the the resonant cases in the inductive construction. The first of them is adapted for the geometric part of 
the base case (Lemma \ref{BC4}).
\begin{lem}\label{L_B5}
Assume that $g, h$ and $s$ satisfy:
$$\|g\|_{C^2(I)},\|s\|_{C^2(I)} <\co\la;$$ 
$$
\begin{aligned}
s(I)&\supset [-\la^{3/4},\la^{3/4}],  \text{ and } ~s'(\theta)>\co \la^{7/8} \text{ if } s(\theta)\in [-\la^{3/4},\la^{3/4}]; \\
g(I)&\supset [-\la^{3/4},\la^{3/4}], \text{ and } ~g'(\theta)<-\co \la^{7/8} \text{ if } g(\theta)\in [-\la^{3/4},\la^{3/4}]; \text{ and}
\end{aligned}
$$
$$
0<\delta\leq h(\theta)\leq 1, ~|h'(\theta)|\leq \sqrt{h(\theta)}, \text{ and } |h''(\theta)|<1.
$$
Let $\psi$ be defined as in (\ref{B_eq0}). Then there exists two (non-empty) intervals $I^1,I^2\subset I$ such that 
either we have
\begin{enumerate} 
\item $|\psi'(\theta)|>\sqrt{\la}$ on $I^1\cup I^2$, $\psi'$ having opposite signs on $I^1$ and $I^2$, 
$$
\psi(I^i)=[-\sqrt{\la},\sqrt{\la}] ~(i=1,2),
$$
and $|\psi(\theta)|>\sqrt{\la}$ if $\theta\notin I^1\cup I^2$ ; \ or

\item writing $I^i=[a_i,b_i]$ ($i=1,2$) we have 
$$
\psi(a_1)=\psi(b_1)=-\la^{2/3} \text{ and } \psi(a_2)=\psi(b_2)=\la^{2/3};
$$
$$
\min_{\theta\in I^1}\max\{|\psi'(\theta)|,-\psi''(\theta)\}> \sqrt{\la} ~\text{ and }~ \min_{\theta\in I^2}\max\{|\psi'(\theta)|,\psi''(\theta)\}>\sqrt{\la};
$$
$$
\min_{\theta\in I^2}\psi(\theta)-\max_{\theta\in I^1}\psi(\theta)>\co\frac{\sqrt{\delta}}{\la^{1/4}};
$$
and $|\psi(\theta)|>\la^{2/3}$ if $\theta\notin I^1\cup I^2$.
\end{enumerate}
\end{lem}
\begin{figure}
\psfrag{t}{$\theta$}
\psfrag{p}{$p$}
\psfrag{psi}{$\psi$}
\includegraphics[width=5cm]{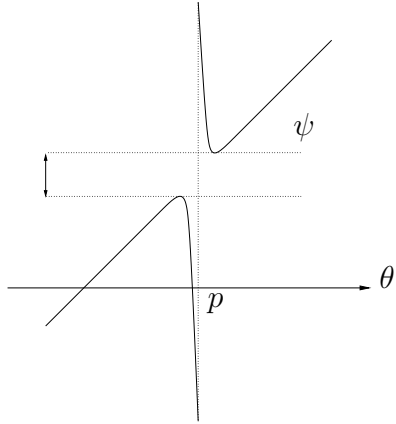}
\caption{The graph $y=\psi(\theta)$ in Lemmas \ref{L_B5} and \ref{L_B6}}\label{fig3}
\end{figure}

\begin{proof}
This is just an easy computation, but we include the details for completness. We let 
$$
L=\{\tht\in I: |g(\tht)|\leq \la^{1/8}\}.
$$
From the assumptions on $g$ it follows that $L$ is an interval, which lies in the interior of $I$, and
$$
|L|<\co \la^{-3/4}. 
$$
We shall write $I=[a,b]$ and $L=[a',b']$.

Let $p\in I$ be the (unique) zero of $g$, i.e., $g(p)=0$. We also let $w(\theta)=-h(\theta)/g(\theta)$, so we can write the function $\psi$ as
$$
\psi(\theta)=s(\theta)+w(\theta).
$$ 
From the assumptions on $h$ and $g$ it immediately follows that $w(\theta)>0$ for $\theta>p$, $w(\theta)<0$ for $\theta<p$, and
$w(\theta)\to -\infty$ as $\theta\to p^-$ and $w(\theta)\to +\infty$ as $\theta\to p^+$.

Easy calculations, using the assumptions on $g$ and $h$, show that we have the estimates
\begin{equation}\label{B4_eq1}
|w(\tht)|\leq \la^{-1/8}, |w'(\tht)|<\co \la^{3/4}  \text{ and } |w''(\tht)|<\co\la^{13/8} \text{ for all } \theta\in I\setminus L.
\end{equation}

We first handle the situation when $|s(p)|\geq 4\sqrt{\la}$. The plan is to show that (1) holds. 
We take the case when $s(p)\geq 4\sqrt{\la}$, the other case being symmetric. 
In this case we have $s(\theta)>3\sqrt{\la}$ for all $\theta\in L$,
which follows from the assumption on $s'(\theta)$ and the length estimate on $L$. Since $w(\theta)>0$ for $\theta>p$, we know that
$\psi(\theta)>s(\tht)\geq 4\sqrt{\la}$ for all $\theta>p$. The last inequality follows from the assumptions on $s$.
Thus, we need only to concentrate on the interval $[a,p)$. From the assumptions on $s$, and the estimates (\ref{B4_eq1}) on $w$ outside $L$, it follows 
that there exists an interval $I^1\subset [a,a']$ such that
$
\psi(I^1)=[-\sqrt{\la},\sqrt{\la}]
$  
and $\sqrt{\la}<\psi'(\theta)$ on $I^1$,
and $\psi(\tht)\notin[-\sqrt{\la},\sqrt{\la}]$ 
for all $\theta\in [a,a']\setminus I^1$. 

It remains to check what happens in the interval $[a',p)\subset L$. We know (from the above estimates) that $\psi([a',p))\supset (-\infty,2\sqrt{\la}]$. 
If $t\in [a',p)$ is such that $\psi(t)\in [-\sqrt{\la},\sqrt{\la}]$, then we must have $w(t)\in [-\co\la,-2\sqrt{\la}]$, since $3\sqrt{\la}<s(\theta)<\co \la$ on $L$.
An easy computation shows  (recalling that $w(t)=-h(t)/g(t)$) that we then have 
$$w'(t) < -\co \la^{1+7/8}\ll -\la.$$
From this it follows that there exists an interval $I^2$ such that $\psi(I^2)=[-\sqrt{\la},\sqrt{\la}]$, 
$\psi'(\theta)<-\sqrt{\la}$ on $I^2$ and $\psi(\theta)\notin [-\sqrt{\la},\sqrt{\la}]$ if $\theta\in [a',p)\setminus I^2$. 
We thus know that (1) holds.

\bigskip
Next we treat the situation when $|s(p)|<4\sqrt{\la}$. We shall prove that in this case statement (2) holds. We first notice (exactly as we did above) that 
we have
$|s(\theta)|<5\sqrt{\la}$ for all $\theta\in L$. We also recall that $\psi(\theta)\to-\infty$ as $\theta\to p^-$. 
Since $|w(\theta)|\leq \la^{-1/8}$ outside $L$, we have $|\psi(a')|<6\sqrt{\la}$ and (from the assumptions on $s$ ) $\psi(a)<-\la^{2/3}$.
We conclude that there must be an interval $I^1=[a_1,b_1]\subset [a,p)$ such that $\psi(a_1)=\psi(b_1)=-\la^{2/3}$ and $\psi(\theta)>-\la^{-2/3}$ for all 
$\theta\in(a_1,b_1)$. Moreover, there must be a turning point $q_1\in I^1$ such that $\psi'(q_1)=0$. By assumption we have 
$s'(\theta)>c_1\la^{7/8}$ for all
$\theta\in I$ such that $s(\theta)\in [-\la^{3/4},\la^{3/4}]$. Since $|s(p)|<4\sqrt{\la}$, we conclude that $s(\theta)\in [-\la^{3/4},\la^{3/4}]$
for all $\theta\in [a_1,p]$, and thus we have $s'(\theta)>\co\la^{7/8}$ for all $\theta\in [a_1,p]$. 
Moreover, we have $\psi'(a_1)>\co\la^{7/8}$ (since $w'(\theta)$ is small outside $L$). 

Assume that $t\in [a_1,p)$ is a point where $|\psi'(t)|\leq \sqrt{\la}$. Then, since $s'(t)>\co\la^{7/8}$, we must have $w'(t)<-\co\la^{7/8}$. 
For this to hold we must have
$0<g(t)<\co\sqrt{h(t)}\la^{1/16}$ (since otherwise $|w'(t)|<\co \la^{7/8}$, which an easy computation shows). 
Using this one sees that $w''(t)<-\co \la^{25/16}$. Hence we have
$$
\max\{|\psi'(\theta)|,-\psi''(\theta)\}> \sqrt{\la} \text{ for all } \theta\in I_1.
$$
This analysis also shows that $\psi(\theta)<-\la^{2/3}$ for all $\theta\in [a,p)\setminus I_1$.

Analogously one shows that there exists an interval $I^2=[a_2,b_2]\subset (p,b]$ such that 
$\psi(a_2)=\psi(b_2)=\la^{2/3}$, $\psi(\theta)>\la^{-2/3}$ for all $\theta\in (p,b]\setminus I^2$ and
$$
\max\{|\psi'(\theta)|,\psi''(\theta)\}> \sqrt{\la} \text{ for all } \theta\in I_2.
$$

Let $q_2\in I^2$ be the unique turning point (i.e., $\psi'(q_2)=0$). Then $$\min_{\theta\in I^2}\psi(\theta)-\max_{\theta\in I^1}\psi(\theta)=\psi(q_2)-\psi(q_1).$$
Since $w(q_1)<0$ and $w(q_2)>0$, we have 
$$\psi(q_2)-\psi(q_1)=s(q_2)+w(q_2)-(s(q_1)+w(q_1))>s(q_2)-s(q_1)>\co \la^{7/8}(q_2-q_1).$$ We must have $|g(q_i)|>\sqrt{\delta}\la^{-1/8}$, since otherwise
we would have $|w'(q_i)|>\co\la^{9/8}$ and thus we would not have $\psi'(q_i)=0$ ($i=1,2$). This implies that $q_2-q_1>\co \sqrt{\delta}\la^{-9/8}$. Consequently, 
the estimate in (2) holds.

\end{proof}

The last lemma in this appendix is an analogue of the previous one. It will be used in the inductive step.

\begin{lem}\label{L_B6}
Assume that the $C^2$-functions $g, h$ and $s$ are defined on an interval $I$ and that they satisfy the following: there are constants $D>d>\la$ such that
$\|g\|_{C^2(I)},\|s\|_{C^2(I)} <D$, 
$$
\begin{aligned}
s(I)&\supset [-\la^{-1},\la^{-1}],  \text{ and } ~s'(\theta)>d\ \text{ if } s(\theta)\in [-\la^{-1},\la^{-1}]; \\
g(I)&\supset [-\la^{-1},\la^{-1}], \text{ and } ~g'(\theta)<-d \text{ if } g(\theta)\in [-\la^{-1},\la^{-1}]; \text{and}
\end{aligned}
$$
$$
0<\delta\leq h(\theta)<\frac{1}{D^4}, ~ |h'(\theta)|\leq \sqrt{h(\theta)}, |h''(\theta)|<1.
$$
Let $\psi$ be defined as in (\ref{B_eq0}). Then there exists two (non-empty) intervals $I^1,I^2\subset I$ such that 
either we have
\begin{enumerate} 
\item $|\psi'(\theta)|>d/2$ on $I^1\cup I^2$, $\psi'$ having opposite signs on $I^1$ and $I^2$, 
$$
\psi(I^i)=[-1/(3\la),1/(3\la)] ~(i=1,2),
$$
and $|\psi(\theta)|>1/(3\la)$ if $\theta\notin I^1\cup I^2$ ; \ or

\item writing $I^i=[a_i,b_i]$ ($i=1,2$) we have $\psi(a_1)=\psi(b_1)=-1/(2\la)$ and $\psi(a_2)=\psi(b_2)=1/(2\la)$,
$$
\min_{\theta\in I^1}\max\{|\psi'(\theta)|,-\psi''(\theta)\}> d/2 ~\text{ and }~ \min_{\theta\in I^2}\max\{|\psi'(\theta)|,\psi''(\theta)\}>d/2;
$$
$$
\min_{\theta\in I^2}\psi(\theta)-\max_{\theta\in I^1}\psi(\theta)>d\frac{\sqrt{\delta}}{D^{3/2}};
$$
and $|\psi(\theta)|>1/(2\la)$ if $\theta\notin I^1\cup I^2$.
\end{enumerate}
\end{lem}
\begin{proof}The proof is similar to the previous one. We write $I=[a,b]$. Furthermore, we let $w=-h/g$, so that we can write $$\psi=s+w.$$
By using the assumptions on $g$ and $h$, it follows that the (tiny) set
$$
L=\{\theta\in I: |g(\theta)|\leq \sqrt{|g'(\theta)h(\theta)|}\}
$$
is an interval. We write $L=[a',b']$. From the estimates on $g'$ and $h$ we have $|g(\theta)|\leq \sqrt{|g'(\theta)h(\theta)|}<D^{-3/2}\ll \la^{-1}$ on $L$.
Thus $g'(\theta)<-d$ on $L$, and we can conclude that 
$$
|L|<1/(dD^{3/2}).
$$
We note that the bounds on $s$ therefore imply that
\begin{equation}\label{L_B6_eq1}
|s(\theta)-s(\theta')|<D\cdot 1/(dD^{3/2})=1/(d\sqrt{D})<\la^{-3/2}, \text{ for all } \theta,\theta'\in L.
\end{equation}
Next we estimate $g(\theta)$ outside $L$. Either $|g(\theta)|\geq 1/\la$, or $|g(\theta)|<1/\la$. In the latter case we have,
by assumption, that $g'(\theta)<-d$. Thus it follows that $|g(\theta)|>\sqrt{|g'(\theta)|h(\theta)}>\sqrt{d h(\theta)}$.
From this we conclude that
\begin{equation}\label{L_B6_eq2}
|w(\theta)|< 1/(D^2\sqrt{d})<\la^{-2} \text{ and } |w'(\theta)|<2 \text{ if } \theta \notin L.
\end{equation}
Moreover, an elementary calculation shows that 
\begin{equation}\label{L_B6_eq3}
|w''(\theta)|> \sqrt{d} D^2 \text{ for } \theta\in L
\end{equation}
(the dominant term in the estimate of $w''(\theta)$ is the term $g'(\theta)^2h(\theta)/g(\theta)^3$).

We let $p\in I$ be the unique point such that $g(p)=0$. Then $\lim_{\theta\to p+}w(\theta)=\infty$ and $\lim_{\theta\to p-}w(\theta)=-\infty$. As in the proof of the previous lemma, we
divide the analysis into two cases.

Case 1. Assume that $|s(p)|>0.4\la^{-1}$. We handle the case $s(p)>0.4\la^{-1}$. From (\ref{L_B6_eq1}) it follows that $s(\theta)>0.39\la^{-1}$ on $L$. 
Moreover, since $w$ and $w'$ are small
outside $L=[a',b']$ (by estimate (\ref{L_B6_eq2})), we have
$\psi([a,a'])\supset [-1/(3\la),1/(3\la)]$ and $\psi'(\theta)>d/2$ if $\psi(\theta)\in  [-1/(3\la),1/(3\la)]$. Furthermore, since $w(\theta)>0$ on $(p,b]$, we have 
$\psi(\theta)>1.3\la^{3/4}$ for all $\theta\in (p,b]$. We recall that $\lim_{\theta\to p-}\psi(\theta)=-\infty$. Now, if $\theta\in L$ is such that $\psi(\theta)\leq 1/(3\la)$, we must have
$w(\theta)<-0.05\la^{-1}$ (since $s(\theta)>0.39\la^{-1}$). This in turn implies that $w'(\theta)<-D^2$, and thus $\psi'(\theta)<D-D^2\ll -D^2/2$. These estimates show that 
case (1) in the statement of the lemma holds.

Case 2. Here we assume that $|s(p)|\leq 0.4\la^{-1}$. We shall first analyze $\psi$ on $L=[a',b']$. From (\ref{L_B6_eq1}) and the assumptions on $s$ it follows that
$s'(\theta)>d$ on $L=[a',b']$. The estimates on $w$ in (\ref{L_B6_eq2}-\ref{L_B6_eq3}) give us that $\psi'(a'),\psi'(b')>d-2$, and  
$$|\psi''(\theta)|>\sqrt{d}D^2-D>D^2 \text{ on } L.$$
Moreover, we recall that 
$\lim_{\theta\to p+}\psi(\theta)=\infty$ and $\lim_{\theta\to p-}\psi(\theta)=-\infty$. Thus there must be exactly two points $q_1<q_2$ in $L$ such that
$\psi'(q_i)=0$; one in $(a',p)$ and one in $(p,b')$. We estimate the size of the "gap" $\psi(q_2)-\psi(q_1)$. First we note that we have, exactly as in the proof
of the lemma above, 
$
\psi(q_2)-\psi(q_1)>s(q_2)-s(q_1)>d(q_2-q_1).
$
Since we must have $|g(q_i)|>\sqrt{\delta/D}$ (otherwise $|w'(q_i)|\gg D$, contradicting $\psi'(q_i)=0$), that is
$g(q_2)-g(q_1)>2\sqrt{\delta/D}$, we have $q_2-q_1>2\sqrt{\delta/D}/D$. Thus we have
$$
\psi(q_2)-\psi(q_1)>d\sqrt{\delta}/D^{3/2}.
$$
Finally, since (\ref{L_B6_eq2}) holds, we have, using the assumptions on $s$, that $\psi(a)<-0.9\la^{-1}$ and 
$\psi(b)>0.9\la^{-1}$, and
$\psi'(\theta)>d/2$ for all $\theta\in I\setminus L$ such that $|\psi(\theta)|\leq 0.9\la^{-1}$.  Taken together, this shows that case (2) holds.

\end{proof}

\bigskip
\noindent{\bf Acknowledgement:} I would like to thank H{\aa}kan Eliasson for many stimulating and fruitful discussions during the years. Several
of the ideas used in this paper are based on these discussions. 
I would also like to thank Russell Johnson for answering numerous questions, and for sharing with me his great knowledge. 
Financial support from the Swedish Research Council (VR) is gratefully acknowledged.


\end{document}